\documentclass[12pt,reqno]{amsart}
\usepackage{amsfonts}
\usepackage{amssymb}
\usepackage{amsmath}
\usepackage{latexsym}
\usepackage{amsthm}

 \theoremstyle{plain}
\newtheorem{theorem}{Theorem}[section]
\newtheorem{proposition}{Proposition}[section]
\newtheorem{corollary}{Corollary}[section]
\newtheorem{lemma}{Lemma}[section]
\newtheorem{remark}{\it Remark}[section]
\theoremstyle{definition}
\newtheorem{definition}{Definition}[section]

\date{Submitted: November 3, 2009. Revised: October 22, 2010. {\bf {To appear in:}} {\it Communications in Pure and Applied
Mathematics.}}
\begin{document}
\title[ Regularization mechanism for the
periodic KdV equation]{On the
regularization mechanism for the periodic Korteweg--de
Vries equation}
\author{}
\maketitle

\begin{center}
\textsc{Anatoli V. Babin\footnote{%
Department of Mathematics, University of California,
Irvine, California 92697, USA, E-mail:
ababine@math.uci.edu},
 Alexei A. Ilyin\footnote{%
\noindent Keldysh Institute of Applied Mathematics, Russian
Academy of Sciences, Miusskaya Sq. 4, 125047 Moscow,
Russia, E-mail: ilyin@keldysh.ru }
}and\textsc{\ Edriss S. Titi\footnote{%
Department of Mathematics and Department of Mechanical and
Aerospace Engineering, University of California, Irvine,
California 92697, USA, E-mail: etiti@math.uci.edu. Also:
Department of Computer Science and Applied Mathematics,
Weizmann Institute of Science, P.O. Box 26, Rehovot, 76100,
Israel, E-mail: edriss.titi@weizmann.ac.il}} \ \
\medskip

This paper is dedicated to the memory of Professor Basil Nicolaenko
\end{center}

\medskip

\begin{quote}
{\normalfont\fontsize{8}{10}\selectfont {\bfseries Abstract.} In
this paper we develop and use successive  averaging methods for
explaining the regularization mechanism in the the periodic
Korteweg--de Vries (KdV) equation in the homogeneous Sobolev spaces
$\dot{H}^s$, for $s\ge0$. Specifically, we prove the
 global existence, uniqueness, and
Lipschitz continuous dependence on the initial data of the solutions
of the periodic KdV. For the case  where the initial data is in
$L_2$ we also show the Lipschitz continuous dependence of these
solutions with respect to the initial data as maps from $\dot{H}^s$
to $\dot{H}^s$, for $s\in(-1,0]$.

\medskip }
\end{quote}



\setcounter{equation}{0}

\section{Introduction}

\label{S:Intro}

This paper is motivated by works on global regularity of solutions
of 3D problems in hydrodynamics (Navier-Stokes or Boussinesq system
with periodic boundary conditions) in the presence of fast rotation
or strong stratification, see, e.g., \cite{Awad}, \cite{BMN97},
\cite{BMN99}, \cite{BMN99a}, \cite{Embid}, \cite{G1}, \cite{gren},
\cite{Liu-Tadmor}, \cite{sch}, \cite{Ziane}, and references therein.
The results are based on the presence of high-frequency waves that
lead to destructive interference and effectively weaken the
nonlinearity through time averaging allowing to prove global
regularity of solutions for these problems, at the limit of infinite
rotation. The above mentioned hydrodynamical problems are rather
complicated (Boussinesq system, for example, involves five unknown
functions of three spatial variables and time, the equations are
coupled through a quadratic nonlinearity). Therefore, the purpose of
this paper is to apply this averaging approach to a slightly simpler
problem to make these ideas more transparent. Here we consider the
Korteweg--de Vries equation with one spatial variable subjected to
periodic boundary condition. The averaging effects are strong
because on high Fourier modes the linear term generates
high-frequency oscillations which weakens the nonlinearity and makes
it milder. In this paper we show connections between the smoothness
properties of solutions of the Korteweg--de Vries (KdV) equation and
the algebraic structure of the nonlinear resonances between the
high-frequency oscillations. Our main goal is to make the relations
between time-averaging effects and smoothness issues more explicit,
rather than to obtain global regularity results under minimal
restrictions (see, e.g., \cite{Bourgain97}, \cite{Colliander03},
\cite{K-P-V} and the references therein). In particular, our
approach and aim are completely different than the machinery and
harmonic analysis tools that were developed over the past decade and
half for investigating dispersive partial differential equations
(see the discussion and references mentioned below). Moreover,  we
also remark that the tools and ideas that are developed in this
contribution can be easily applied to multi-dimensional equations
and multi-component systems, a subject of future investigation.

It is worth mentioning that similar approach, taking advantage of
time averaging, has also been applied to the two-dimensional
rotating inviscid Burgers equations in \cite{Awad} and
\cite{Liu-Tadmor} to show that fast rotation prevents the formation
of a finite time singularity in the two-dimensional Burgers system.
We will use this warmup toy model, and in particular the approach in
\cite{Awad}, to fix ideas and to illustrate in
section~\ref{S:Burgers} the averaging mechanism, which is
instrumental in regularizing the solutions and in preventing  the
formation of a finite time singularity for the two-dimensional
inviscid Burgers system.

Our approach consists of three steps. First, we rewrite the problem
as a system of ODE for  time-dependent Fourier coefficients. This is
a standard approach to the periodic problem which is commonly used.
Second, to make  the effects of time averaging explicit  we single
out oscillating factors and do several integration by parts, with
respect to time, to obtain several generations of equations for
slowly varying coefficients. Resonances reveal themselves as
obstacles to the integrations by parts and produce resonant terms in
the equations, integrated terms become more and  more regular.
Higher generations of equations allow solutions with less
regularity. This algorithm is simple and straightforward, it does
not use fine methods such as analytic reduction to linear forms and
analysis of self similar solutions or the construction of conserved
or almost conserved quantities.     The third step is analysis of
the obtained equations. To show their regularity we use
straightforward estimates of multilinear operators, energy estimates
and the Contraction principle. To use the contraction principle in
negative Sobolev spaces we use splitting to high and low Fourier
modes and exploit averaging-induced squeezing of higher modes. In
order to justify our estimates we use a Galerkin approximations
procedure. We try to use methods which can be easily adjusted to
general multidimensional problems such as those mentioned above in
hydrodynamic problems. We do not use in our analysis the specific
properties of the KdV equation such as complete integrability or
special conserved quantities. Moreover, one can expect that in order
to obtain more delicate information one has to resort to the use of
higher generations of equations, i.e. more integrations by part.

The paper is organized as follows. In section~\ref{S:Trans} we
reformulate the KdV equation in terms of its Fourier components and
introduce the notion of weak solution; we also make two types of
formal transformations of the equation which will be used for
proving the relevant well-posedness results. In section~\ref{S:Uniq}
we use the Galerkin procedure for proving the global existence
 in the Sobolev spaces $\dot{H}^s$ for
$s>0$.  In section~\ref{S:Uniq1/2} we use the Contraction Principle
to show that the solutions established in section~\ref{S:Uniq} are
unique and Lipschitz continuous, with respect to the initial data,
in the Sobolev spaces $\dot{H}^s$ for $s\ge1/2$. A more technical
section~\ref{Sec:nregin} is devoted to proving similar results for
$s\in[0,1/2]$, where the case $s=0$ is dealt with separately.
Finally, for solutions with initial data in $L_2$ we show the
Lipschitz continuous dependence of these solutions, with respect to
the initial data, as maps from $\dot{H}^s$ to $\dot{H}^s$ for all
$s\in(-1,0]$.

There is a long history of the KdV equations in the periodical
setting. We do not attempt to give here a complete survey of this
long history, but we are  satisfied by mentioning, below, some of
the key landmark results. First Bourgain \cite{Bourgain93} proved
local well-posedness in $L_2$. We observe that the result of
Bourgain also implies global existence for the real valued KdV in
$L_2$. Later the local well-posedness in the Sobolev space $H^s$,
for $s>-1/2$,  was established for the real and complex valued KdV
by Kenig, Ponce, and Vega \cite{K-P-V}. Later global well-posedness
in the Sobolev space $H^s$, for $s>-1/2$,  was established for the
real valued KdV by Colliander, Keel, Staffilani, Takaoka and Tao
\cite{Colliander03}. Using the fact that the KdV equation is an
integrable system, and by implementing the inverse method, Kappeler
and Topalov \cite{Kappeler-Topalov} proved  global well-posedness in
the Sobolev space $H^s$, for $s>-1$.

Even though the results presented in this contribution are weaker
than the state-of-the-art results mentioned above, our approach is
substantially different and easy to master. Specifically, and as we
have stressed above, our approach is based on systematic application
of time averaging. It is worth mentioning that since we rewrite the
KdV equation in different forms, based on the differentiation by
parts with respect to the time variable,  there is superficial
similarity with the reduction to normal forms (cf. \cite{Shatah}).  The difference is
that the consecutive forms of the KdV, which we derive, are based
not on the geometric properties such as a reduction to a linear form
or a reduction to invariant tori, but exclusively on time averaging
properties. In order to stress the effects, which are due solely to
the time averaging, we do not use energy conservation or modified
energy functionals as in the I-method. We use only the $L_2$  norm
conservation to mimic the situation which occurs in the
three-dimensional Euler equations and other hydrodynamical systems.

Moreover, we believe that the averaging approach presented here can
be pushed further to achieve the best known results concerning the
KdV equation. Furthermore, different nonlinear dispersive and wave
equations, such as the m-KdV, the Klein-Gordon equation and
water wave equations
(see for example, \cite{GMS1}, \cite{Shatah},
\cite{Wu1}-\cite{Wu3},
and references therein),
can possibly be studied quite similarly, a subject of future
research.

\setcounter{equation}{0}

\section{The Complex Burgers equation with fast rotation --- a paradigm for
the averaging mechanism} \label{S:Burgers}

In this section we consider the complex Burgers equation with fast
rotation as a warmup toy model to demonstrate the averaging
mechanism for preventing the formation of singularity.

The one-dimensional Burgers equation
\[
u_t+u u_x=0
\]
is known to develop singularity in finite time. Considering the
complex version of this equation:
%%%
\begin{equation}\label{Bur}
u_t+u u_z=0, \qquad u(z,0)= \varphi(z),
\end{equation}
%%%
where $u\in \mathbb{C}$, and the initial value, $\varphi(z)$, is a
bounded complex analytic function with a bounded derivative in the
complex strip $D=\{z\in\mathbb{C}: ~ |z| < d\}$. Using the usual
characteristic method one can show that the solution of~(\ref{Bur})
is given by the implicit relation:
$$
u(z,t)=\varphi(z-t u(z,t)).
$$
It is easy to show that the above relation defines a unique complex
analytic solution $u(z,t)$ in a smaller sub-strip of $D$, for small
values of $|t|$, depending on the $\sup_{z\in D} |\varphi(z)|$ and
$\sup_{z\in D} |\varphi'(z)|$. Moreover, the solution of~(\ref{Bur})
satisfies:
%%%
\begin{equation}\label{D-Bur}
\partial_z u = \varphi'(z-t u(z,t))\left(1+ t \varphi'(z-t
u(z,t))\right)^{-1}.
\end{equation}
%%%
Equation (\ref{D-Bur}) shows that for a large class of
analytic initial data, $\varphi$, the solution
of~(\ref{Bur}) develops a singularity in a finite time. In
particular, and as for the case of the real Burgers
equation, the solution develops a finite time singularity
if the initial data, $\varphi$, maps $\mathbb{R}$ into
$\mathbb{R}$ and is a monotonic decreasing function in some
interval of $\mathbb{R}$. Moreover, it also follows
from~(\ref{D-Bur}) and the Picard little theorem for entire
functions that if the initial data $\varphi$ is an entire
function then the solution of~(\ref{Bur}) instantaneously
ceases to be an entire function, unless
$\varphi'(z)=constant$, i.e., $\varphi(z)= az+b$ for some
constants $a, b \in \mathbb{C}$.

Adding a complex rotation term to~(\ref{Bur}) one obtains the
rotating complex Burgers equation
%%%
\begin{equation}\label{Bur-R}
u_t+u u_z=i \Omega u, \qquad u(z,0)= \varphi(z),
\end{equation}
%%%
in the strip $D$, with rotation rate $\Omega \in \mathbb{R}$. Making
the change of variables $v=e^{-i \Omega t} u$ gives the equivalent
equation
%%%
\begin{equation}\label{Bur-R-v}
v_t+ e^{i \Omega t} v v_z=0, \qquad v(z,0)= \varphi(z),
\end{equation}
%%%
which can also be solved by the characteristic method to obtain the
solution in the implicit form
%%%
\begin{equation}\label{Sol-v}
v(z,t)=\varphi(z- \frac{e^{i \Omega t} - 1}{i \Omega}v(z,t)).
\end{equation}
%%%
Equation (\ref{Sol-v}) has a unique complex analytic solution, with
respect to the spatial complex variable $z$, for small values of
$|t|$ and in a fixed smaller sub-strip of $D$, whose size depends on
$|\Omega|$, $\sup_{z\in D} |\varphi(z)|$ and $\sup_{z\in D}
|\varphi'(z)|$. Observe that
%%%
\begin{equation}\label{D-Bur-R}
\partial_z v(z,t) = \varphi'(z- \frac{e^{i \Omega t} - 1}{i
\Omega}v(z,t)) \left(1+\frac{e^{i \Omega t} - 1}{i
\Omega}\varphi'(z- \frac{e^{i \Omega t} - 1}{i
\Omega}v(z,t))\right)^{-1}.
\end{equation}
%%%
As a result, it is clear that  if we choose $|\Omega|$ large enough,
such that
$$
|\Omega| \ge 2 \sup_{z\in D} |\varphi'(z)|,
$$
%%%
then $|\partial_z v(z,t)|$ remains finite for all $t\in \mathbb{R}$,
and the solution  exists globally in time. Consequently, we have
just demonstrated that for fast rotation, which depends on the
initial data, the unique solution remains regular in a fixed
sub-strip of $D$, for all $t\in \mathbb{R}$. It is worth observing,
however, that, again, by virtue of~(\ref{D-Bur-R}) and the Picard
little theorem the solution of~(\ref{Bur-R}) instantaneously ceases
to be an entire function even if the initial data  $\varphi$ is an
entire function, unless $\varphi(z)=a z +b$, for some constants $a,
b \in \mathbb{C}$.

Observe that  equation~(\ref{Bur-R}) can be rewritten in the
integrated form:
%%%
$$
v(z,t) = \varphi(z) - \int_0^t e^{i \Omega \tau} v(z,\tau)
\partial_z v(z,\tau) d\tau.
$$
%%%%
This form exhibits the role that is played by the time averaging
process in prolonging the life-span of the solutions for this kind
of equations.  Indeed, the averaging against the fast oscillating
term $e^{i \Omega t}$, for $|\Omega|$  large enough, reduces the
``strength" of  the nonlinear term and makes it milder, which is the
underlined mechanism for  prolonging the life-span of the solutions.

Setting $u=u_1+ i u_2$ and $z= x+iy$, we observe that, thanks to the
Cauchy-Riemann equations, equation~(\ref{Bur-R}) is equivalent the
two-dimensional Burgers equations with rotation, for the vector
field $(u_1(x,y), u_2(x,y))$. This  is the same  system  that  was
investigated in \cite{Liu-Tadmor}; and similar results, to the ones
mentioned above, were proved using completely different tools.

\setcounter{equation}{0}

\section{Transformations of the Korteweg--de
Vries equation} \label{S:Trans}

The toy model presented in section \ref{S:Burgers} is only an
illustrative example to demonstrate the underlying main idea of the
averaging mechanism. Here, we consider the Korteweg--de Vries (KdV)
equation with space-periodic boundary condition
\begin{equation}
\partial _{t}u=u\partial _{x}u+\partial _{x}^{3}u,\qquad u(0,x)=u^{0}(x),
\label{KdV}
\end{equation}
where $x\in \mathbb{S}^{1}=[0,2\pi ]$ and $u(t,0)=u(t,2\pi )$. Let
$u(x,t)$ be a smooth solution of equation~(\ref{KdV}).
 Writing the
nonlinear term as $u\partial _{x}u=\frac{1}{2}\partial _{x}(u^{2})$
and integrating with respect to $x$ we see that $\int_{0}^{2\pi
}u(t,x)dx=\int_{0}^{2\pi }u(0,x)$. Therefore, without loss of
generality we can (and shall) assume that the initial data and the
solution both have zero spatial mean
\begin{equation}  \label{zpm}
\int_{0}^{2\pi }u(t,x)dx=\int_{0}^{2\pi }u^{0}(x) dx=0.
\end{equation}
Multiplying (\ref{KdV}) by $u$ and integrating with respect to $x$
we formally obtain the well-known $L_{2}$-norm conservation
property:
\begin{equation*}
\|u(t)\|_{L_{2}}=\|u^{0}\|_{L_{2}}.
\end{equation*}

We use the Fourier series in $x$:
\begin{equation}
u(t,x)=\sum_{k\in \mathbb{Z}_0}u_{k}(t)e^{ikx},\quad
u_{k}\in \mathbb{C},\quad u_{k}(t)=\frac{1}{2\pi }
\int_0^{2\pi }u(t,x)e^{-ikx}dx, \label{Four}
\end{equation}
where, according to (\ref{zpm}) $u_{0}(t) \equiv 0$, hence the
summation is over $\mathbb{Z}_{0}=\mathbb{Z}\setminus \{0\}$, and we
have, in addition, that $u_{-k}=\bar{u}_{k}$, since $u$ is
real-valued.

It will be convenient below to normalize the $L_{2}$ norm
so that
\begin{equation*}
\|u\|^{2}_{\dot{H}^0}=\|u\|_{L_{2}}^{2}=
\frac{1}{2\pi}\int_{0}^{2\pi }u(x)^{2}dx=\sum_{k\in
\mathbb{Z}_{0}}|u_{k}|^{2}.
\end{equation*}
Accordingly, the norms in the Sobolev spaces $H^{s}$ are
defined as follows. We set
\begin{equation}
\Vert u\Vert _{H^{s}}^{2}=|u_{0}|^{2}+\Vert u\Vert
_{\dot{H}^{s}}^{2}, \label{Sobolev}
\end{equation}%
where $\Vert u\Vert _{\dot{H}^{s}}$ is the norm in the homogeneous space $%
\dot{H}^{s}$ of functions with mean value zero
\begin{equation*}
\Vert u\Vert _{\dot{H}^{s}}^{2}=\Vert \{|k|^{s}u_{k}\}\Vert
_{l_{2}}^{2}=\sum_{k\in \mathbb{Z}_{0}}|k|^{2s}|u_{k}|^{2},\quad s\in \mathbb{R%
}.
\end{equation*}
We will use the above definition for a norm of a function
$u(x)$ as well as for a norm of a sequence
$\left\{u_{k}\right\}$.

Using the Fourier representation of $u$ (a smooth solution
of~(\ref{KdV}))
 we write  equation~(\ref{KdV}) as
the infinite coupled system of ordinary differential
equations for the coefficients~$u_{k}(t)$:
\begin{equation}
\partial _{t}u_{k}=\frac{1}{2}ik%
\sum_{k_{1}+k_{2}=k}u_{k_{1}}u_{k_{2}}-ik^{3}u_{k},\quad
u_{k}(0)=u_{k}^{0},\quad k\in \mathbb{Z}_{0}.  \label{KdVu}
\end{equation}%
Next, we use the following transformation of variables
\begin{equation}
u_{k}(t)=e^{-ik^{3}t}v_{k}(t),\qquad k\in \mathbb{Z}_{0}.
\label{utov}
\end{equation}
Since $v_{-k}(t)\equiv \bar{v}_{k}(t)$ for all $k\in \mathbb{Z}_{0}$
then the functions $v(t,x)$ is real-valued along with $u(t,x)$,  and
for every $s$ the Sobolev norm is preserved:
\begin{equation*}
\Vert u(t)\Vert _{\dot{H}^{s}}=\Vert v(t)\Vert
_{\dot{H}^{s}}.
\end{equation*}
We observe that this  change of variables is similar to the one
introduced for~(\ref{Bur-R}), and has also been used in
\cite{BMN97}, \cite{BMN99} for proving global regularity of the 3D
rotating Navier--Stokes equations, see also \cite{BMN99a},
\cite{MoiseZiane} and \cite{Ziane}. In the first place, the
substitution~(\ref{utov}) eliminates the linear term $ik^{3}$ \
in~(\ref{KdVu}), which has the highest, namely cubic, order of
growth as $|k|\to\infty $ and, most importantly, introduces
oscillating exponentials into the nonlinear term.
Substituting~(\ref{utov}) in~(\ref{KdVu}), multiplying by
$e^{ik^{3}t}$, and using the identity
\begin{equation} \label{k1k2}
(k_{1}+k_{2})^{3}-k_{1}^{3}-k_{2}^{3}=
3(k_{1}+k_{2})k_{1}k_{2},
\end{equation}
we obtain the equivalent  coupled system of equations
\begin{equation}\label{v}
\partial _{t}v_{k}=\frac{1}{2}ik%
\sum_{k_{1}+k_{2}=k}e^{i3kk_{1}k_{2}t}v_{k_{1}}v_{k_{2}},\quad
v_{k}(0)=v_{k}^{0}=u_{k}^{0},\quad k\in \mathbb{Z}_{0}.
\end{equation}
Since our
techniques are based on Fourier expansions, we take
equation (\ref{v}) as our primary form of the KdV,
which is equivalent to the classical KdV
equation~(\ref{KdV}) for smooth solutions.
\begin{definition}\label{D:defmain}
A function $v$ is called a solution of~(\ref{v})
over the time interval $[0,T]$
if  $v\in L_\infty([0,T];\dot{H}^0)$ and the integrated
equation~(\ref{v})
\begin{equation} \label{vint}
v_{k}\left( t\right) -v_{k}\left( 0\right) =\frac{1}{2}ik
\int_{0}^{t}\sum_{k_{1}+k_{2}=k}
e^{i3kk_{1}k_{2}t^{\prime}}v_{k_{1}}(t')v_{k_{2}}(t')
dt^{\prime }, \quad k\in \mathbb{Z}_{0}
\end{equation}
is satisfied for every $k\in\mathbb{Z}_{0}$.
Accordingly, $u(x,t)$ with $u_k(t)=e^{-ik^3t}v_k(t)$
will then be called a weak solution of the KdV.
\end{definition}

\begin{remark}\label{R:regularity}
\rm{
We observe that if $v$ is a solution in the sense
of Definition~\ref{D:defmain}, then by setting
$d_k(t)=\sum_{k_{1}+k_{2}=k}
e^{i3kk_{1}k_{2}t}v_{k_{1}}(t)v_{k_{2}}(t)$
we have
\begin{equation}\label{d}
\sup_{t\in[0,T]}|d_k(t)|
\le
\|v\|_{L_\infty([0,T];\dot{H}^0)}^2,
\end{equation}
and hence we automatically get from~(\ref{vint}) that $v_k(t)$ is an
absolutely continuous function over $[0,T]$  for each $k$, and
consequently (\ref{v}) is satisfied a.\,e.\ in $[0,T]$. Moreover, by
Lemma~\ref{L:B1}  we have $B_1(v,v)\in
L_\infty([0,T];\dot{H}^{-\theta})$, for $\theta>3/2$, where
\begin{equation}\label{B1}
B_1(\varphi,\psi)_k=
\frac{1}{2}ik
\sum_{k_{1}+k_{2}=k}e^{i3kk_{1}k_{2}t}\varphi_{k_{1}}
\psi_{k_{2}},
\quad k\in \mathbb{Z}_{0}.
\end{equation}
Therefore, satisfying (\ref{v}) in the sense of
Definition~\ref{D:defmain} implies satisfying the functional
equation
\begin{equation}\label{vneg}
    \partial_t v=B_1(v,v),\quad v(0)=v^0,
\end{equation}
 in the sense of
$L_p([0,T];\dot{H}^{-\theta})$, where $\theta>3/2$ and
$1\le p\le\infty$.
}
\end{remark}

Assume we have a smooth enough solution of the KdV
equation~(\ref{KdV}), and hence of equation~(\ref{v}). We will do
all kinds of manipulations to get various forms of the equations.
This will be done by differentiation by parts procedure and the
repeated use of equation~(\ref{v}). Notice that formally the
solutions of~(\ref{v}) will satisfy the newly derived equations. We
will derive these equations formally in order to motivate the use of
the corresponding  Galerkin truncated versions of these equations,
which will play a major role in the rigorous justification of the
steps of our proofs.

\subsubsection*{\textbf{{\rm 3.1.} First differentiation by parts in time.}}

Equation (\ref{v}) \ does not involve $ik^{3}$ as\ in~(\ref{KdVu}),
but still involves explicitly the factor $\frac{1}{2}ik$, which
tends to infinity as $|k|\to\infty $.  To obtain a system for the
Fourier coefficients that has coefficients uniformly bounded in $k$
we rewrite (\ref{v}) in a different form. The formal transformation,
which is valid for smooth enough solutions,  corresponds to the
integration by parts in~(\ref{vint}), but since we rewrite the
differential equation (\ref{v}),
 we call it differentiation by parts. Since $v_{0}(t)\equiv0$,
we can assume in~(\ref{v}) that $k_{1},k_{2}\neq 0$ so that
$\partial _{t}((i3kk_{1}k_{2})^{-1}e^{i3kk_{1}k_{2}t})=
e^{i3kk_{1}k_{2}t}$, and therefore
\begin{equation}\label{intby1}
\aligned
\partial _{t}v_{k}=\partial _{t}\left( \frac{1}{2}%
ik\sum_{k_{1}+k_{2}=k}\frac{e^{i3kk_{1}k_{2}t}v_{k_{1}}v_{k_{2}}}{%
i3kk_{1}k_{2}}\right) -\frac{1}{2}ik\sum_{k_{1}+k_{2}=k}\frac{%
e^{i3kk_{1}k_{2}t}}{i3kk_{1}k_{2}}\partial _{t}(v_{k_{1}}v_{k_{2}})=
 \\
\frac{1}{6}\partial _{t}\left( \sum_{k_{1}+k_{2}=k}\frac{%
e^{i3kk_{1}k_{2}t}v_{k_{1}}v_{k_{2}}}{k_{1}k_{2}}\right) -\frac{1}{6}%
\sum_{k_{1}+k_{2}=k}\frac{e^{i3kk_{1}k_{2}t}}{k_{1}k_{2}}(v_{k_{2}}\partial
_{t}v_{k_{1}}+v_{k_{1}}\partial _{t}v_{k_{2}}).
\endaligned
\end{equation}
The last two terms are symmetric with respect to $k_{1}$ and $k_{2}$, and it
suffices to con\-si\-der one of them. For the term containing, say, $%
v_{k_{2}}\partial _{t}v_{k_{1}}$ we use~(\ref{v}) for $\partial
_{t}v_{k_{1}} $ and obtain
\begin{gather*}
\aligned\sum_{k_{1}+k_{2}=k}\frac{e^{i3kk_{1}k_{2}t}}{k_{1}k_{2}}%
v_{k_{2}}\partial _{t}v_{k_{1}}=\frac{1}{2}i\sum_{k_{1}+k_{2}=k}\frac{%
e^{i3kk_{1}k_{2}t}}{k_{2}}v_{k_{2}}\sum_{\alpha +\beta
=k_{1}}e^{i3k_{1}\alpha \beta t}v_{\alpha }v_{\beta }= \\
\frac{1}{2}i\sum_{\alpha +\beta +k_{2}=k}\frac{e^{i3(k(\alpha +\beta
)k_{2}+(\alpha +\beta )\alpha \beta )t}}{k_{2}}v_{k_{2}}v_{\alpha }v_{\beta
} =\frac{1}{2}i\sum_{k_{1}+k_{2}+k_{3}=k}\frac{%
e^{i3(k_{1}+k_{2})(k_{2}+k_{3})(k_{3}+k_{1})t}}{k_{2}}%
v_{k_{1}}v_{k_{2}}v_{k_{3}},\endaligned
\end{gather*}%
where we have factorized the exponent as follows
\begin{equation*}
kk_{2}+\alpha \beta =(\alpha +\beta +k_{2})k_{2}+\alpha \beta =(k_{2}+\alpha
)(k_{2}+\beta ).
\end{equation*}

\begin{remark}\label{R:3-identity}
\rm{
This is, of course, just the identity
\begin{equation}\label{k1k2k3}
(k_1+k_2+k_3)^3-k_1^3-k_2^3-k_3^3=
3(k_1+k_2)(k_2+k_3)(k_3+k_1).
\end{equation}
We also observe that the cubic identities \eqref{k1k2} and
\eqref{k1k2k3} are known to be very important in the
analysis of the KdV equation both on
$\mathbb{R}$
and in the periodic setting \cite{Bourgain97},
\cite{Colliander03}.
}
\end{remark}

Since the expression for the term containing
$v_{k_{1}}\partial_{t}v_{k_{2}}$ is exactly the same, we finally
obtain the following   form of the KdV
\begin{equation}\label{one}
\partial _{t}\left( v_{k}-\frac{1}{6}B_{2}(v,v)_{k}\right) =
\frac{i}{6}R_{3}(v,v,v)_{k},\qquad k\in \mathbb{Z}_{0},
\end{equation}%
where
\begin{equation}
B_{2}(u,v)_{k}=B_{2}(u,v,t)_{k}=\sum_{k_{1}+k_{2}=k}
\frac{e^{i3kk_{1}k_{2}t}u_{k_{1}}v_{k_{2}}}{k_{1}k_{2}},
\qquad k\in \mathbb{Z}_{0},
\label{B2}
\end{equation}%
and
\begin{equation}
R_{3}(u,v,w)_{k}=\sum_{k_{1}+k_{2}+k_{3}=k}\frac{%
e^{i3(k_{1}+k_{2})(k_{2}+k_{3})(k_{3}+k_{1})t}}{k_{1}}%
u_{k_{1}}v_{k_{2}}w_{k_{3}},
\qquad k\in \mathbb{Z}_{0}.
\label{R3}
\end{equation}%
Recall that the transformation between the variables $u $ and $v$ is
an isometry between the Sobolev spaces $\dot{H}^s$. Moreover, we
will show later that all the terms involved in the above form of the
KdV are bounded operators, in the appropriate functional spaces.
Consequently, the above form of the KdV is a much milder form than
the original forms in equations~(\ref{KdV}) and~(\ref{v}). Since we
will introduce later another form of this system using one more
differentiation by parts, we also call (\ref{one}) {\it the first
 form of the KdV}. We integrate~(\ref{one}) with respect to $t$ and
arrive at the following {\it integrated first   form of the  KdV}:
\begin{equation} \label{defofsol}
v_k(t)-\frac{1}{6}B_{2}(v(t),v(t))_k=v_k(0)-\frac{1}{6}
 B_{2}(v^0,v^0)_k+\frac{i}{6}\int_{0}^{t}
 R_{3}(( v(\tau ))^{3})_k d\tau,\quad v_k(0)=v^{0}_k,
\end{equation}%
where $R_{3}\left( v^{3}\right)=R_{3}(v,v,v)$ is given
by~(\ref{R3}) and $B_{2}(v,v)$ is given by~(\ref{B2}).
\begin{remark}\label{R:one-oper}
\rm{ Let $v$ be a solution of~(\ref{v}) in the sense of
Definition~\ref{D:defmain}. Then we observe that in view of
Lemma~\ref{L:minR3} (with $\beta=0$) the trilinear operator $R_3$ is
a bounded map from $(\dot{H}^0)^3$ into $\dot{H}^{-\theta}$, for
$\theta>1/2$, so that   equation~(\ref{one}) in the functional form
holds in $L_\infty([0,T];\dot{H}^{-\theta})$. }
\end{remark}

Here we will go from~(\ref{one}) along
two different paths depending on the purpose:
\begin{itemize}
    \item \textit{a priori} estimates
    (second differentiation by parts in time);
    \item uniqueness and Lipschitz continuous dependence of the
    solutions on the initial data.
\end{itemize}

\subsubsection*{\textbf{{\rm 3.2.}~Second differentiation
 by parts in time }}

We prove later in Lemma~\ref{L:R3} that the trilinear operator
$R_{3}(u,v,w)$ is a bounded map from the Sobolev space $(
\dot{H}^{s})^{3}$ into $\dot{H}^{s}$ for any $s>1/2$. The bilinear
operator $B_2(u,v)$ has nicer continuity properties and is bounded
from $( \dot{H}^{s})^{2}$ into $\dot{H}^{s}$ for any $s>-1/2$, see
Lemma~\ref{L:B2}. Therefore, from  the continuity properties of
$R_3$ we are unable to use~(\ref{one}) to establish the required
\textit{a priori} estimates for $s\ge0$. For this reason we  use
again the idea of differentiation by parts, and once again represent
the exponential in~(\ref{R3}) as the time derivative. But before
doing that we have to take care of the resonances which are the
obstruction to the integration by parts procedure.

\subsubsection*{\textit{Resonances.}}

We single out the terms in~(\ref{R3}) for which
\begin{equation}  \label{rescond}
(k_1+k_2)(k_2+k_3)(k_3+k_1)=0, \qquad k_1+k_2+k_3=k\in
\mathbb{Z}_0.
\end{equation}
Accordingly, we have
\begin{equation}\label{R3res}
\aligned
&R_{3}(v,v,v)_{k}=R_{3\text{res}}(v^{3})_{k}+
R_{3\text{nres}}(v^{3})_{k},
 \\
&R_{3\text{res}}(v^{3})_{k}=\sum_{k_{1}+k_{2}+k_{3}=k}^{\text{res}}\frac{%
v_{k_{1}}v_{k_{2}}v_{k_{3}}}{k_{1}},
\\
&R_{3\text{nres}}(v^{3})_{k}=%
\sum_{k_{1}+k_{2}+k_{3}=k}^{\text{nonres}}
\frac{e^{i3(k_{1}+k_{2})(k_{2}+k_{3})(k_{3}+k_{1})t}}{k_{1}}%
v_{k_{1}}v_{k_{2}}v_{k_{3}},
\endaligned
\end{equation}%
where the first summation is carried out over the set
subscripts $k_{1},k_{2},k_{3}$ satisfying~(\ref{rescond})
(the resonance), while in the second summation
$(k_{1}+k_{2})(k_{2}+k_{3})(k_{3}+k_{1})\neq 0$ (the
non-resonant terms).

Since it easy to see that not all three factors
in~(\ref{rescond}) can be
zero at the same time, it follows that the set of
$k_{1},k_{2},k_{3}$
satisfying~(\ref{rescond}) is the union of six disjoint sets
$S_{1},\dots,S_{6}$:
\begin{gather*}
\aligned
&S_{1}=\{k_{1}+k_{2}=0\}\cap \{k_{2}+k_{3}=0\}\
\Leftrightarrow
k_{1}=k,k_{2}=-k,k_{3}=k, \\
&S_{2}=\{k_{1}+k_{2}=0\}\cap \{k_{3}+k_{1}=0\}\
\Leftrightarrow
k_{1}=-k,k_{2}=k,k_{3}=k, \\
&S_{3}=\{k_{2}+k_{3}=0\}\cap \{k_{3}+k_{1}=0\}\
\Leftrightarrow
k_{1}=k,k_{2}=k,k_{3}=-k, \\
&S_{4}=\{k_{1}+k_{2}=0\}\cap \{k_{2}+k_{3}\neq 0\}\cap
\{k_{3}+k_{1}\neq 0\}\
\Leftrightarrow k_{1}=j,k_{2}=-j,k_{3}=k,|j|\neq k, \\
&S_{5}=\{k_{2}+k_{3}=0\}\cap \{k_{1}+k_{2}\neq 0\}\cap
\{k_{3}+k_{1}\neq 0\}\
\Leftrightarrow k_{1}=k,k_{2}=j,k_{3}=-j,|j|\neq k, \\
&S_{6}=\{k_{3}+k_{1}=0\}\cap \{k_{1}+k_{2}\neq 0\}\cap
\{k_{2}+k_{3}\neq 0\}\
\Leftrightarrow k_{1}=j,k_{2}=k,k_{3}=-j,|j|\neq k,
\endaligned
\end{gather*}%
where $j\in \mathbb{Z}_{0}$. Therefore,
\begin{gather}
\aligned
R_{3\text{res}}(v^{3})_{k}=\sum_{m=1}^{6}\sum_{S_{m}}
\frac{v_{k_{1}}v_{k_{2}}v_{k_{3}}}{k_{1}}=
\frac{v_{k}v_{-k}v_{k}}{k}+\frac{%
v_{-k}v_{k}v_{k}}{-k}+\frac{v_{k}v_{k}v_{-k}}{k}+
\label{Ares} \\
v_{k}\sum_{j\in \mathbb{Z}_{0},|j|\neq k}
\frac{v_{j}v_{-j}}{j}+\frac{v_{k}}{k}
\sum_{j\in \mathbb{Z}_{0},|j|\neq k}|v_{j}|^{2}+v_{k}
\sum_{j\in \mathbb{Z}_{0},|j|\neq k}\frac{v_{j}v_{-j}}{j}= \\
\frac{v_{k}}{k}\left( |v_{k}|^{2}+\sum_{j\in
\mathbb{Z}_{0}}|v_{j}|^{2}-|v_{k}|^{2}-|v_{-k}|^{2}\right) =
\frac{v_{k}}{k}(\|v\| _{L_{2}}^{2}-|v_{k}|^{2})=:
%\\
%\frac{v_{k}}{k}(\| u^{0}\| _{L_{2}}^{2}-|v_{k}|^{2})=:
A_{\mathrm{res}}(v)_{k},
\endaligned
\end{gather}%
where we used the fact that the first two terms add up to zero, the fourth
and the sixth terms are both zero by symmetry
$j\to -j$.
\begin{remark}\label{R:conservation}
{\rm
We observe that if the energy is conserved,
$\|v(t)\|_{L_2}=\|v(0)\|_{L_2}=\|v^0\|_{L_2}=\|u^0\|_{L_2}$,
then clearly
\begin{equation}\label{R3resconserv}
R_{3\text{res}}(v^{3})_{k}=
\frac{v_{k}}{k}(\|v\| _{L_{2}}^{2}-|v_{k}|^{2})=
\frac{v_{k}}{k}(\| v^{0}\| _{L_{2}}^{2}-|v_{k}|^{2})=:
A_{\mathrm{res}}(v)_{k}.
\end{equation}
}
\end{remark}

For smooth enough solutions equation~(\ref{one}) can now be written
in the form
\begin{equation}
\aligned
\partial _{t}\left( v_{k}-\frac{1}{6}B_{2}(v,v)_{k}\right) =
\frac{i}{6}A_{\mathrm{res}}(v)_{k}+\frac{i}{6}
R_{3\text{nres}}(v^{3}).
\endaligned
\label{oneres}
\end{equation}
Since the exponent in the last term on the right-hand side does not
vanish, we can differentiate by parts with respect to $t$ a second
time:
\begin{equation}  \label{second}
\aligned
R_{3\text{nres}}(v^{3})_{k}=
\sum_{k_{1}+k_{2}+k_{3}=k}^{\text{nonres}}\frac{%
e^{i3(k_{1}+k_{2})(k_{2}+k_{3})(k_{3}+k_{1})t}}{k_{1}}%
v_{k_{1}}v_{k_{2}}v_{k_{3}}=\frac{1}{3i}\partial
_{t}B_{3}(v,v,v)_{k}-\\
\frac{1}{3i}\sum_{k_{1}+k_{2}+k_{3}=k}^{\text{nonres}}\frac{%
e^{i3(k_{1}+k_{2})(k_{2}+k_{3})(k_{3}+k_{1})t}}{%
k_{1}(k_{1}+k_{2})(k_{2}+k_{3})(k_{3}+k_{1})}(\partial
_{t}v_{k_{1}}v_{k_{2}}v_{k_{3}}+v_{k_{1}}\partial
_{t}v_{k_{2}}v_{k_{3}}+v_{k_{1}}v_{k_{2}}
\partial _{t}v_{k_{3}}),
\endaligned
\end{equation}
where
\begin{equation}
B_{3}(u,v,w)_{k}=\sum_{k_{1}+k_{2}+k_{3}=k}^{\text{nonres}}\frac{%
e^{i3(k_{1}+k_{2})(k_{2}+k_{3})(k_{3}+k_{1})t}}{%
k_{1}(k_{1}+k_{2})(k_{2}+k_{3})(k_{3}+k_{1})}u_{k_{1}}v_{k_{2}}w_{k_{3}}.
\label{B3}
\end{equation}

As before, we express the time derivatives in the last term on the
right-hand side in~(\ref{second}) by means of equation~(\ref{v}).
The terms containing $\partial _{t}v_{k_{2}}$ and $\partial
_{t}v_{k_{3}}$ produce the same two expressions and after a straight
forward calculation we obtain
\begin{gather}\label{B4}
\aligned\sum_{k_{1}+k_{2}+k_{3}=k}^{\text{nonres}}
\frac{e^{i3(k_{1}+k_{2})(k_{2}+k_{3})(k_{3}+k_{1})t}}
{k_{1}(k_{1}+k_{2})(k_{2}+k_{3})(k_{3}+k_{1})}
&(\partial_{t}v_{k_{1}}v_{k_{2}}v_{k_{3}}+v_{k_{1}}\partial
_{t}v_{k_{2}}v_{k_{3}}+v_{k_{1}}v_{k_{2}}
\partial _{t}v_{k_{3}})=
\\
=&iB_{4}(v,v,v,v)_{k},
\endaligned
\end{gather}
where
\begin{equation}\label{B4sum}
B_{4}(u,v,w,\varphi )_{k}=\frac{1}{2}B_{4}^{1}(u,v,w,\varphi
)_{k}+B_{4}^{2}(u,v,w,\varphi )_{k}
\end{equation}%
and the term corresponding to $\partial _{t}v_{k_{1}}$ is $B_{4}^{1}$:
\begin{equation} \label{B41}
B_{4}^{1}(u,v,w,\varphi )_{k}=
\sum_{k_1+k_2+k_3+k_4=k}^{\text{nonres}}
\frac{e^{i\Phi (k_1,k_2,k_3,k_4)t}}
{(k_1+k_2)(k_1+k_3+k_4)(k_2+k_3+k_4)}
u_{k_1}v_{k_2}w_{k_3}\varphi_{k_4},
\end{equation}%
and the sum of the terms corresponding to $\partial _{t}v_{k_{2}}$ and $%
\partial _{t}v_{k_{3}}$ is $B_{4}^{2}$:
\begin{equation}  \label{B42}
B_{4}^{2}(u,v,w,\varphi )_{k}=\sum_{k_{1}+k_{2}+k_{3}+k_{4}=k}^{\text{nonres}%
}\frac{e^{i\Phi (k_{1},k_{2},k_{3},k_{4})t}\ (k_{3}+k_{4})}{%
k_{1}(k_{1}+k_{2})(k_{1}+k_{3}+k_{4})(k_{2}+k_{3}+k_{4})}%
u_{k_{1}}v_{k_{2}}w_{k_{3}}\varphi _{k_{4}}.
\end{equation}%
The phase function $\Phi $ here is
\begin{equation*}
\Phi
(k_{1},k_{2},k_{3},k_{4})=(k_{1}+k_{2}+k_{3}+k_{4})^{3}-k_{1}^{3}-k_{2}^{3}-k_{3}^{3}-k_{4}^{3}.
\end{equation*}%
Unlike (\ref{k1k2}) or (\ref{k1k2k3}), the above phase
function does not have a nice factorization in the general
case. However, the particular analytic
expression for $\Phi $ is not used in our subsequent analysis.

Hence, we have for $R_{3\text{nres}}(v^{3})_{k}$:
\begin{equation}\label{R3nres}
R_{3\text{nres}}(v^{3})_{k}=\frac{1}{3i}\partial _{t}
B_{3}(v,v,v)_{k}-\frac{1}{3}
\left( \frac{1}{2}
B_{4}^{1}(v,v,v,v )_{k}+B_{4}^{2}(v,v,v,v )_{k}\right).
\end{equation}

Combining~(\ref{oneres})--(\ref{B3}),
(\ref{B4sum}) and (\ref{R3nres}) we finally formally deduce
the following equivalent form of the original equation:
\begin{equation} \label{fin2}
\partial _{t}\left( v_{k}-\frac{1}{6}B_{2}(v,v)_{k}-\frac{1}{18}%
B_{3}(v,v,v)_{k}\right) =\frac{i}{6}A_{\mathrm{res}}(v)_{k}+\frac{i}{18}%
B_{4}(v,v,v,v)_{k},\ \
k\in \mathbb{Z}_{0},
\end{equation}%
where  $B_{2}$ is defined in~(\ref{B2}), $B_{3}$ is defined in
(\ref{B3}), $A_{\mathrm{res}}$ is defined in~(\ref{Ares}), and
$B_{4}$ is defined in~(\ref{B4}). We call this equation {\it the
second form of the KdV}. The smoothing properties of these
multi-linear operators, involved in~(\ref{fin2}) are established in
section~\ref{S:App}. The integrated equation of~(\ref{fin2}) takes
the fom
\begin{equation}\label{fin3}
\aligned
\left( v_{k}-\frac{1}{6}B_{2}((v)^{2})_{k}-\frac{1}{18}
B_{3}((v)^{3})_{k}\right)&(t)-\left( v_{k}-\frac{1}{6}B_{2}
((v) ^{2})_{k}-\frac{1}{18}B_{3}(( v)^{3})_{k}\right)
(0)=\\
&=\int_{0}^{t}\left( \frac{i}{6}A_{\mathrm{res}}(v)_{k}
( t^{\prime}) +\frac{i}{18}B_{4}((v)^{4})_{k}(t^{\prime})
\right) dt^{\prime }.
\endaligned
\end{equation}
\begin{remark}\label{R:two-oper}
\rm{
Let $v$ be a solution of~(\ref{v}) in the sense of
Definition~\ref{D:defmain}. Then in view of Lemma~\ref{L:B4}
and  Lemma~\ref{L:Ares} equation~(\ref{fin2}), as a
functional equation, holds in $L_\infty([0,T];\dot{H}^{0})$.
}
\end{remark}

\subsubsection*{\textbf{{\rm 3.3.}
Uniqueness and Lipschitz continuous
dependence  on the initial data}}

As will be shown in section~\ref{S:Uniq} for the case of regular
initial data, i.e.  $v^0\in\dot{H}^{s}$ for $s>1/2$, one will be
able to use the continuity properties of the trilinear operator
$R_3$ to show the uniqueness and Lipschitz continuous dependence on
the initial data of the solutions of~(\ref{v}) in the sense of
Definition~\ref{D:defmain} by the direct use of the
equation~(\ref{one}).

However, for a less regular class of initial data, i.e.
$v^0\in\dot{H}^{s}$ for $s\in[0,1/2]$ ({\it hereafter this class is
called non-regular}), we are unable to use  equation~(\ref{one}), as
in the case of regular initial data, to show the uniqueness and
Lipschitz continuity of the solutions of~(\ref{v}). Therefore, we
are required to provide a more delicate analysis to show the
uniqueness. In particular, we will use certain family of equations
for this task which are even more technical to present here and will
be the subject of section~\ref{Sec:nregin} (see
equations~(\ref{fin4}) parameterized by $n\in\mathbb{N}$).

\setcounter{equation}{0}
\section{Truncated system,
{\it a priori} estimates and global existence for $s>0$.
\label{S:Uniq}}

In this section we establish the global existence (without
uniqueness) of solutions of equation~(\ref{v}) in the sense of
Definition~\ref{D:defmain} for initial data $v^0\in\dot{H}^s$, where
$s>0$. Moreover, we will show that the established solutions
conserve the energy, namely, they satisfy
$\|v(t)\|_{L_2}=\|v^0\|_{L_2}$. For this purpose we introduce a
Galerkin version of equation~(\ref{v}), namely we replace~(\ref{v})
by an equation, with \textit{truncated} nonlinearity, for the
approximate solution $v=v^{(m)}$ for $m\in \mathbb{N}$ (in what
follows we omit the superscript $(m)$ where it causes no ambiguity)
\begin{equation} \label{vm0}
\partial _t v_{k}^{(m)}=
\frac12 ik\Pi _m\sum_{k_1+k_2=k} e^{i3kk_1k_2t}(\Pi_mv_{k_1}^{(m)})
(\Pi _m v_{k_2}^{(m)}) \qquad v_k^{(m)}(0)=v_{k}^{0}= u_{k}^{0},
\end{equation}
where $k\in\mathbb{Z}_{0}$. Here
 the truncation (projection) operator $\Pi _{m}$
acts on the sequence of  Fourier coefficients $\varphi=\{ \varphi_{k}\} $ as follows:
\begin{equation}
\Pi _{m}\varphi_{k}=(\Pi _{m}\varphi)_{k}=\begin{cases}
\varphi_{k}&\text{ if }\left\vert k\right\vert \leq m \\
0&\text{ if }\left\vert k\right\vert >m
\end{cases} ~~,  \label{Pm}
\end{equation}
so that $\partial_tv_k^{(m)}\equiv0$ for $|k|>m$.
Consequently, system~(\ref{vm0}) is in principle
a finite system of ODEs.
\begin{theorem}\label{T:locah}
Let $s_0 \geq 0$ and $v(0)=v^{0}\in \dot{H}^{s_0}$. Let
$m\in\mathbb{N}$ be fixed. Then there exits $T^*>0$, which
might depend on m, such that system~$(\ref{vm0})$ has a
unique solution on the time interval $[0,T^{\ast }]$. This
solution $v=v^{(m)}(t)$ can be extended to the maximal
interval $[ 0,T_{\max}^{\ast })$ such that either
$T_{\max }^{\ast }=+\infty $ or
$\limsup_{t\to T_{\max}^*-0}\|v^{(m)}(t)\|_{\dot{H}^0}=
+\infty$.
\end{theorem}

\begin{proof}
We observe that  system~(\ref{vm0}) is essentially a finite system
of ordinary differential equations with quadratic nonlinearity.
Therefore one can guarantee short time existence and uniqueness of
solution on an interval $[0,T^*]$ and the maximal interval of
existence $[0,T^*_{\max})$. Observe, however, that in principle
$T^*$ and $T^*_{\max}$ may depend on $m$. But later we show that
this is not the case.
\end{proof}

Next we show that $T_{\max }^{\ast }=+\infty $. To establish this
 and to be able to pass to the
limit, as $m\rightarrow \infty $,  we need to use global
\textit{a-priori} estimates for the solutions of~(\ref{vm0}).

\begin{proposition}\label{L:L2m}
Let $v^0\in\dot{H}^0$. Then for every $m\in\mathbb{N}$
the solution $v=v^{(m)}(t)$ of $(\ref{vm0})$
exists globally in time. Moreover,  the $\dot{H}^{0}$-norm
$($that is, the $L_2$-norm$)$ of
 $v$ is constant in time:
\begin{equation}\label{L2m0}
\| v^{(m)}( t)\| _{\dot{H}^{0}}^{2}= \| v( 0)\|_{\dot{H}^{0}}^{2}=
\| v^0\|_{\dot{H}^{0}}^{2}\text{ \ for all }\ t\geq 0.
\end{equation}
\end{proposition}

\begin{proof}
First we establish~(\ref{L2m0}) for $t\in [0,T^*_{\max})$ (see
Theorem~\ref{T:locah}). We observe that $\partial_tv_k^{(m)}=0$, for
all $|k|>m$, and hence $\|(I-\Pi_m)v^{(m)}(t)\|_{\dot{H}^0}^2=
\|(I-\Pi_m)v^{(m)}(0)\|_{\dot{H}^0}^2$, for $t\in [0,T^*_{\max})$.
Next we show that
\begin{equation}\label{pmconst}
\| \Pi_mv( t)\| _{\dot{H}^{0}}^{2}= \| \Pi_mv( 0)\|
_{\dot{H}^{0}}^{2}~, \text{ \ for  }\ t\in [0,T^*_{\max}).
\end{equation}
For each $|k|\le m$, we multiply (\ref{vm0})  by $\Pi
_{m}\bar{v}_{k}=\Pi _{m}v_{-k}$ and the complex conjugate of
equation~(\ref{vm0}) by $\Pi_mv_{k}$ and sum over all $|k|\le m$:
\begin{gather}
\partial _{t}\|\Pi _{m}v\| _{\dot{H}^{0}}^{2}=
\sum_{|k| \leq m}\partial _{t}v_{k}\bar{v}_{k}+
v_{k}\partial _{t}\bar{v}_{k}=\notag\\
\sum_{k\in \mathbb{Z}_{0}}(\Pi _{m}\bar{v}%
_{k})\frac{ik}{2}\sum_{k_{1}+k_{2}=k,}e^{i3kk_{1}k_{2}t}\Pi
_{m}v_{k_{1}}\Pi
_{m}v_{k_{2}}-\sum_{k\in \mathbb{Z}_{0}}(\Pi _{m}v_{k})\frac{ik}{2}%
\sum_{k_{1}+k_{2}=k,}e^{-i3kk_{1}k_{2}t}\Pi _{m}\bar{v}_{k_{1}}\Pi _{m}%
\bar{v}_{k_{2}}.\notag
\end{gather}%
Elementary transformations yield
\begin{gather*}
\partial _{t}\|\Pi _{m}v\| _{\dot{H}^{0}}^{2}=\\
\sum_{k_{1}+k_{2}-k=0}\frac{ik}{2}e^{i3kk_{1}k_{2}t} \Pi
_{m}v_{-k}\Pi_{m}v_{k_{1}}\Pi _{m}v_{k_{2}}-
\sum_{k_{1}+k_{2}-k=0}\frac{ik}{2}e^{-i3kk_{1}k_{2}t} \Pi
_{m}v_{k}\Pi _{m}v_{-k_{1}}\Pi _{m}v_{-k_{2}}.
\end{gather*}%
Setting $k=-k_{3}$ we obtain
\begin{gather*}
\partial _{t}\| \Pi _{m}v\| _{\dot{H}^{0}}^{2}=-
\sum_{k_{1}+k_{2}+k_{3}=0}\frac{1}{2}
ik_{3}e^{-i3k_{3}k_{1}k_{2}t}\Pi_{m}v_{k_{3}}
\Pi _{m}v_{k_{1}}\Pi _{m}v_{k_{2}}+\\
\sum_{k_{1}+k_{2}+k_{3}=0}
\frac{1}{2}ik_{3}e^{i3k_{3}k_{1}k_{2}t}\Pi _{m}v_{-k_{3}}
\Pi_{m}v_{-k_{1}}\Pi _{m}v_{-k_{2}},
\end{gather*}%
and changing in the second sum $k_{j}$ to $-k_{j}$ \ yields
\begin{equation}\label{pm1}
\partial _{t}\| \Pi _{m}v\| _{\dot{H}^{0}}^{2}=-%
\sum_{k_{1}+k_{2}+k_{3}=0}ik_{3}e^{-i3k_{3}k_{1}k_{2}t}
\Pi _{m}v_{k_{3}}\Pi
_{m}v_{k_{1}}\Pi _{m}v_{k_{2}}.
\end{equation}%
Since $k_3=-k_1-k_2$, this equation is equivalent to
\begin{gather*}
\partial _{t}\| \Pi _{m}v\| _{\dot{H}^{0}}^{2}=\\
\sum_{k_{1}+k_{2}+k_{3}=0}ik_{1}e^{-i3k_{3}k_{1}k_{2}t}
\Pi _{m}v_{k_{3}}\Pi
_{m}v_{k_{1}}\Pi
_{m}v_{k_{2}}+\sum_{k_{1}+k_{2}+k_{3}=0}
ik_{2}e^{-i3k_{3}k_{1}k_{2}t}\Pi
_{m}v_{k_{3}}\Pi _{m}v_{k_{1}}\Pi _{m}v_{k_{2}}.
\end{gather*}%
Exchanging $k_{1}$ and $k_{3}$  in the first term and $k_{2}$ and $k_{3}$
in the second term we get
\begin{equation}\label{pm2}
\partial _{t}\| \Pi _{m}v\|_{\dot{H}^{0}}^{2}=2
\sum_{k_{1}+k_{2}+k_{3}=0}ik_{3}e^{-i3k_{3}k_{1}k_{2}t}
\Pi _{m}v_{k_{3}}\Pi_{m}v_{k_{1}}\Pi _{m}v_{k_{2}}.
\end{equation}
Obviously, (\ref{pm1}) and (\ref{pm2}) imply that
$\partial _{t}\|\Pi _{m}v\| _{\dot{H}^{0}}^{2}=0$, which gives
(\ref{pmconst}) and, consequently,
\begin{equation*}%\label{pmconst}
\| v( t)\| _{\dot{H}^{0}}^{2}=
\| v( 0)\| _{\dot{H}^{0}}^{2}\text{ \ for  }\ t\in [0,T^*_{\max}).
\end{equation*}
Therefore, based on Theorem~\ref{T:locah}, $T^*_{\max}=\infty$,
which completes the proof.
\end{proof}

\begin{remark}\label{R:L-2_conservation}
\rm{ We remark that a somewhat similar argument is used in
\cite{Colliander03} for the Fourier proof of the $L_2$-conservation
property for the KdV equation on $\mathbb{R}$. }
\end{remark}

We now derive higher order Sobolev norm estimates for the solutions
$v=v^{(m)}(t)$ of the truncated system~(\ref{vm0}). For this purpose
we carry out the first and the second differentiation by parts for
the truncated system~(\ref{vm0}), as we have done above in section
3.1. Similarly to~(\ref{intby1}) we obtain that solution
$v^{(m)}(t)$, which was constructed in Theorem~\ref{T:locah} and
Proposition~\ref{L:L2m}, satisfies the system
\begin{equation}\label{onem}
\partial _{t}\left( v_{k}^{(m)}-\frac{1}{6}{B}_{2,m}
(v^{(m)},v^{(m)})_{k}\right) =%
\frac{i}{6}{R}_{3,m}(v^{(m)},v^{(m)},v^{(m)})_{k},
\quad k\in \mathbb{Z}_{0},
\end{equation}
where
\begin{equation}\label{B2m}
{B}_{2,m}(v,v)_{k}=\Pi_{m}{B}_{2}(\Pi _{m}v,\Pi_{m}v)_{k}=\Pi_{m}
\sum_{k_{1}+k_{2}=k} \frac{e^{i3kk_{1}k_{2}t}\Pi_{m}v_{k_{1}}\Pi
_{m}v_{k_{2}}}{k_{1}k_{2}}\,,
\end{equation}
and
\begin{equation}  \label{R3m}
R_{3,m}(v,v,v)_{k}=\Pi_{m}\sum_{k_{1}+k_{2}+k_{3}=k}
\frac{e^{i3(k_{1}+k_{2})(k_{2}+k_{3})(k_{3}+k_{1})t}}{k_{1}}
\Pi _{m}v_{k_{1}} \Pi _{m}v_{k_{2}}\Pi _{m}v_{k_{3}}.
\end{equation}
After the second differentiation by parts we see
that $v^{(m)}(t)$ satisfies the following
finite dimensional analog of~(\ref{fin2}):
\begin{equation}  \label{fin2m}
\partial _{t}\left( v_{k}^{(m)}-\frac{1}{6}
B_{2,m}((v^{(m)})^2)_{k}-\frac{1}{18}%
B_{3,m}((v^{(m)})^3)_{k}\right) =
\frac{i}{6}A_{\mathrm{res},m}(v^{(m)})_{k}+\frac{i}{18}%
B_{4,m}((v^{(m)})^4)_{k},
\end{equation}
where
\begin{equation}\label{trm}
B_{3,m}(v^3)=\Pi_mB_3((\Pi_mv)^3),
\quad
B_{4,m}(v^4)=\Pi_mB_4((\Pi_mv)^4).
\end{equation}
Concerning the resonant operator, by Proposition~\ref{L:L2m} we have
the conservation of energy for the truncated system~(\ref{vm0});
hence, arguing as in Remark~\ref{R:conservation} and similar
to~(\ref{R3resconserv}) we have
\begin{equation}\label{Aresm}
 A_{\mathrm{res},m}(v)_k=
 \left\{%
\begin{array}{ll}
\frac{1}{k}(\Pi_m v_k)(\|v^0\|_{\dot{H}^0}^2-|\Pi_mv_k|^2)
 & \hbox{$|k|<m$} \\
    0 & \hbox{$|k|\ge m$}
\end{array}%
\right.\, .
\end{equation}

We now establish estimates for the solutions $v^{(m)}$
in higher order Sobolev spaces.

\begin{theorem} \label{T:mainc}
Let $M_0$ be given and let $s_{0}\geq 0$. Assume that
$v^{(m)}(0)=v^0\in\dot{H}^{s_0}$ and suppose further that
$\|v^{(m)}(0)\|_{\dot{H}^{0}}\le M_0$.
Let $T>0$  and let
$v^{(m)}(t)$
be a solution of $(\ref{vm0})$ over $[0,T]$.
Then $v^{(m)}(t)$ solves~$(\ref{fin2m})$ and,
uniformly in $m$,  satisfies the estimate
\begin{equation}\label{ms}
\|v^{(m)}(t)\|_{\dot{H}^{s_{0}}}\leq M_{s_{0}}\text{ \ \ for
\ }0\leq t\leq T,
\end{equation}
where $M_{s_0}$ depends only on $M_0$, $T$,
$\|v(0)\|_{\dot{H}^{s_0}}$ and $s_0$.
\end{theorem}

\begin{proof}
Since we are dealing with finite systems of ordinary differential
equations and the sums are carried out over finite number of
indices, it is clear that the solution of $(\ref{vm0})$
solves~$(\ref{fin2m})$. Moreover, thanks to Proposition~\ref{L:L2m}
we have
\begin{equation}\label{m0}
\|v^{(m)}(t)\|_{\dot{H}^{0}}=
\|v^{(m)}(0)\|_{\dot{H}^{0}}\leq M_{0}.
\end{equation}
Since $\partial_tv_k^{(m)}=0$ for $|k|>m$, it follows that
\begin{equation}\label{Qv}
\|(I-\Pi_m)v^{(m)}(t)\|_{\dot{H}^{s_0}}=
\|(I-\Pi_m)v^{(m)}(0)\|_{\dot{H}^{s_0}}\le
\|v^0\|_{\dot{H}^{s_0}}.
\end{equation}
Therefore, we need to focus on estimating
$\|\Pi_mv^{(m)}\|_{\dot{H}^{s_0}}$. We set in~(\ref{fin2m})
$$
z_{k}^{(m)}=v_{k}^{(m)}-w_{k}^{(m)}(v^{(m)})-
\varphi_{k}^{(m)}(v^{(m)}),\ k\in \mathbb{Z}_{0},
$$
where
$$
w^{(m)}(v^{(m)})=\frac{1}{6}B_{2,m}(v^{(m)},v^{(m)}),\ \
\varphi^{(m)}(v^{(m)})=
\frac{1}{18}B_{3,m}(v^{(m)},v^{(m)},v^{(m)})
$$
and
\begin{equation*}
Q_m(v^{(m)})_{k}=\frac{i}{6}A_{\mathrm{res,m}}(v^{(m)})_{k}+
\frac{i}{18}B_{4,m}(v^{(m)},v^{(m)},v^{(m)},v^{(m)})_{k},
\quad k\in \mathbb{Z}_{0}.
\end{equation*}
Then~(\ref{fin2m}) goes over to
\begin{equation}\label{apreq}
\partial _{t}z_{k}^{(m)}=Q_m(v^{(m)})_{k}=
Q_m(z^{(m)}+w^{(m)}(v^{(m)})+\varphi^{(m)}(v^{(m)}))_{k}
\end{equation}
with $z^{(m)}(0)=z^{0}=v^{0}-w^{(m)}(v^{0})-\varphi^{(m)}(v^{0})$,
where $z^{(m)}(0)\in \dot{H}^{s_{0}}$ in view of Lemmas~\ref{L:B2}
and \ref{L:B3}. We observe that the operators $B_{2,m}$, $B_{3,m}$
and $B_{4,m}$ in~(\ref{fin2m}) satisfy the same estimates in the
same spaces as $B_2$, $B_{3}$  and $B_{4}$, respectively, and the
estimates are uniform in~$m$. We fix a positive integer $n_{0}$ so
that $s_{0}/n_{0}=\varepsilon _{0}<1/2$. Our goal is to establish an
explicit bound for $\sup_{0\le t\le
T}\|\Pi_mv^{(m)}(t)\|_{\dot{H}^{s_{0}}}$. This is done in $n_{0}$
steps. By Lemmas~\ref{L:B4} and~\ref{L:Ares} below, $Q_m$ is bounded
from $\dot{H}^{s}$ to $\dot{H}^{s+\varepsilon }$, for $s\geq 0$ and
$0 < \varepsilon <1/2$:
\begin{equation}\label{EE11}
\|Q_m(v^{(m)})\|_{\dot{H}^{s+\varepsilon }}\leq \frac{1}{18}
c_{4}(s,\varepsilon )\|v^{(m)}\|_{\dot{H}^{s}}^{4}+
\frac{1}{6}c_{5}(s)\|v^{(m)}\|_{\dot{H}^{s}}.
\end{equation}
In view of~(\ref{m0}) and~(\ref{EE11})
 the right-hand side in~(\ref{apreq}) is bounded in
$\dot{H}^{\varepsilon _{0}}$ (we take $s=0$ and
$\varepsilon=\varepsilon_0$ in~(\ref{EE11})):
\begin{equation*}
\|Q_m(v^{(m)}(t))\|_{\dot{H}^{\varepsilon _{0}}}\leq
C_{\varepsilon_{0}}=C_{\varepsilon _{0}}(M_0).
\end{equation*}%
We take the scalar product of~(\ref{apreq}) and $\Pi_mz(t)$ in
$\dot{H}^{\varepsilon _{0}}$, that is, we multiply the
equation~(\ref{apreq}) by $|k|^{2\varepsilon _{0}}\bar{z}_{k}$ and
sum the results over all $|k| \le m$. We obtain
\begin{equation*}
\partial_{t}\|\Pi_mz^{(m)}\|_{\dot{H}^{\varepsilon _{0}}}^{2}=
2(Q_m(v^{(m)}),\Pi_mz^{(m)})_{\dot{H}^{\varepsilon _{0}}}\leq
\|Q_m(v^{(m)})\|_{\dot{H}^{\varepsilon_{0}}}^{2}+
\|\Pi_mz^{(m)}\|_{\dot{H}^{\varepsilon _{0}}}^{2},
\end{equation*}%
where $\Pi_mz(0)$ is bounded in
$\dot{H}^{s_{0}}\subseteq \dot{H}^{\varepsilon_{0}}$ since
$s_{0}\geq \varepsilon _{0}$. By the Gronwall inequality
\begin{equation*}
\|\Pi_mz^{(m)}(t)\|_{\dot{H}^{\varepsilon _{0}}}\leq
C_{\varepsilon_{0}}(T,M_0),\quad t\in \lbrack 0,T],
\end{equation*}%
and therefore $\Pi_mv^{(m)}(t)=\Pi_mz^{(m)}(t)+w^{(m)}(v^{(m)}(t))+
\varphi^{(m)}(v^{(m)}(t))$ is bounded in $\dot{H}^{\varepsilon
_{0}}$ for $t\in[0,T]$ with bound that depends only on $T$, $M_0$
and $\|v^0\|_{\dot{H}^{s_0}}$ (note that $w^{(m)}(v^{(m)})$ and
$\varphi^{(m)}(v^{(m)})$ for $v^{(m)}\in L_{2}$ are bounded in
$\dot{H}^{1}$ and $\dot{H}^{2}$, respectively, see Lemmas~\ref{L:B2}
and~\ref{L:B3} below; and thus they are bounded in
$\dot{H}^{\varepsilon _{0}}$ since $\varepsilon _{0} < 1/2$).
Consequently, $Q_m(v^{(m)}(t))$ on the right-hand side
in~(\ref{apreq}) is bounded in $\dot{H}^{2\varepsilon_{0}}$ (thanks
to~(\ref{EE11})) and taking the scalar product of~(\ref{apreq}) and
$\Pi_mz^{(m)}$ in $\dot{H}^{2\varepsilon _{0}}$ we obtain, as above,
that $\Pi_mz^{(m)}(t)$ and, hence, $\Pi_mv^{(m)}(t)$ are bounded in
$H^{2\varepsilon_{0}}$ uniformly for $t\in[0,T]$ with bound that
depends only on $T$, $M_0$ and $\|v^0\|_{\dot{H}^{s_0}}$. After
$n_{0}$ steps we get on the interval $[0,T]$
\begin{equation}\label{aprest}
\|\Pi_mv^{(m)}(t)\|_{\dot{H}^{s_{0}}}\leq C_{s_{0}}=
C_{s_{0}}(T,\|v^{0}\|_{\dot{H}^{s_{0}}},M_0),
\end{equation}
which together with~(\ref{Qv}) completes the proof
of the theorem.
\end{proof}

\begin{proposition}
\label{L:ests1m} Let $T>0$, $s_{0}\ge 0$ and $S>3/2$. Then  the
solutions $v=v^{(m)}$ of~$(\ref{vm0})$ with $v(0)\in\dot{H}^{s_{0}}$
satisfy, uniformly in $m$, $T$, and $s_0$,  the estimate
\begin{equation}\label{dtv}
\| \partial _{t}v^{(m)}(t)\|_{\dot{H}^{-S}}\le
 c_1(S-1)\| v(0)\|_{\dot{H}^{0}}^2.
\end{equation}
%where $c(s)$ is defined in~$(\ref{c2s})$.
\end{proposition}
\begin{proof}
This immediately follows from Lemma~\ref{L:B1} and the conservation
of energy of the  solutions $v^{(m)}$ of~(\ref{vm0}), see
Proposition~\ref{L:L2m}. The constant $c_1(S-1)$ is also defined in
Lemma~\ref{L:B1}.
\end{proof}

We are now ready to establish the existence of weak solutions to
equation~{(\ref{v}) as the limit of a subsequence of $v^{(m)}(t)$,
the solutions of  the truncated system~(\ref{vm0}). We first deal
with the case $s>0$. The case $s=0$ is considered in
Theorem~\ref{T:globs=0} below.

\begin{theorem}\label{T:gal}
 Let $s_0>0$, $v(0)\in\dot{H}^{s_0}$ and let $T>0$ be fixed.
Let $\sigma$ be fixed
 so that $0<\sigma<s_0$. Then there exists a
subsequence of~$v^{(m)}(t)$ $($which we still label by $(m)$$)$ of
solutions of~$(\ref{vm0})$,
 with the same  initial
data $v(0)\in\dot{H}^{s_0}$,  such that $v^{(m)}(t)$ converges
strongly in $L_p([0,T];H^{\sigma})$, for any fixed $1<p<\infty$, and
$*$-weakly in $L_{\infty }([0,T];\dot{H}^{s_0})$ to $v^{\infty
}(t)$; and $v^{\infty}(t)$ is a solution of~$(\ref{v})$ in the sense
of Definition~$\ref{D:defmain}$ and is bounded in
$L_\infty([0,T],\dot{H}^{s_0})$ with norm
\begin{equation}\label{vinftynorm}
\|v^\infty\|_{L_\infty([0,T],\dot{H}^{s_0})}\le
M_{s_0},
\end{equation}
where $M_{s_0}$ is as in $(\ref{ms})$. Moreover,
\begin{equation}\label{Energy-infinity}
\|v^\infty(t)\|_{\dot{H}^0}=\|v^0\|_{\dot{H}^0} \quad \text {\it
a.e. in} \quad [0,T].
\end{equation}
\end{theorem}

\begin{proof}
We first observe that based on Proposition~\ref{L:L2m} the solution
$v^{(m)}(t)$ of~(\ref{vm0}) exists globally in time. Next we set
$\theta=-S$, where $S>3/2$ is as in Proposition~\ref{L:ests1m}. It
follows from (\ref{ms}) and (\ref{dtv}) that $v^{(m)}$ is bounded in
$L_\infty([0,T];H^{s_0})\subset L_p([0,T];H^{s_0})$; while the time
derivative $\partial_tv^{(m)}$ is bounded in
$L_\infty([0,T];H^\theta)\subset L_p([0,T];H^\theta)$ uniformly with
respect to $m$. Since the imbedding $\dot{H}^{\sigma}\subset
\dot{H}^{\theta }$ is compact, the existence of a subsequence that
converges strongly in $L_p([0,T];H^{\sigma})$
 and $*$-weakly in
$L_{\infty }([0,T];\dot{H}^{s_0})$ to a function $v^{\infty }$
follows from the classical compactness theorem (Aubin compactness
theorem). Furthermore, $v^{\infty}$ is bounded and continuous with
values in $\dot{H}^\theta$ and is bounded and weakly continuous with
values in $\dot{H}^{\gamma}$ for any $\gamma$, $\theta\le\gamma\le
s_0$
%in particular, for $\gamma=0$
(see, for instance, \cite{CF88}, \cite{TNS}).

Since $v^{(m)}$ converges to $v^{\infty}$ strongly in
$L_p([0,T],\dot{H}^\sigma)$ with $\sigma>0$, it follows that there
is a subsequence, also denoted by $v^{(m)}(t)$, which converges to
$v^{\infty}(t)$ strongly in $\dot{H}^\sigma$ and hence strongly in
$\dot{H}^0$ for almost every $t\in[0,T]$. However,
$\|v^{(m)}(t)\|_{\dot{H}^0}=\|v(0)\|_{\dot{H}^0}$ for every $t$.
Therefore $\|v^{\infty}(t)\|_{\dot{H}^0}=\|v(0)\|_{\dot{H}^0}$ for
almost every $t\in [0,T]$, which proves~(\ref{Energy-infinity}).

Every $v_{k}^{(m)}(t)$ is a solution of~(\ref{vm0}),
which we write, using~(\ref{B1}), as follows
\begin{equation}\label{B1m}
v^{(m)}_k(t)-v_k(0)=
\int_0^t
\Pi_mB_1(\Pi_mv^{(m)}(\tau),\Pi_mv^{(m)}(\tau))_kd\tau.
\end{equation}
Therefore, using the symmetry of $B_1$ and setting
$\Pi_{-m}=I-\Pi_m$ we obtain
\begin{equation}\label{passlim}
\aligned
v^{(m)}_k(t)-v^\infty_k(t)+v^\infty_k(t)-v_k(0)=\\
\int_0^t
\Pi_mB_1\bigl(\Pi_m(v^{(m)}(\tau)-v^\infty(\tau)),
\Pi_m(v^{(m)}(\tau)+v^\infty(\tau))\bigr)_kd\tau-\\
\int_0^t
\Pi_mB_1\bigl(\Pi_{-m}v^\infty(\tau),
\Pi_mv^{\infty}(\tau)+v^\infty(\tau)\bigr)_kd\tau-\\
\int_0^t \Pi_{-m}B_1\bigl(v^\infty(\tau),
v^\infty(\tau)\bigr)_kd\tau+\\
\int_0^t
B_1\bigl(v^\infty(\tau),
v^\infty(\tau)\bigr)_kd\tau.
\endaligned
\end{equation}
By  Lemma~\ref{L:B1} and the readily established
convergence of $v^{(m)}\to v^\infty$ in
$L_p([0,T];\dot{H}^\sigma)$ and the fact that
$v^\infty\in L_\infty([0,T];\dot{H}^{s_0})$ we
observe that the first three terms on the
right-hand side converge to zero. For example, in
view of Lemma~\ref{L:B1} the first term on the
right-hand side is less in absolute value than
$$
\aligned
|k|^{-\theta}
\int_0^t
\|B_1\bigl(\Pi_m(v^{(m)}(\tau)-v^\infty(\tau)),
\Pi_m(v^{(m)}(\tau)+v^\infty(\tau))\bigr)\|_{\dot{H}^\theta}
d\tau\le\\
c_1(\theta)|k|^{-\theta}\int_0^t
\|v^{(m)}(\tau)-v^\infty(\tau)\|_{\dot{H}^0}
\|v^{(m)}(\tau)+v^\infty(\tau)\|_{\dot{H}^0}d\tau\le\\
2c_1(\theta)|k|^{-\theta}\|v(0)\|_{\dot{H}^0}
\|v^{(m)}-v^\infty\|_{L_1([0,T],\dot{H}^0)}
\to 0\ \text{ as }\ m\to\infty.
\endaligned
$$
The second term is treated similarly. By the
definition of $\Pi_{-m}$  the third term is
identically zero for $m\ge|k|$. Since
$v_{k}^{(m)}(t) \to
v_{k}^{\infty }(t)$, as $m\to\infty$,
  for each $k\in \mathbb{Z}_{0}$ and almost every $t\in[0,T]$,
 passing to the limit in (\ref{passlim})
  we obtain the equality
\begin{equation}\label{B1infty}
v^{\infty}_k(t)-v_k(0)=
\int_0^t
B_1(v^{\infty}(\tau),v^{\infty}(\tau))_kd\tau,
\end{equation}
which holds for almost every $t$.
As in Remark~\ref{R:regularity} we see that
(\ref{B1infty}) holds for every $t$ and $v^{\infty}$
is a solution of~$(\ref{vint})$ in the sense
of Definition~\ref{D:defmain}.
\end{proof}
\begin{corollary}\label{C:one-two}
Let $s_0>0$, $v(0)\in\dot{H}^{s_0}$. The solution $v^{\infty}\in
L_\infty([0,T];\dot{H}^{s_0})$ constructed in Theorem~$\ref{T:gal}$
also satisfies the integrated first form of the
KdV~$(\ref{defofsol})$ and the integrated second  form of the KdV
equation~$(\ref{fin3})$.
\end{corollary}
\begin{proof}
Applying the formal steps introduced in section~\ref{S:Trans} to the
solutions $v^{(m)}(t)$ of equation~(\ref{vm0}) one can rigorously
show that $v^{(m)}$ is the solution of the corresponding Galerkin
versions of (\ref{defofsol}) and  (\ref{fin3}). Therefore we can
pass to the limit along a subsequence, as $m\to\infty$, in the same
way as we did in~(\ref{passlim}). We also observe that by the
continuity properties of the operators $B_2$ and $R_3$ in
(\ref{defofsol}) and $B_2$, $B_3$, $B_4$, and $A_{\mathrm{res}}$ in
(\ref{fin3}), respectively, all the terms in (\ref{defofsol}) are
bounded from $\dot{H}^s$ to $\dot{H}^s$, for $s>1/2$, while  the
terms in (\ref{fin3}) are bounded from $\dot{H}^s$ to $\dot{H}^s$
for $s\ge0$.
\end{proof}

\setcounter{equation}{0}
\section{ Global well-posedness and Lipschitz
continuity
(regular case, $s>1/2$)\label{S:Uniq1/2}}

Here we will consider regular initial data in $\dot{H}^{s}$, for
$s>1/2$, and will work with  equation~(\ref{defofsol}). Less regular
data will be considered in section~\ref{Sec:nregin}.

By Corollary~\ref{C:one-two} there exists a solution
of~(\ref{defofsol}).
For the proof of the  uniqueness
we need the following auxiliary result.
We introduce the projection $\Pi _{0}$ on  the
zero Fourier mode,
$\Pi _{0}( \{ g_{k},k\in \mathbb{Z}\})=g_{0}$.

\begin{lemma}
\label{L:linB2} Let $\varphi \in \dot{H}^{s-1}$, $s>1/2$. Then the
linear operator $L_{\varphi }:\dot{H}^{\theta }\rightarrow
\dot{H}^{\theta}$, for $\theta >-1/2$,
\begin{equation}
L_{\varphi }v=v-B_{2}(\varphi ,v)
,\qquad
(L_{\varphi }v)_{k}=v_{k}-B_{2}(\varphi ,v)_{k}
\quad k\in \mathbb{Z}%
_{0},  \label{B2lin}
\end{equation}%
where $B_{2}(u,v)_{k}=B_{2}(u,v,t)_{k}$  is defined
in~$(\ref{B2})$, has range $\dot{H}^{\theta }$
for every fixed $t$. Moreover, for every
$f\in \dot{H}^{\theta}$ the equation
\begin{equation*}
L_{\varphi }v=f
\end{equation*}
has a unique solution
$v=L_{\varphi }^{-1}f\in \dot{H}^{\theta}$,
and the following estimate holds
\begin{equation}
\|v\|_{\dot{H}^\theta}\leq\|L_{\varphi
}^{-1}\|_{\mathcal{L}(\dot{H}^\theta)}\|f\|_{\dot{H}^\theta}.
\label{B2linest}
\end{equation}

Furthermore, for every $\varphi \in \dot{H}^{-1}\cap\dot{H}^{s-2}$,
with $s> 1/2$, the inverse operator $L_{\varphi }^{-1}$ can be
extended as a bounded linear map from $\dot{H}^{\theta }$ into
itself, for $ \theta \in [-1,s]$, with norm satisfying the estimate
\begin{equation}
\Vert L_{\varphi }^{-1}\Vert _{\mathcal{L}(\dot{H}^\theta)}\leq
F(\Vert \varphi \Vert _{\dot{H}^{-1}},
 \Vert \varphi \Vert _{\dot{H}^{s-2}}),  \label{B2estofnorm}
\end{equation}
where $F$ is a monotonic increasing function in each argument,
which is independent of $t$.
\end{lemma}
\begin{remark}
\label{R:s-and-theta}
{\rm
Note that if $s\le1$, then clearly
$F(\Vert \varphi \Vert _{\dot{H}^{-1}},
 \Vert \varphi \Vert _{\dot{H}^{s-2}})\le
 F(\Vert \varphi \Vert _{\dot{H}^{-1}},
 \Vert \varphi \Vert _{\dot{H}^{-1}})$.
}
\end{remark}

\begin{proof}
We first show that the homogeneous equation
\begin{equation}
L_{\varphi }v=v-B_{2}(\varphi ,v)
=0,\quad v\in \dot{H}^{\theta }\ \
\text{for}\ \
\theta >-1/2
\label{bhom}
\end{equation}
has only the trivial solution $v=0$. Note that according to
Lemma~\ref{L:B2} the  bilinear operator $B_{2}(\varphi ,v)$ is
bounded from $\dot{H}^{\theta}\times \dot{H}^{\theta }$ into
$\dot{H}^{\theta +1}$, for $\theta >-1/2$. Therefore, if $v\in
\dot{H}^{\theta }$ is a solution of~(\ref{bhom}), then $v\in
\dot{H}^{\, \min(\theta,s-1) +1}$. Hence, it is sufficient to
consider solutions $v$  in $\dot{H}^{\sigma}$, for $\sigma>1/2$.

In terms of the Fourier coefficients~(\ref{bhom}) reads
\begin{equation*}
v_{k}-\sum_{k_{1}+k_{2}=k}\frac{e^{i3kk_{1}k_{2}t}\varphi _{k_{1}}v_{k_{2}}}{%
k_{1}k_{2}}=0,\quad k\in \mathbb{Z}_{0}.
\end{equation*}
Setting $u_{k}=e^{-ik^{3}t}v_{k}$ and $\psi _{k}=e^{-ik^{3}t}\varphi _{k}$
and taking into account~(\ref{k1k2}) we obtain
\begin{equation}
u_{k}-\sum_{k_{1}+k_{2}=k}
\frac{\psi _{k_{1}}u_{k_{2}}}{k_{1}k_{2}}=0
\label{omu}
\end{equation}
with $\|u\| _{\dot{H}^{s}}=\| v\| _{\dot{H}^{s}}$.
We now define $w(x)$ and $\xi(x)$ by  setting
\begin{equation*}
w_{k}=\frac{{u_{k}}}{{ik}},
\quad
\xi _{k}=\frac{\psi _{k}}{ik},\
k\in\mathbb{Z}_0,
\quad
w_0=\xi _{0}=0.
\end{equation*}
By definition, both $w$ and $\xi$ have
 mean value zero:
$\int_{0}^{2\pi }w(x)dx=\int_{0}^{2\pi }\xi(x)dx=0$ and,
in addition, $w'(x)=u(x)$, $\xi'(x)=\psi(x)$
so that
 $\|u\| _{\dot{H}^{s}}=\| w\| _{\dot{H}^{s+1}}$,
  $\|\psi\| _{\dot{H}^{s}}=\|\xi\| _{\dot{H}^{s+1}}$.
Multiplying~(\ref{omu}) by $e^{ikx}$ and summing with respect to
\textit{all} $k\in \mathbb{Z}$ we see that
\begin{equation*}
w(x)=\sum_{k\in \mathbb{Z}_{0}}\frac{u_{k}}{ik}\,e^{ikx},
\qquad w^{\prime}(x)=u(x),
\end{equation*}
 satisfies the
following boundary value problem:
\begin{equation}
\aligned w^{\prime }(x)+\xi (x)w(x)=\frac{1}{2\pi }\int_{0}^{2\pi }\xi
(x)w(x)dx,  \label{bvp} \\
w\text{ is periodic, with period } 2\pi, \quad \text{and} \quad
\int_{0}^{2\pi }w(x)dx=0,\endaligned
\end{equation}%
in the classical sense since it is given that $\xi\in H^{s+1}\subset
C^{1}$, for $s>1/2$, and we are looking for a solution $w\in
H^{\theta+1}\subset C^{1}$, for $\theta>1/2$. Note that the
right-hand side of (\ref{bvp}) is the projection $\Pi _{0}(\xi w)$.
Then
\begin{equation*}
w^{\prime }(x)+\xi (x)w(x)=\tilde{c},
\end{equation*}
where
\begin{equation}
\tilde{c}=
\Pi _{0}(\xi w)=
\frac{1}{2\pi }\int_{0}^{2\pi }\xi (x)w(x)dx  \label{tildec}
\end{equation}
and we can treat this equation as a first-order linear ordinary differential
equation with general solution
\begin{equation*}
w(x)=Ce^{-\int_{0}^{x}\xi (\sigma )d\sigma }+\tilde{c}\int_{0}^{x}e^{%
\int_{x}^{r}\xi (\sigma )d\sigma }dr,
\quad C\in\mathbb{R}.
\end{equation*}
Since only the first term on the right-hand side is periodic, it
follows that $w(x)$ is periodic if and only if $\tilde{c}=0$, and
then in this case $w(x)$ has mean value zero if and only if $C=0$.
Hence, $w=0$ and, therefore, $u=0$ and $v=0$.

Now we prove that the range $L_{\varphi }\dot{H}^{\theta }$ of
$L_{\varphi }$  coincides with $\dot{H}^{\theta }$, and that
$L_{\varphi }$ has  a bounded inverse. For $s> \theta$ this follows
from the Fredholm theory. Observe that the embedding
$\dot{H}^{\sigma }\to\dot{H}^{\sigma-\varepsilon }$ is compact if
$\varepsilon >0$. By Lemma~\ref{L:B2}, $B_{2}(\varphi ,v)$ is
bounded from $\dot{H}^{s-1 }\times \dot{H}^{\theta }$ into
$\dot{H}^{\theta+1}$ for $\theta\le s-1$, while for $\theta> s-1$ it
is bounded from $\dot{H}^{s-1 }\times \dot{H}^{\theta }$ into
$\dot{H}^{s}$. Hence $B_2(\varphi,\cdot)$ is compact from
$\dot{H}^{\theta }$ into itself and $L_{\varphi }=I-B_{2}(\varphi
,\cdot )$ is a Fredholm operator in $\dot{H}^{\theta }$, the index
(dimension of nullspace minus codimension of range) of this operator
is zero. Therefore, the range $L_{\varphi }\dot{H}^{\theta }$ is a
closed subspace in $\dot{H}^{\theta }$ with finite codimension, and
the codimension equals dimension of the nullspace of $L_{\varphi }$.
We already have proven that the nullspace is trivial, therefore the
codimension is zero and $L_{\varphi }\dot{H}^{\theta
}=\dot{H}^{\theta }$.
 Hence, the inverse operator $L_{\varphi }^{-1}$ exists and
is bounded.

Note that $L_{\varphi }$ depends on $t$ as well as on $\varphi $ and
we have boundedness of $L_{\varphi }^{-1}$ for every given $t$ and
$\varphi $. Next we prove a uniform estimate for $L_{\varphi }^{-1}$
for $t\in [0, T]$ and $\varphi$ satisfying $\|\varphi
\|_{\dot{H}^{s-1}}\leq C$.
 To estimate the norm of $L_{\varphi }^{-1}$ is sufficient to
consider it on a dense set $\dot{H}^{s}$, $s>1/2$, of smooth
enough functions~$v$.  For such functions we write inverse
operator $L_{\varphi }^{-1}$ explicitly
and obtain the explicit norm estimate~(\ref{B2estofnorm}).
To prove~(\ref{B2linest}) and (\ref{B2estofnorm}) we
consider for $\varphi\in \dot{H}^{s-1}$ and
$f\in \dot{H}^{s}$ the nonhomogeneous equation
\begin{equation*}
L_{\varphi }v=v-B_{2}(\varphi ,v)=f.
\end{equation*}
For $k\in\mathbb{Z}_0$ we set $u_{k}=e^{-ik^{3}t}v_{k}$,
$\psi _{k}=e^{-ik^{3}t}\varphi _{k}$,
$g_{k}=e^{-ik^{3}t}f_{k}$ and write the equation for~$w$:
\begin{equation*}
w^{\prime }+\xi w=g+\tilde{c},
\end{equation*}%
where, as before, $w^{\prime }=u$, $\xi ^{\prime }=\psi $ and
$w$, $\xi $, $g$ depend on $t$ as a parameter. Since $g$ has
mean value zero, the constant $\tilde{c}=\tilde{c}(w)$ as
before satisfies~(\ref{tildec}). The general solution is
\begin{equation}
w(x)=\left( \int_{0}^{x}(g(r)+\tilde{c})e^{\Xi(r) }dr+C\right)
e^{-\Xi(x)},
\quad
\Xi (x) =\int_{0}^{x}\xi (\sigma)d\sigma,
\label{w}
\end{equation}
where $C\in\mathbb{R}$ is a constant. Since $\xi $ is
periodic with mean value zero, $\Xi (x) $ is also periodic.
The condition on $w$ to be periodic gives
\begin{equation*}
\int_{0}^{2\pi }(g(r)+\tilde{c})e^{\Xi(r)}dr=0,
\end{equation*}
which uniquely defines $\tilde{c}$ in terms of $g$ and $\xi $:
\begin{equation}
\tilde{c}=-\frac{\int_{0}^{2\pi }g(r)e^{\Xi (r)}dr}
{\int_{0}^{2\pi }e^{\Xi (r) }dr}\,.
\label{cc0}
\end{equation}%
We observe that $w$ defined in~(\ref{w})
satisfies~(\ref{tildec}) (for any $C$). The condition
$\int_{0}^{2\pi }w(x)dx=0$ uniquely defines the constant $C$:
\begin{equation}
C=-\frac{\int_{0}^{2\pi }e^{-\Xi \left( x\right) }\int_{0}^{x}(g(r)+\tilde{c}%
)e^{\Xi \left( r\right) }drdx}{\int_{0}^{2\pi }e^{-\Xi \left( x\right) }dx}%
\,,  \label{cc}
\end{equation}%
and, hence, the solution $w(x)$ of~(\ref{w}) is uniquely defined. We
denote
\begin{equation*}
G(r)=\int_{0}^{r}g(r_{1})dr_{1},
\end{equation*}%
since $g$ is periodic with mean value zero, $G(r)$ is also periodic,
with $G(0)=G(2\pi)$. Integrating by parts we get
\begin{equation}
\int_{0}^{x}g(r)e^{\Xi (r) }dr=\int_{0}^{x}e^{\Xi (r)}dG(r)=
G(x)e^{\Xi(x)}-\int_{0}^{x}G(r)\xi (r)e^{\Xi(r)}dr
\label{gbp}
\end{equation}
and rewrite (\ref{w})  in the form
\begin{equation}
w(x)=G(x)+\left( C-\int_{0}^{x}G(r)\xi (r)e^{\Xi \left( r\right) }dr+\tilde{c%
}\int_{0}^{x}e^{\Xi \left( r\right) }dr\right) e^{-\Xi \left( x\right) }.
\label{w1}
\end{equation}%

Using (\ref{gbp}) and $G(2\pi)=0$
we write (\ref{cc0}) and (\ref{cc})
as follows:
\begin{gather}
\tilde{c}=\frac{\int_{0}^{2\pi }G(r)\xi (r)
e^{\Xi (r) }dr}{%
\int_{0}^{2\pi }e^{\Xi ( r) }dr},\notag\\
C=
\frac{\int_{0}^{2\pi }e^{-\Xi ( x) }
\int_{0}^{x}G(r)\xi
(r)e^{\Xi ( r) }drdx-\tilde{c}\int_{0}^{2\pi }
e^{-\Xi(x)}
\int_{0}^{x}e^{\Xi(r) }drdx-\int_0^{2\pi}G(x)dx}
{\int_{0}^{2\pi }e^{-\Xi
(x) }dx}\,.\notag
\end{gather}

According to the Sobolev embedding theorem, since $s>1/2$
\begin{equation*}
-c(s)\Vert \xi \Vert _{\dot{H}^{s-1}}\leq
\Xi ( r)
\leq c(s)\Vert \xi
\Vert _{\dot{H}^{s-1}},
\quad
\Vert \xi \Vert _{\dot{H}^{s-1}}=
\Vert \varphi
\Vert _{\dot{H}^{s-2}}.
\end{equation*}%
Therefore,
$e^{\vert \Xi( x) \vert }
\leq e^{c(s)\Vert \varphi \Vert
_{\dot{H}^{s-2}}}$, $e^{-\vert \Xi ( x)
\vert }\geq
e^{-c(s)\Vert \varphi \Vert _{\dot{H}^{s-2}}}$ and the
denominators in (\ref{cc0}) and (\ref{cc}) are
bounded away from zero and we obtain
by the Cauchy--Schwarz inequality
(recalling that the norms are normalized)
\begin{equation*}\label{tildecest}
\begin{aligned}
\left\vert \tilde{c}\right\vert \leq \frac{1}{2\pi }
e^{2c(s)\Vert \varphi \Vert_{\dot{H}^{s-2}}}
\int_{0}^{2\pi }\!G(r)\xi (r)dr\leq&\\
e^{2c(s)\Vert \varphi \Vert_{\dot{H}^{s-2}}}
\Vert G\Vert _{\dot{H}^{0}}\Vert \xi
\Vert _{\dot{H}^{0}}=&
e^{2c(s)\Vert \varphi \Vert _{\dot{H}^{s-2}}}\Vert g\Vert _{\dot{H}%
^{-1}}\Vert \varphi \Vert _{\dot{H}^{-1}},
\end{aligned}
\end{equation*}
and using this
we get the following estimate for $C$
\begin{gather}
|C|\leq
\frac{1}{2\pi }
e^{c(s)\Vert \varphi \Vert _{\dot{H}^{s-2}}}
\biggl( \int_{0}^{2\pi }e^{-\Xi \left( x\right)
}\int_{0}^{x}\left\vert G(r) \xi (r)
\right\vert e^{\Xi
\left( r\right) }drdx+\notag\\
\left\vert \tilde{c}\right\vert
\int_{0}^{2\pi}e^{-\Xi \left( x\right) }
\int_{0}^{x}e^{\Xi \left( r\right) }drdx
+2\pi\|G\|_{\dot{H}^0}\biggr)
\leq
\notag
\\
e^{c(s)\|\varphi \|_{\dot{H}^{s-2}}}
\biggl( e^{2c(s)\|\varphi\|_{\dot{H}^{s-2}}}
%\int_{0}^{2\pi}|G(r)||\xi(r)|dr
\|G\|_{\dot{H}^0}\|\xi\|_{\dot{H}^0}
+
|\tilde{c}| 2\pi
e^{2c(s)\|\varphi\|_{\dot{H}^{s-2}}}
+\|G\|_{\dot{H}^0}\biggr) \leq
\notag\\
%e^{c(s)\Vert \varphi \Vert _{\dot{H}^{s-2}}}\left(
%e^{2c(s)\Vert \varphi \Vert _{\dot{H}^{s-2}}}\int_{0}^{2\pi}
%\left\vert
%G(r)\right\vert \left\vert \xi (r)\right\vert dr+
%\left\vert \tilde{c}%
%\right\vert 2\pi e^{2c(s)\Vert \varphi \Vert _{\dot{H}^{s-2}}}
%\right)\le
%\notag
%\\
%2\pi e^{c(s)\Vert \varphi \Vert _{\dot{H}^{s-2}}}
%\left( e^{2c(s)\Vert \varphi
%\Vert _{\dot{H}^{s-2}}}\Vert G\Vert _{\dot{H}^{0}}
%\Vert \xi \Vert _{\dot{H}%
%^{0}}+\left\vert \tilde{c}\right\vert e^{2c(s)\Vert \varphi
%\Vert _{\dot{H}%
%^{s-2}}}\right)
\leq 2(\pi+1)
e^{5c(s)\|\varphi\|_{\dot{H}^{s-2}}}\|g\|_{\dot{H}^{-1}}
(\|\varphi\|_{\dot{H}^{-1}}+1).\notag
\end{gather}%
To estimate the operator norm
we use the  well-known Banach algebra property~(\ref{product}) of
the Sobolev spaces $H^{s}$ for $s>1/2$
(see also Corollary~\ref{C:B2}) and expanding the
exponential into Taylor series we have
\begin{equation*}
\bigr\|e^{\pm \Xi \left( x\right) }\bigr\|_{H^{s}}\leq K_{1}(s)^{-1}\bigl(%
e^{K_{1}(s)\Vert \xi \Vert _{\dot{H}^{s-1}}}+K_{1}(s)-1\bigr)%
=K_{1}(s)^{-1}\left( e^{K_{1}(s)\Vert \varphi \Vert _{\dot{H}%
^{s-2}}}+K_{1}(s)-1\right).
\end{equation*}%
Therefore, from (\ref{w1}) we get \ for $0\leq s_{0}\leq s$
\begin{equation*}
\Vert w\Vert _{\dot{H}^{s_{0}}}\leq \Vert G\Vert _{\dot{H}%
^{s_{0}}}+\max_{0\leq x\leq 2\pi }\left\vert C-\int_{0}^{x}G(r)\xi (r)e^{\Xi
\left( r\right) }dr+\tilde{c}\int_{0}^{x}e^{\Xi \left( r\right)
}dr\right\vert \Vert e^{-\Xi \left( x\right) }\Vert _{\dot{H}^{s}},
\end{equation*}
and from the previous estimates we infer that
\begin{gather}
\Vert v\Vert _{\dot{H}^{s_{0}-1}}=
\Vert w\Vert _{\dot{H}^{s_{0}}}\leq \notag\\
\Vert g\Vert _{\dot{H}^{s_{0}-1}}+(6\pi+2)
e^{5c(s)\Vert \varphi \Vert _{\dot{H}^{s-2}}}
\Vert g\Vert _{\dot{H}^{-1}}
(\Vert \varphi \Vert _{\dot{H}^{-1}}+1)
K_{1}(s)^{-1}\left( e^{K_{1}(s)
\Vert \varphi \Vert _{\dot{H}^{s-2}}}+K_{1}(s)-1\right),
\notag
\end{gather}%
or in a more concise form
\begin{equation}
\| v\|_{\dot{H}^{s_{0}-1}}=\|w\|_{\dot{H}^{s_{0}}}\leq
\|g\|_{\dot{H}^{s_{0}-1}}+\|g\|_{\dot{H}^{-1}}
F_{1}\left( \|\varphi \| _{\dot{H}^{-1}},
 \|\varphi \| _{\dot{H}^{s-2}}\right) ,
\label{w4}
\end{equation}%
for some (explicitly known) function $F_{1}$.
Since $\|g\|_{\dot H^{-1}}\le
\|g\|_{\dot H^{s_0-1}}$
and
$\|f\|_{\dot H^{\theta}}=\|g\|_{\dot H^{\theta}}$, we
obtain~(\ref{B2estofnorm}) with $\theta =s_{0}-1$.
The proof is complete.
\end{proof}

\begin{remark}
\label{R:Z0}
{\rm
 The estimate (\ref{B2estofnorm}) for negative values of  $\theta$
 will be used in Theorem~\ref{T:negllip}. Otherwise,
 for $\theta\ge0$ %(and $s\le 2$)
 a simpler form of this estimate, i.e., $\| L_{\varphi
}^{-1}\|_{\mathcal{L}(\dot{H}^{\theta})}\leq F(\|\varphi
\|_{\dot{H}^{0}})$,
 satisfies all our needs.
}
\end{remark}

\begin{remark}
\label{R:Miura}
{\rm
In terms of $w$ this lemma establishes the
invertibility of the linearization of the Miura transform
$M[w]=w'+w^2$. See \cite{Colliander05}
for the results on the invertibility
of the Miura transform itself.
}
\end{remark}

We now prove global well-posedness in $\dot{H}^s$
for $s>1/2$.

\begin{theorem}
\label{T:loca} Let $s>1/2$, $v(0)=v^{0}\in \dot{H}^{s}$, and let $T
>0$ be  fixed. Then the solution $v=v^\infty$ of~$(\ref{vint})$, in the sense of
Definition~$\ref{D:defmain}$, which was constructed in
Theorem~$\ref{T:gal}$ is unique. Moreover, the solution is of the
class $C([0,T];\dot{H}^s)$, and depends Lipschitz continuously on
the initial data in the sense that is described in~$(\ref{E1})$
below. Furthermore, $
\|v^\infty(t)\|_{\dot{H}^0}=\|v(0)\|_{\dot{H}^0}$ for all $t \in
[0,T]$.
\end{theorem}
\begin{proof}
In view of Corollary~\ref{C:one-two} the solution $v=v^\infty$ of
equation~$(\ref{vint})$,  which was constructed in
Theorem~\ref{T:gal}, also satisfies equation (\ref{defofsol}) for
all $t\in[0,T]$.
 Setting
\begin{equation}
y(t)=v(t)-v^{0},\qquad y(0)=0 , \label{yv}
\end{equation}%
then by Theorem~\ref{T:gal} $y\in L_{\infty}([0,T];\dot{H}^s)$.
Using the symmetry of $B_{2}$, $B_{2}(u,v)=B_{2}(v,u)$,
 we have from~(\ref{defofsol})
\begin{equation} \label{soly}
y(t)-\frac{1}{3}B_{2}(v^{0},y(t))=
\frac{1}{6}B_{2}(y(t),y(t))+\frac{i}{6}
\int_{0}^{t}R_{3}(\left( y(\tau )+v^{0}\right) ^{3})d\tau.
\end{equation}%
Setting $ L_{v^{0}}\,y=y-\frac{1}{3}B_{2}(v^{0},y)$, then by virtue
of Lemma~\ref{L:linB2}, Lemma~\ref{L:B2} and Lemma~\ref{L:R3} we
have
\begin{equation} \label{invinteq}
y(t)=L_{v^{0}}^{-1}\left( t\right) \biggl(\frac{1}{6}B_{2}(y(t),y(t))+\frac{i%
}{6}\int_{0}^{t}R_{3}(\left( y(\tau )+v^{0}\right) ^{3})d\tau \biggr)=:%
\mathcal{F}(y)(t).
\end{equation}
Let $T^*\in (0,T]$, to be determined later. We consider the Banach
space $C([0,T^{\ast }];\dot{H}^{s})$ and the subspace
$C_{0}([0,T^*];\dot{H}^{s})= \{y\in C([0,T^*];\dot{H}^{s}):\
y(0)=0\}$. Next, we show that the nonlinear operator $\mathcal{F}$
maps the ball of radius $A$, that is, $\{y\in
C_0([0,T^*];\dot{H}^{s}):\ \|y\|_{C([0,T^*];\dot{H}^{s})}\le A\}$,
into itself and is a contraction map provided that $A$ and $T^*$ are
small enough. In fact, for $A$ and $T^*$ small enough we obtain by
Lemma~\ref{L:linB2}, Lemma~\ref{L:B2} and Lemma~\ref{L:R3} that
\begin{equation}\label{contr1}
\aligned \|\mathcal{F}(y)(t)\|_{\dot H^s}&\le c_7(s)
\|L_{v^0}^{-1}\|_{\mathcal{L}(\dot H^s)}
\left(A^2+T^*(A+\|v^0\|_{\dot H^s})^3\right)\le \frac A4,
\\
\|\mathcal{F}(y_1)(t)-\mathcal{F}(y_2)(t)\|_{\dot H^s}&\le c_7(s)
\|L_{v^0}^{-1}\|_{\mathcal{L}(\dot H^s)} \left(A+T^*(A+\|v^0\|_{\dot
H^s})^2\right)
\|y_1-y_2\|_{C([0,T^*];\dot H^s)}\le\\
&\le\frac12\|y_1-y_2\|_{C([0,T^*];\dot H^s)},
\endaligned
\end{equation}
where $c_{7}(s)=(5/3)\max (c_{2}(s),c_{6}(s))$.
We now fix  $A$ and $T^*$ small enough such that the above
inequality holds.
(Note that
$A$ depends only on~$\|v^{0}\| _{\dot{H}^{s}}$ and
therefore
$T^{\ast }$ also depends only on
 $\|v^{0}\| _{\dot{H}^{s}}$.)
 Hence by the
Banach Contraction Principle there exists a unique solution
$y(t)$ of~(\ref{invinteq}) on the
interval $[0,T^{\ast }]$.

Denote by $[0,T_{\max}^*)$ the maximal interval of existence of the
solutions of~(\ref{invinteq}). By short time existence so obtained
we have $T^*_{\max}>0$. If  $T^*_{\max} > T$ we are done with the
proof. However, if $T^*_{\max}\le T$ then  by standard arguments one
can show that the $\limsup_{t\to
T^*_{\max}-0}\|v(t)\|_{\dot{H}^s}=\infty$. Based on
Theorem~\ref{T:gal} and the uniqueness of solutions established
above we have $v^\infty(t)=v(t)=y(t)+v_0$ for all
$t\in[0,T^*_{\max})$. By virtue of~(\ref{vinftynorm}) $\limsup_{t\to
T^*_{\max}-0}\|v(t)\|_{\dot{H}^s}<\infty$. Consequently, $T^*_{\max}
> T$. Hence $v=v^\infty\in C([0,T],\dot{H}^s)$ is unique on $[0,T]$.
In particular, thanks to~(\ref{Energy-infinity}) and the above we
have $\|v^\infty(t)\|_{\dot{H}^0}=\|v(0)\|_{\dot{H}^0}$, for all
$t\in [0,T]$, i.e.,  $v^\infty(t)$ conserves the energy for all
$t\in[0,T]$.

To prove the continuous dependence on the initial data (in fact,
Lipschitz continuous) we consider two solutions $w(t)$ and
$v(t)$
evolving from two close initial points $w(0)=w^{0}$ and
$v(0)=v^{0}$. Then $z(t)=w(t)-w^{0}$ and $y(t)=v(t)-v^{0}$
satisfy on $[0,T^*]$, where $T^*$ to be chosen later,
 the equations
\begin{gather*}
\aligned
z(t)=\frac{1}{6}B_{2}(z(t),z(t))+
\frac{1}{3}B_{2}(w^{0},z(t))+
\frac{i}{6}\int_{0}^{t}{R}_{3}((z(\tau )+w^{0})^3)d\tau,
\\
y(t)=\frac{1}{6}B_{2}(y(t),y(t))+\frac{1}{3}B_{2}(v^{0},y(t))+\frac{i}{6}%
\int_{0}^{t}{R}_{3}((y(\tau )+v^{0})^3)d\tau.
\endaligned
\end{gather*}%
Therefore $\varphi (t)=z(t)-y(t)$ satisfies
\begin{gather*}
\aligned\varphi (t)-\frac{1}{3}B_{2}(w^{0},\varphi (t))=
\frac{1}{6}B_{2}(z(t)+y(t),\varphi (t))+\frac{1}{3}
B_{2}(w^{0}-v^{0},y(t))+
\\
\frac{i}{6}\int_{0}^{t}({R}_{3}((z(\tau )+w^{0})^3)-
{R}_{3}((y(\tau )+v^{0})^3)d\tau.
\endaligned
\end{gather*}
Inverting the operator $L_{w^{0}}$,
$L_{w^{0}}\varphi (t)=\varphi (t)-
\frac{1}{3}B_{2}(w^{0},\varphi (t))$,
arguing as above and using the estimate~(\ref{vinftynorm})
for $v$ and $w$, we obtain on
$[0,T^*]$
\begin{gather*}
\|\varphi (t)\|_{\dot{H}^{s}}\le
c_{7}(s)\|L_{w^{0}}^{-1}\|_{\mathcal{L}(\dot{H}^s)} \bigl(A\Vert
\varphi (t)\Vert _{\dot{H}^{s}}+ A\Vert w^{0}-v^{0}\Vert
_{\dot{H}^{s}}+ M_s^2
\int_{0}^{t}\Vert \varphi (\tau )\Vert _{%
\dot{H}^{s}}d\tau \bigr),
\end{gather*}
where $A=\max(\|y\|_{C([0,T^*];\dot{H}^{s})},
\|z\|_{C([0,T^*];\dot{H}^{s})})$
and $M_s=
M_s(T,\|v^{0}\|_{\dot{H}^{s}}+\|w^{0}\|_{\dot{H}^{s}})$
is as in~~(\ref{vinftynorm}). Hence
\begin{gather*}
\| \varphi\|_{C([0,T^*];\dot{H}^{s})}\le
 c_{7}(s)\|L_{w^{0}}^{-1}\|_{\mathcal{L}(\dot{H}^s)}\bigl(A+T^*
M_s^2\bigr) \| \varphi\|_{C([0,T^*];\dot{H}^{s})} +C(s,A)
\|w^{0}-v^{0}\|_{\dot{H}^{s}}.
\end{gather*}
Notice that since $w(t)$ conserves energy then by Remark~\ref{R:Z0}
we have $\|L_{w(t)}^{-1}\|_{\mathcal{L}(\dot{H}^s)}\le
F(\|w^0\|_{\dot{H}^{0}})$, for all $t\in [0,T]$.
 Taking $A$ and $T^*$  small enough (both are
depending now only on $\|w^0\|_{\dot{H}^{s}}+\|w^0\|_{\dot{H}^{s}}$
and $T$) such  that
$$
c_{7}(s)\|L_{w^{0}}^{-1}\|_{\mathcal{L}(\dot{H}^s)}\bigl(A+T^*
M_s^2\bigr)
 \le c_{7}(s) F(\|w^0\|_{\dot{H}^{0}}) \bigl(A+T^*
M_s^2\bigr)\le \frac12
 $$ we obtain
\begin{equation*}
\| \varphi\|_{C([0,T];\dot{H}^{s})}\le
2C(s,A)
\|w^{0}-v^{0}\|_{\dot{H}^{s}}=
C'(s,T,\|v^0\|_{\dot{H}^{s}},\|w^0\|_{\dot{H}^{s}})
\|w^{0}-v^{0}\|_{\dot{H}^{s}},
\end{equation*}
which gives for
$w(t)-v(t)=\varphi(t)-(w^0-v^0)$
\begin{equation*}
\|w(t)-v(t)\|_{\dot{H}^{s}}\leq
\bigl(C'(s,T,\|v^0\|_{\dot{H}^{s}},\|w^0\|_{\dot{H}^{s}})+1
\bigr)
\|w^{0}-v^{0}\|_{\dot{H}^{s}},
\quad t\in[0,T^*].
\end{equation*}%
 Therefore, after $N$  steps, where $N=[T/T^*]+1$,
we obtain the Lipschitz estimate
\begin{equation} \label{E1}
\|w(t)-v(t)\|_{\dot{H}^{s}}\leq
\bigl(C'(s,T,\|v^0\|_{\dot{H}^{s}}+
\|w^0\|_{\dot{H}^{s}})+1\bigr)^N
\|w^{0}-v^{0}\|_{\dot{H}^{s}},
\quad t\in \lbrack 0,T].
\end{equation}
\end{proof}

\begin{remark}
\label{R:HstoH0} {\rm Generally speaking  the Lipschitz constant
in~(\ref{E1}) may grow with respect to $T$ at a rate higher than
exponential. In section~\ref{Sec:nregin}, and  for $s \in [0,1/2)$,
the Lipschitz estimate~(\ref{E1}) will be proved in a stronger
 form
 \begin{equation} \label{E1111}
\|w(t)-v(t)\|_{\dot{H}^{s}}\leq
\bigl(C(s,\|v^0\|_{\dot{H}^{0}}+
\|w^0\|_{\dot{H}^{0}})+1\bigr)^T
\|w^{0}-v^{0}\|_{\dot{H}^{s}},
\quad t\in \lbrack 0,T].
\end{equation}
}
\end{remark}

\setcounter{equation}{0}

\section{Uniqueness of solutions with non-regular
initial data ($0\le s\le1/2$)
 and Lipschitz dependence in weaker norms}
\label{Sec:nregin}

Here we will show the uniqueness and Lipschitz continuous dependence
on the initial data for the class of solutions of~(\ref{v}) in the
sense of Definition~\ref{D:defmain} with initial data
$v^0\in\dot{H}^s$ for $s\in[0,1/2]$. The existence of such
solutions, for any $s>0$, has been established in
Theorem~\ref{T:gal}; and we observe that so far we have not proved
the {\it existence} of solutions with initial data in $\dot{H}^{0}$.
The existence and Lipschitz continuity of such solutions, when
$s=0$, will be proved in Theorem~\ref{T:globs=0}
 at the end of this section.
We also study the Lipschitz dependence on the initial
data in the norm of $\dot{H}^{\theta }$, for $\theta >-1$ of
the solutions of~(\ref{v}) with initial data bounded in
$\dot{H}^{0}$.

The main role in the proof is played by {\it time averaging induced
squeezing}, which is described later in this section. First, we give
a sketch of the subsequent treatment. Our strategy to prove
uniqueness and the Lipschitz continuous dependence of $v\in
L_{\infty }([ 0,T],\dot{H}^{\theta })$, with $v(0)\in
\dot{H}^{\theta }$, for $\theta \ge 0$,  is as follows. Let $v$, $w$
be two solutions with initial data in $\dot{H}^{\theta }$ with
$v(0)=w(0)$;  and let $[0,T_{1}]$ be the maximal interval on which
they coincide (thanks to Remark \ref{R:regularity} such a maximal
interval is closed; also it is possible that $T_1=0$). If $T_{1}<T$
we may take $v(T_{1})=w(T_{1})=v^{0}$ as a new initial data and
consider~(\ref{v}) on $[T_{1},T_{1}+\tau]$ with a small $\tau$.
Similarly to the proof of Theorem~\ref{T:loca} we want to transform
the problem to an equation in
$L_{\infty}([T_{1},T_{1}+\tau],\dot{H}^{\theta })$ for $y(t)$, where
$v(t)=v^{0}+y(t)$,  of the form
\begin{equation}\label{Ftaun}
y=\mathcal{F}_{\tau ,n}(y,v^{0}),
\end{equation}
where $\mathcal{F}_{\tau ,n}$ is a Lipschitz map in $L_{\infty}([
T_{1},T_{1}+\tau],\dot{H}^{\theta })$ with a  Lipschitz constant
less than one. Accordingly, equation~(\ref{Ftaun}) will have a
unique small solution $y$ in $L_{\infty}([
T_{1},T_{1}+\tau],\dot{H}^{\theta })$. The parameter $n$ describes
the construction of the operator $\mathcal{F}_{\tau,n}$, which
involves the splitting of the Fourier modes  $y_{k}$ of the solution
$y$ into high modes (with $|k|>n$) and low modes (with $|k|\le n$).
The Lipschitz  estimate for $y(t)=v(t)-v^0$ and $z(t)=w(t)-w^0$ will
have the form
\begin{equation} \label{Lipf}
\aligned
\|\mathcal{F}_{\tau,n}(y,v^{0})-
\mathcal{F}_{\tau,n}(z,w^{0})\|_{L_{\infty }
([ T_{1},T_{1}+\tau],\dot{H}^{\theta })}&\leq
\\
 C( C_{0})(F_{1}( n) +\tau F_{2}( n))
\biggl(\|y-z\|_{L_{\infty}
([ T_{1},T_{1}+\tau],\dot{H}^{\theta})}+&
\|v^{0}-w^{0}\| _{\dot{H}^{\theta }}\biggr),
\endaligned
\end{equation}%
where $C( C_{0}) $ is bounded provided that the solutions are
uniformly bounded, over the interval  $[T_1,T_1+\tau]$, in the
$\dot{H}^{\theta_0}$ norm, for some $\theta_0 \ge 0$ (we show below
that we can take $\theta_0=0$ when $\theta \in (0,1/2)$)
\begin{equation}\label{C(C0)}
\|v\|_{L_{\infty }([ T_{1},T_{1}+\tau],\dot{H}^{\theta_0})}+
\|w\|_{L_{\infty}([T_{1},T_{1}+\tau],\dot{H}^{\theta_0})}\le
 C_{0}.
\end{equation}
 Moreover, the time-independent part of the Lipschitz
estimate will enjoy the property
\begin{equation}\label{F1n}
F_{1}\left( n\right) \to 0\text{ \ as \ }n\to\infty.
\end{equation}%
Obviously, if we have two solutions $y$, $z$
in $L_{\infty }([0,T],\dot{H}^{\theta_0})$, we can take
\begin{equation}\label{C_0}
C_{0}=C_0(T)=\|v\|_{L_{\infty }([ 0,T],\dot{H}^{\theta_0})}+
\|w\|_{L_{\infty}([0,T],\dot{H}^{\theta_0})}
<\infty.
\end{equation}%
Based on this property, we will  first choose $n$ large enough  such
that
\begin{equation*}
C( C_{0}) F_{1}(n) \leq 1/4,
\end{equation*}%
and then choose $\tau $ small enough so that
\begin{equation*}
\tau \leq \frac{1}{4C( C_{0}) F_{2}(n)}\,.
\end{equation*}%
Together, the above   implies
\begin{equation}\label{ftnl}
\|\mathcal{F}_{\tau ,n}(y)-\mathcal{F}_{\tau,n}(z)\|_
{L_{\infty}([ T_{1},T_{1}+\tau ],\dot{H}^{\theta })}\le
\frac12\|y-z\|_{L_{\infty}([T_{1},T_{1}+\tau],\dot{H}^{\theta})}
+C\|v^{0}-w^{0}\|_{\dot{H}^{\theta }}.
\end{equation}%
Therefore the two solutions of~(\ref{Ftaun}) satisfy the
Lipschitz estimate
\begin{equation*}
\|y-z\|_
{L_{\infty}([T_{1},T_{1}+\tau],\dot{H}^{\theta })}\leq
2C\|v^{0}-w^{0}\|_{\dot{H}^{\theta }},
\end{equation*}
for some $C=C(T)$,
where $\tau$ depends on the norm of $y,z$
in $L_{\infty}([0,T],\dot{H}^{\theta_0})$, and hence for
$v(t)=v^{0}+y(t)$ and $w(t)=w^{0}+z(t)$ this gives
\begin{equation*}
\|v-w\|_{L_{\infty}([T_{1},T_{1}+\tau],\dot{H}^{\theta })}
\le
(2C(T)+1)\|v^{0}-w^{0}\|_{\dot{H}^{\theta }}.
\end{equation*}
If (\ref{C_0}) holds, we can iterate the
estimate and obtain
\begin{equation} \label{ylip}
\|v-w\|_{L_{\infty }([0,T],\dot{H}^{\theta })}\le
C^{\prime }
(2C(T)+1)^{[T/\tau]+1}
\|v^0-w^0\|_{\dot{H}^{\theta }}.
\end{equation}%
Since the $\dot{H}^{0}$-norm estimates of the solutions constructed
in Theorem~\ref{T:gal} are uniform in $T$, it follows that in the
case when we can take $\theta_0=0$, the above Lipschitz continuity
estimate can be written in the usual form (see also
Remark~\ref{R:HstoH0})
\begin{equation} \label{ylip0}
\|v-w\|_{L_{\infty }([0,T],\dot{H}^{\theta })}\le
C^{\prime\prime }
(2C+1)^{T}
\|v^0-w^0\|_{\dot{H}^{\theta }}.
\end{equation}%

Note that two terms in (\ref{Lipf}) control the size of the
Lipschitz constant. The first is the Picard short time factor
$\tau$, which ensures solvability of ordinary differential equations
locally in time. The second term $F_{1}(n)$ is small thanks to the
time averaging induced squeezing which is crucial for the continuous
dependence we prove here, and which will be described in detail
below.

For the discussion below the reader is referred to the relevant
estimates in section~\ref{S:App}. Recall that by
Corollary~\ref{C:one-two} every solution $v(t)$ of~(\ref{v}) in the
sense of Definition~\ref{D:defmain} with $v^0\in\dot{H}^s$, $s>0$
also satisfies~(\ref{defofsol}). Therefore $y(t)=v(t)-v^0$
satisfies~(\ref{soly}) which we would like to  transform
to~(\ref{invinteq}). Unfortunately, the operator
$\int_{0}^{t}{R}_{3}dt^{\prime }$ in (\ref{soly}) is Lipschitz in a
rather narrow space $\dot{H}^{s}$, $s>1/2$, which is out of the
range of our interest, namely, $s\in[0,1/2]$. Therefore we want to
use an equation similar to~(\ref{invinteq}), but  defined in a wider
space. A possible candidate is  equation~(\ref{fin2}) with the
integrated form~(\ref{fin3}). The right-hand side of~(\ref{fin3})
has a  small Lipschitz constant in $L_{\infty }([
T_{1},T_{1}+\tau],\dot{H}^{\theta })$ if $\tau $ is small. But this
equation is difficult to use directly for the proof of uniqueness,
since the linearization of the left-hand side about $v=v^{0}$  may
be not invertible. But if we try to invert the linearization of the
operator $I-B_2$ about $v^0$ then for $y(t)$, where
$v(t)=y(t)+v^{0}$ with $y(0)=0$, we obtain from~(\ref{fin3}) the
equation
\begin{equation*}\label{yvnnn}
\aligned
 y(t)-\frac{1}{3}B_{2}(v^{0},y(t))=
 \frac{1}{6}B_{2}(y(t),y(t))
+\frac{1}{18}\biggl(
B_3\left((y(t)+v^{0})^{3}\right)-
B_3\left((v^{0})^{3}\right)
\biggr)
+
\\
+
\frac{i}{6}\int_{0}^{t}\biggl( A_{\text{res}}
(( y(t')+v^{0})^{3})
+\frac{1}{3}B_4(( y(t')+v^{0})^{4})
 \biggr)dt'.
\endaligned
\end{equation*}
Here we can invert $I-\frac{1}{3}B_{2}(v^{0},y)$ on
the left-hand side with norm of the inverse operator
 depending on
$\|v^0\|_{\dot{H}^{\theta_0}}$, where $\theta_0\ge-1$,
see Lemma~\ref{L:linB2}, but the Lipschitz
constant in $L_{\infty }([ T_{1},T_{1}+\tau],\dot{H}^{\theta })$ of
the nonlinear operator so obtained is not small because of the
second term on the right-hand side (notice that the Lipschitz
constant of the first term is small since it is quadratic and $y$ is
small, the third term has a  small Lipschitz constant because of a
short time interval in the integral).

Therefore we derive a modified version of~(\ref{fin3})
(in fact, a family of equations parameterized
by $n\in\mathbb{N}$) with
the  same invertible operator on left-hand side and
the second term (corresponding to $B_3$) having
a small Lipschitz constant as
$n\to\infty$. To this end we
use a proper splitting of a solution into high and low
Fourier modes using the projection $\Pi _{n}$
defined by (\ref{Pm}), where we choose  $n$ later
to be large enough depending on
$\|v^{0}\|_{\dot{H}^{\theta_0}}$.

Now we present a detailed discussion and proofs. We need to
introduce some notation.
We use projections $\Pi_{n}$  with integer $n\geq 0$ defined
in~(\ref{Pm})  and we also set
\begin{equation*}
\Pi _{-n}=I-\Pi _{n}.
\end{equation*}%
Obviously,
\begin{equation}\label{Pz}
u=\Pi _{n}u+\Pi _{-n}u=\sum_{\zeta =\pm 1}\Pi _{\zeta n}u.
\end{equation}%
We can rewrite
\begin{equation}  \label{R3sum}
R_{3\text{nres}}(u,v,w)_{k}=
\sum_{\vec{\zeta}\in\{ -1,1\}^3}R_{3\text{nres}}
(\Pi _{\vec{\zeta}}(u,v,w))_{k},
\end{equation}
where
\begin{equation*}
\Pi _{\vec{\zeta}}(u,v,w)=( \Pi _{\zeta _{1}n}u,
\Pi _{\zeta _{2}n}v,\Pi_{\zeta _{3}n}w),\ \
\vec{\zeta}=( \zeta _{1},\zeta _{2},\zeta_{3}) \in
\{ -1,1\}^{3},\ \ \zeta _{j}=\pm 1,
\end{equation*}%
and  where (see~(\ref{R3res}))
\begin{equation*}
R_{3\text{nres}}(u,v,w)_{k}=
\sum_{k_{1}+k_{2}+k_{3}=k}^{\text{nonres}}
\frac
{e^{i3(k_{1}+k_{2})(k_{2}+k_{3})(k_{3}+k_{1})t}}
{k_{1}}u_{k_{1}}v_{k_{2}}w_{k_{3}}.
\end{equation*}%
Note that in~(\ref{R3sum}) the terms with $\zeta _{j}=+1$ smooth
(filter) the $j-$th
argument, also the first argument is always smoothed by the
factor ${1}/{k_{1}}$. Therefore the only term which has only
one smoothed argument (factor) corresponds to
$\vec{\zeta}=( \pm 1,-1,-1)$, when the already smoothed
first factor is multiplied by $\Pi _{\pm n}$ and the
remaining two are multiplied by $\Pi_{-n}$. The remaining
6
terms in~(\ref{R3sum}) have at least two smoothed
factors. Hence, we have for $R_{3\text{nres}}(u,v,w)$
\begin{equation}\label{R3sum1}
R_{3\text{nres}}(u,v,w)=R_{3\text{nres}0}^{(n)}(u,v,w)+
R_{3\text{nres}1}^{(n)}(u,v,w)
\end{equation}
with
\begin{equation}\label{nR1}
\aligned
&R_{3\text{nres}0}^{(n)}(u,v,w)=R_{3\text{nres}}
(\Pi _{\vec{\zeta}_{0}}(u,v,w))+R_{3\text{nres}}
(\Pi _{\vec{\zeta}_{1}}(u,v,w)),\\
&R_{3\text{nres}1}^{(n)}(u,v,w)=\sum_{\vec{\zeta}\in
\{ -1,1\} ^{3},\vec{\zeta}\neq \vec{\zeta}_{0},
\vec{\zeta}\neq \vec{\zeta}_{1}}R_{3\text{nres}}
(\Pi _{\vec{\zeta}}(u,v,w)),
\endaligned
\end{equation}
where
$\vec{\zeta}_{0}=( +1,-1,-1)$,
$\vec{\zeta}_{1}=( -1,-1,-1)$.
Since $\Pi _n+\Pi _{-n}=I$, we obviously have
\begin{equation} \label{R3nr0}
R_{3\text{nres}0}^{(n)}(u,v,w)_k=
\sum_{k_{1}+k_{2}+k_{3}=k}^{\text{nonres}}
\frac
{e^{i3(k_{1}+k_{2})(k_{2}+k_{3})(k_{3}+k_{1})t}}
{k_{1}}u_{k_{1}}\Pi_{-n}v_{k_{2}}\Pi _{-n}w_{k_{3}}.
\end{equation}
The operator $R_{3\text{nres}1}$ has better boundedness properties
than $R_{3}$, as it is shown in the following lemma. However, the
corresponding constant is increasing, as $n\to\infty$, and will play
the role of the constant $F_2(n)$ in~(\ref{Lipf}).

\begin{lemma}
\label{L:R31} Let $0\leq s\leq 1$, $\alpha \ge 0$. Then the
 operator $R_{3\mathrm{nres}1}^{(n)}$ in
 $(\ref{nR1})$ satisfies the estimate
\begin{equation}  \label{nr3s}
\| R_{3\mathrm{nres}1}^{(n)}(u,v,w)\|_{\dot{H}^{s}}\leq
c_{4}n^{s+1+\alpha }\|u\|_{\dot{H}^{0}}
\|v\|_{\dot{H}^{-\alpha }}
\|w\|_{\dot{H}^{0}}+c_{4}n^{1+\alpha}\|u\|_{\dot{H}^{0}}
\|v\|_{\dot{H}^{-\alpha }}
\|w\|_{\dot{H}^{s}}.
\end{equation}
\end{lemma}

\begin{proof}
We consider one of the terms in the second formula in (\ref{nR1}),
namely, the sum of  terms with $\vec{\zeta}=(-1,+1,-1)$ and
$\vec{\zeta}=(1,+1,-1)$. We set
\begin{equation*}
\tilde{R}^{(n)}_{3\text{nres}1}(u,v,w)_{k}:=
\sum_{k_{1}+k_{2}+k_{3}=k}^{\mathrm{nonres}}
\frac{e^{i3(k_{1}+k_{2})(k_{2}+k_{3})
(k_{3}+k_{1})t}}{k_{1}}u_{k_{1}}\Pi_{n}v_{k_{2}} \Pi _{-n}w_{k_{3}};
\end{equation*}%
the four remaining terms are estimated in exactly the same
way.  We use duality
\begin{gather*}
|(\tilde{R}^{(n) }_{3\mathrm{nres1}}(u,v,w),z)|\leq
\sum_{k\in\mathbb{Z}_0} \sum_{k_{1}+k_{2}+k_{3}=k}^{\text{nonres}}
\frac{|u_{k_{1}}||\Pi_{n}v_{k_{2}}|| \Pi _{-n}w_{k_{3}}||z_{k}|}
{|k_{1}|}\le
\\
\sum_{k\in\mathbb{Z}_0}\sum_{k_{1}+k_{2}+k_{3}=k}\frac{|u_{k_{1}}||
\Pi_{n}v_{k_{2}}||\Pi _{-n}w_{k_{3}}| |z_{k}|}{|k_{1}|}=
\sum_{k_{1},k_{2},k_{3}}\frac{|u_{k_{1}}|
|\Pi_{n}v_{k_{2}}||\Pi _{-n}w_{k_{3}}|
|z_{k_{1}+k_{2}+k_{3}}|}{|k_{1}|}\leq
\\
 \sum_{| k_{1}|>0}\sum_{0<| k_{2}| \leq n}
 \sum_{|k_{3}|>n}\frac{|u_{k_{1}}||v_{k_{2}}|
|w_{k_{3}}| |z_{k_{1}+k_{2}+k_{3}}|}{|k_{1}|}\,.
\end{gather*}
For $0\leq s\leq 1$ we set $\tilde{z}_{k}=|z_{k}|/|k|^{s}$
and obtain
\begin{gather*}
|(\tilde{R}^{(n) }_{3\text{nres}1}(u,v,w),z)|\leq \sum_{|k_{1}|>0}
\sum_{0<|k_{2}|\leq n}\sum_{|k_{3}|>n} \frac{|u_{k_{1}}||v_{k_{2}}|
|w_{k_{3}}||k_{1}+k_{2}+k_{3}|^{s}|\tilde{z}_{k}|} {|k_{1}|}\le
\\
\sup_{0<|k_{2}|\leq n}|v_{k_{2}}|\sum_{0<|k_{2}|\leq
n}\sum_{|k_{1}|>0}\sum_{|k_{3}|>n}%
\frac{|u_{k_{1}}||w_{k_{3}}|( |k_{1}|^{s}+|k_{2}|^{s}+
|k_{3}|^{s})
\tilde{z}_{k_{1}+k_{2}+k_{3}}}{|k_{1}|}\leq
\\
\|v\|_{\dot{H}^{-\alpha }}n^{\alpha }
\sum_{0<|k_{2}|\leq n}\sum_{|k_{1}|>0}
\sum_{|k_{3}|>n}\frac{|u_{k_{1}}||w_{k_{3}}|(|k_{1}|^{s}+
|k_{2}|^{s}+|k_{3}|^{s})
\tilde{z}_{k_1+k_2+k_3}}{|k_{1}|}.  \notag
\end{gather*}%
We estimate the sum in $k_{1}$ and $k_{3}$ as follows:
\begin{gather*}
\sum_{|k_{1}|>0}\sum_{|k_{3}|>n}
\frac{|u_{k_{1}}||w_{k_{3}}|(|k_{1}|^{s}+|k_{2}|^{s}+
|k_{3}|^{s}) \tilde{z}_{k_{1}+k_{2}+k_{3}}}{|k_{1}|}\leq
 \\
n^{s}\sum_{|k_{1}|>0}\sum_{|k_{3}|>n}
\frac{|u_{k_{1}}||w_{k_{3}}|\tilde{z}_{k_{1}+k_{2}+k_{3}}}
{|k_{1}|}+
\sum_{|k_{1}|>0}\sum_{|k_{3}| >n}
\frac{|u_{k_{1}}||w_{k_{3}}||k_{1}|^{s}
\tilde{z}_{k_{1}+k_{2}+k_{3}}}{|k_{1}|}+
\\
\sum_{|k_{1}|>0}\sum_{| k_{3}| >n}\frac{|u_{k_{1}}|
|w_{k_{3}}|| k_{3}|^{s}
\tilde{z}_{k_{1}+k_{2}+k_{3}}}{|k_{1}|} \le
\\
n^{s}\left( \sum_{k_{1}}\frac{1}{|k_{1}|^{2}}
\sum_{k_{3}}\tilde{z}_{k_{1}+k_{2}+k_{3}}^{2}\right)^{1/2}
\left(
\sum_{k_1}\sum_{k_3}|u_{k_1}|^{2}|w_{k_3}|^2\right)^{1/2}+
\\
2\left( \sum_{k_{1}}\frac{1}{| k_{1}|^{2}}%
\sum_{k_{3}}\tilde{z}_{k_{1}+k_{2}+k_{3}}^{2}\right)^{1/2}
\left(
\sum_{k_{1}}\sum_{k_{3}}|w_{k_{3}}|^{2}|k_{3}|^{2s}
|v_{k_{1}}|^{2}\right) ^{1/2} \le
\\
c_{3}n^{s}\|v\|_{\dot{H}^{0}}\|w\|_{\dot{H}^{0}}
\|z\|_{\dot{H}^{-s}}+2c_{3}\|v\|_{\dot{H}^{0}}
\|w\|_{\dot{H}^{s}}\|z\|_{\dot{H}^{-s}},
\end{gather*}
where $c_3\le\pi/\sqrt{3}$.
Hence, after a finite summation in $k_{2}$ we obtain
\begin{equation*}
\aligned |(\tilde{R}^{(n)}_{3\text{nres}1}(u,v,w),z)|\leq&
\\
 c_{4}n^{s+1+\alpha}\|u\|_{\dot{H}^{0}}&
\|v\|_{\dot{H}^{-\alpha }}
\|w\|_{\dot{H}^{0}}\|z\|_{\dot{H}^{-s}}+c_{4}n^{1+\alpha}
\|u\|_{\dot{H}^{0}}\|v\|_{\dot{H}^{-\alpha }}
\|w\|_{\dot{H}^{s}}\|z\|_{\dot{H}^{-s}}
\endaligned
\end{equation*}
and (\ref{nr3s}) is proven.
\end{proof}

We now consider the operator $R_{3\text{nres}0}^{(n)}$ defined in
 (\ref{R3nr0}):
\begin{equation*}
R_{3\text{nres}0}^{(n)}(v^{3})_k=
\sum_{k_{1}+k_{2}+k_{3}=k}^{\text{nonres}}
\frac
{e^{i3(k_{1}+k_{2})(k_{2}+k_{3})(k_{3}+k_{1})t}}
{k_{1}}v_{k_{1}}\Pi_{-n}v_{k_{2}}\Pi _{-n}v_{k_{3}}.
\end{equation*}
We make the transformation of the form
(\ref{second})--(\ref{B42}) (second differentiation
 by parts in time) applied to
$R_{3\text{nres}0}^{(n)}$
instead of $R_{3\text{nres}}$. Namely,
transformation  (\ref{second}) is replaced by
\begin{gather*}
R_{3\text{nres}0}^{(n)}(v^{3})_{k}=
\frac{1}{3i}\partial _{t}
B_{30}^{(n)}(v,v,v)_{k}-
\frac{1}{3i}
\sum_{k_{1}+k_{2}+k_{3}=k}^{\text{nonres}}
\frac{e^{i3(k_{1}+k_{2})(k_{2}+k_{3})(k_{3}+k_{1})t}}
{k_{1}(k_{1}+k_{2})(k_{2}+k_{3})(k_{3}+k_{1})}\times\\
\bigl(\partial
_{t}v_{k_{1}}\Pi _{-n}v_{k_{2}}\Pi
_{-n}v_{k_{3}}+v_{k_{1}}\Pi _{-n}\partial _{t}v_{k_{2}}\Pi
_{-n}v_{k_{3}}+v_{k_{1}}\Pi _{-n}v_{k_{2}}\Pi _{-n}\partial
_{t}v_{k_{3}}\bigr).
\end{gather*}%
Similarly to~(\ref{B3}), the operator $B_{30}^{(n)}$ is
given  by
\begin{equation}\label{B30}
B_{30}^{(n)}(v,v,v)_{k}=
\sum_{k_{1}+k_{2}+k_{3}=k}^{\text{nonres}}
\frac{e^{i3(k_{1}+k_{2})(k_{2}+k_{3})(k_{3}+k_{1})t}}
{k_{1}(k_{1}+k_{2})(k_{2}+k_{3})(k_{3}+k_{1})}v_{k_{1}}
\Pi _{-n}v_{k_{2}}\Pi_{-n}v_{k_{3}}.
\end{equation}%
We obtain similarly to (\ref{R3nres})
\begin{equation}\label{R3ns0}
R_{3\text{nres}0}^{(n)}(v^3)_k=
\frac{1}{3i}\partial
_{t}B_{30}^{(n)}(v^3)_{k}-\frac{1}{3}
B_{40}^{(n)}(v^4)_k,
\end{equation}
with
$$
B_{40}^{(n)}(v^4)_k= \frac{1}{2}B_{40}^{1}(v^4)_{k}+
B_{40}^{2}(v^4)_{k},
$$
where  $B_{40}^{1}=B_{40}^{1(n)}$,
$B_{40}^{2}=B_{40}^{2(n)}$ are defined similarly to
(\ref{B41}), (\ref{B42}):
\begin{equation} \label{B41n}
B_{40}^{1}(v^{4})_k=\sum_{k_{1}+k_{2}+k_{3}+k_{4}=k}^{\text{nonres}}\frac{%
e^{i\Phi (k_{1},k_{2},k_{3},k_{4})t}}{%
(k_{1}+k_{2})(k_{1}+k_{3}+k_{4})(k_{2}+k_{3}+k_{4})}\Pi _{-n}v_{k_{1}}\Pi
_{-n}v_{k_{2}}v_{k_{3}}v_{k_{4}},
\end{equation}%
\begin{equation} \label{B42n}
B_{40}^{2}(v^{4})_k=\sum_{k_{1}+k_{2}+k_{3}+k_{4}=k}^{\text{nonres}}\frac{%
e^{i\Phi (k_{1},k_{2},k_{3},k_{4})t}\ (k_{3}+k_{4})}{%
k_{1}(k_{1}+k_{2})(k_{1}+k_{3}+k_{4})(k_{2}+k_{3}+k_{4})}v_{k_{1}}\Pi
_{-n}v_{k_{2}}\Pi _{-n}\left( v_{k_{3}}v_{k_{4}}\right).
\end{equation}
In the last formula the operator $\Pi _{-n}( v_{k_{3}}v_{k_{4}})$ is
defined as follows:
\begin{equation*}\label{Pi(vv)}
 \Pi_{-n}( v_{k_{3}}v_{k_{4}})=
\begin{cases}
v_{k_3}v_{k_4}&\text{ if \ }|k_3+k_4|>n\\
0&\text{ if \ }|k_3+k_4|\le n
\end{cases}\, .
\end{equation*}

Now, similarly to (\ref{fin2}), based on (\ref{R3sum1}) and
(\ref{R3ns0}), we obtain  from (\ref{one}) for every fixed
$n\in\mathbb{N}$ the following family of equations, which we call
{\it the third  form of the~KdV}:
\begin{equation}\label{fin4}
\partial _{t}\left( v_{k}-\frac{1}{6}B_{2}(v^2)_{k}-
\frac{1}{18}B_{30}^{(n)}(v^3)_{k}\right) =
\frac{i}{6}R_{3\text{res}}(v^{3})_{k}+
\frac{i}{6}R_{3\text{nres}1}^{(n)}(v^{3})_{k}+
\frac{i}{18}B_{40}^{(n)}(v^4)_k,
\end{equation}
where $k\in\mathbb{Z}_{0}$.
Integrating (\ref{fin4}) we obtain
\begin{equation}\label{fin4int}
\aligned
\left( v_{k}-\frac{1}{6}B_{2}(v^{2})_{k}-
\frac{1}{18}B_{30}^{(n)}(v^{3})_{k}\right)(t) -
\left( v_{k}-\frac{1}{6}B_{2}(v^{2})_{k}-
\frac{1}{18}B_{30}^{(n)}(v^{3})_{k}\right)(0) =
\\
\frac{i}{6}\int_{0}^{t}\biggl( R_{3\text{res}}(v^{3})_k+
R_{3\text{nres}1}^{(n)}(v^{3})_k+\frac{1}{3}
B_{40}^{(n)}(v^4)_k
\biggr)(t')dt'.
\endaligned
\end{equation}
For $v(t)=y(t)+v^{0}$, $y(0)=0$,
 we obtain as in (\ref{soly}) the equivalent equation
\begin{equation}\label{yvn}
\aligned
 y(t)-\frac{1}{3}B_{2}(v^{0},y(t))=
 \frac{1}{6}B_{2}(y(t),y(t))
+\frac{1}{18}\biggl(
B_{30}^{(n)}\left((y(t)+v^{0})^{3}\right)-
B_{30}^{(n)}\left((v^{0})^{3}\right)
\biggr)
+
\\
+
\frac{i}{6}\int_{0}^{t}\biggl( R_{3\text{res}}
(( y(t')+v^{0})^{3})
+R_{3\text{nres}1}^{(n)}(( y(t')+v^{0})^{3})+
\frac{1}{3}B_{40}^{(n)}(( y(t')+v^{0})^{4})
 \biggr)dt'.
\endaligned
\end{equation}

%%%%%%%%%%%%%%%%%%%%%%%%%

%%%%%%%%%%%%%%%%%%
\begin{theorem}
\label{T:Fn} Equation $(\ref{yvn})$ can be written in the form
$(\ref{Ftaun})$ with $\mathcal{F}_{\tau ,n}(y,v^{0})$ satisfying
$(\ref{Lipf})$ and $(\ref{F1n})$ with $\theta\in [0,1/2]$. Here
$\theta_0=0$ for $\theta<1/2$ and $\theta_0>0$ for $\theta=1/2$.
\end{theorem}

\begin{proof}
We apply the linear operator $L_{v^{0}}^{-1}$, which is
inverse to $I-\frac{1}{3}B_{2}(v^{0},y)$,
to (\ref{yvn}) and obtain similarly to (\ref{invinteq}) the
equation (\ref{Ftaun}). The Lipschitz estimate (\ref{Lipf})
and the fact that $C(C_0)$ is bounded
if (\ref{C(C0)}) holds   follow from
the estimate
$\|L_{v^0}^{-1}\|_{\dot{H}^\theta\to\dot{H}^\theta}
\le F(\|v^0\|_{\dot{H}^{0}})$, $\theta\ge0$
(see Lemma~\ref{L:linB2}, where, in fact, $\theta\ge-1$) and
the boundedness of the multilinear operators
$B_{40}^{1},B_{40}^{2}$,
$R_{3\text{res}}=A_{\mathrm{res}}$,
$R_{3\text{nres}1}^{(n)}$ in
$\dot{H}^{s}$,
see Lemma~\ref{L:B41n}, Lemma~\ref{L:B42n},
Lemma~\ref{L:Ares} and Lemma~\ref{L:R31}
(with $\alpha=0$).
The estimate (\ref{F1n})  follows from Lemma~\ref{L:R310}.
\end{proof}

\begin{theorem}\label{T:Fnex}
Let $0\le s\le 1/2$. Let
$v(0)=v^{0}\in \dot{H}^{s}$
be given. Then there exists $n_0$ large enough
and $T^*$ small enough, both depending on
$\|v^0\|_{\dot{H}^0}$, such that
for each $n\ge n_0$ the equation~$(\ref{yvn})$
has a
unique local solution $y\in
C_0([0,T^*];\dot{H}^{s})$ and
$v=y+v^{0}\in C([0,T^*];\dot{H}^{s})$
is a unique local solution of $(\ref{fin4int})$ on the time
interval $[0,T^*]$  with sufficiently small $T^*$. This
solution $v(t)$ can be extended to the maximal interval
$[0,T_{\max }^*)$
such that either $T_{\max }^*=+\infty $  or
$\limsup_{t\to T_{\max }^*-0}\|v(t)\|_{\dot{H}^{s}}=+\infty$.
\end{theorem}

\begin{proof}
The proof is similar to the proof of the first part of
Theorem~\ref{T:loca}. We use the Contraction Principle.  Based on
(\ref{Lipf}) and (\ref{F1n})  we derive (\ref{ftnl}) for $[0,T]$
with $T$ small enough and $n$ large enough. The existence and
uniqueness of a solution to (\ref{yvn}) follows in a standard way
from the Contraction Principle.  Extension to the maximal interval
is treated as in Theorem \ref{T:loca}, in particular,  the
uniqueness of a solution to (\ref{fin4int}) on a maximal interval
follows in a standard way from the proof of the local uniqueness.
\end{proof}

\begin{remark}
{\rm
Analogous local existence and uniqueness theorem for
solutions of (\ref{fin4int}) (similar results can be found in
\cite{Colliander03}, \cite{Colliander04}) can be proven for
$0\geq s>-1$ using the boundedness in negative spaces of
linear operator $L_{v^{0}}^{-1}$ and multilinear operators
which enter (\ref{fin4int}). We do not consider this case
here.
}
\end{remark}

\begin{theorem}\label{T:globs>0}
 Let $s\in (0,1/2]$ and $T>0$. For any initial data
$v(0)=v^{0}\in \dot{H}^{s}$ the solution $v=v^\infty$ of
equation~$(\ref{vint})$, constructed in Theorem~$\ref{T:gal}$, is
unique on the interval $t\in [0,T]$, is of the class
$C([0,T],\dot{H}^s)$ and depends Lipschitz continuously on the
initial data in the sense of~$(\ref{ylip})$.
\end{theorem}
\begin{proof}
First we observe that for each $n\in\mathbb{N}$ one
can construct a Galerkin truncated version of
(\ref{fin4}). In fact, this can be done step by
step as the previous derivation of~(\ref{fin4}) but
starting from~(\ref{onem}). Therefore it is clear
that the Galerkin solution of equation~(\ref{vm0})
(and~(\ref{onem})) established in
Proposition~\ref{L:L2m} is also the unique solution
of the Galerkin version of~(\ref{fin4int}) for
every $n\in\mathbb{N}$. Notice that $v^{(m)}$ is
independent of $n$. Similarly to
Theorem~\ref{T:gal} we can pass to the limit along
a subsequence $m_j\to\infty$ in the Galerkin
version of the equation~(\ref{fin4int}) (as in
Corollary~\ref{C:one-two}) to see that the solution
$v=v^\infty$ of (\ref{vint}) also
satisfies~(\ref{fin4int}) for each $n$. Therefore
$v=v^\infty$ belongs to $C([0,T^*_{\max}),
\dot{H}^s)$ on the maximal interval of existence in
Theorem~\ref{T:Fnex} and is unique.  In addition,
$\|v(t)\|_{\dot{H}^0}=\|v(0)\|_{\dot{H}^0}$ for all
$t\in [0,T^*_{\max})$. From Theorem~\ref{T:gal} we
have that $v$ is bounded in $\dot{H}^s$ on any
finite time interval. Hence $T^*_{\max}=\infty$,
the energy is conserved for all $t$ and
$v=v^\infty\in C([0,T], \dot{H}^s)$ for any finite
$T$. Finally, the Lipschitz continuity follows
from~(\ref{ylip}) with $\theta=s$.
\end{proof}

Finally, we consider the case $s=0$, which still remains
unsettled. For this purpose we use the regularization of the
initial data.
\begin{theorem}\label{T:globs=0}
Let $s=0$ and $T>0$. For any initial data
$v(0)=v^{0}\in \dot{H}^{0}$ the
equation~$(\ref{vint})$ has on $t\in [0,T]$ a unique solution
 of class $C([0,T],\dot{H}^0)$,
which  depends Lipschitz
continuously on the initial data.
\end{theorem}
\begin{proof}
Let $v(0)=v^0\in\dot{H}^0$.
We approximate the initial data $v^0$ by a sequence of
smooth functions $v_{(j)}^0\in\dot{H}^{s_0}$, where
$s_0>0$, and $\|v^0-v_{(j)}^0\|_{\dot{H}^{0}}\to0$
as $j\to\infty$. By Theorem~$\ref{T:globs>0}$,
for each $j$ there exits a unique global solution
$v_{(j)}\in C([0,T], \dot{H}^{s_0})$, which conserves energy,
that is,
$\|v_{(j)}(t)\|_{\dot{H}^0}=\|v_{(j)}(0)\|_{\dot{H}^0}$.
Moreover, from  the Lipschitz estimate~(\ref{ylip})
we have
$$
\|v_{(j)}-v_{(i)}\|_{L_{\infty }([0,T],\dot{H}^0)}\le
C(T,\|v^0\|_{\dot{H}^0})
\|v^0_{(j)}-v^0_{(i)}\|_{\dot{H}^0}.
$$
Since $v_{(j)}^0$ is a Cauchy sequence in $\dot{H}^{0}$, it follows
that $v_{(j)}(t)$ is a Cauchy sequence in $C([0,T],\dot{H}^0)$. We
again denote the corresponding limit by $v^\infty(t)$, although the
origin of this limit is different from that in Theorem~\ref{T:gal}.
We can pass to the limit in equation~(\ref{vint}) for $v=v_{(j)}$,
as $j\to\infty$, in a similar (simpler) way as we did in
Theorem~\ref{T:gal}. We obtain that $v^\infty$ is a solution
of~(\ref{vint}), which conserves energy since
$\|v^\infty(t)\|_{\dot{H}^0}=
\lim_{j\to\infty}\|v_{(j)}(t)\|_{\dot{H}^0}=
\lim_{j\to\infty}\|v_{(j)}^0\|_{\dot{H}^0}= \|v^0\|_{\dot{H}^0}$,
for all $t\in[0,T]$. A totally similar passage to the limit shows
that $v^\infty$ is a solution of (\ref{defofsol}), (\ref{fin3}),
and, most importantly, of  equation~(\ref{fin4int}) for any
$n\in\mathbb{N}$. Equation~(\ref{fin4int}) for a sufficiently large
$n$ depending on $\|v^0\|_{\dot{H}^0}$
  has a unique
local solution in view of Theorem~\ref{T:Fnex}.
However, there we have $T^*_{\max}=\infty$, since
$\|v^\infty(t)\|_{\dot{H}^0}=\|v^\infty(0)\|_{\dot{H}^0}$.
Since $v^\infty$ is the unique solution of~(\ref{fin4int})
for $n$ large enough, we apply Theorem~\ref{T:Fnex}
to show the Lipschitz continuous dependence on the initial
data in the $\dot{H}^0$ norm.
\end{proof}

We now study the Lipschitz dependence of $v\left( t\right)$ on
the initial data $v(0)$ in $\dot{H}^{s}$, $s<0$ if
$v\left( 0\right) $  is bounded in $\dot{H}^{0}$.

\begin{theorem} \label{T:negllip}
Let $T>0$ be given, and $\theta \in (-1, 0]$. If
$\|v^{0}\|_{\dot{H}^{0}}+\|w^{0}\|_{\dot{H}^{0}}\leq M$, then the
two solutions $v(t)$ and $w(t)$ of equation~$(\ref{vint})$, with
$v(0)=v^0$ and $w(0)=w^0$,  satisfy the estimate
\begin{equation}\label{vwlip}
\|v-w\|_{L_{\infty }([0,T],\dot{H}^{\theta }) }\leq C^{\prime }C_{M}^{T}
\|v^{0}-w^{0}\|_{\dot{H}^{\theta }},%\text{ \ for all \ }T>0,
\end{equation}
where $C^{\prime },C_{M}$ depend only on $M$ and $\theta$.
\end{theorem}

\begin{proof}
We derive (\ref{ftnl}) and (\ref{ylip}) based on the
boundedness in negative spaces of linear operator
$L_{v^{0}}^{-1}$ and multilinear operators which enter
(\ref{fin4int}). Namely, we use Lemma~\ref{L:negB21},
Lemma~\ref{L:R310},  Lemma~\ref{L:B41n} and
  Lemma~\ref{L:B42n}.
The intersection of the range of the values of $s$ in the
above lemmas  and Lemma~\ref{L:linB2} gives the restriction
on the range of $\theta$, namely $\theta\in (-1,0]$. From
 (\ref{ftnl})  we infer  (\ref{ylip})  which yields  (\ref{vwlip}).
\end{proof}

\setcounter{equation}{0}

\section{Appendix. Estimates for convolution operators}

\label{S:App}

In this section we study the continuity and
 smoothing properties of the
convolution operators $B_1$, $B_2$, $B_3$,
$B_{30}^{(n)}$, $B_4$, $B_{40}^{(n)}$,
$B_{40}^1$, $B_{40}^2$, $R_3$ and the resonant operator
$A_{\mathrm{res}}$ in the Sobolev spaces $\dot H^s$.
Although the corresponding constants $c_k$ do not play a
role in our analysis, the explicit expressions
for them can easily be given and, in fact,
are given in most cases at the end of the proofs.

\begin{lemma}\label{L:B1}
Let $\theta>3/2$. Then the bilinear operator
$B_{1}$ defined in $(\ref{B1})$ maps
$\dot{H}^{0}\times\dot{H}^{0}$ into
$\dot{H}^{-\theta}$
and
satisfies the estimate
\begin{equation}  \label{estB1}
\|B_{1}(u,v)\|_{\dot{H}^{-\theta}}\leq c_{1}(\theta)
\|u\|_{\dot{H}^{0}}\|v\|_{\dot{H}^{0}}.
\end{equation}
\end{lemma}
\begin{proof}
We observe that setting
$d_k=\sum_{k_{1}+k_{2}=k}
e^{i3kk_{1}k_{2}t}u_{k_{1}}v_{k_{2}}$
we have
$
|d_k|
\le
\|u\|_{\dot{H}^0}\|v\|_{\dot{H}^0}.
$
Furthermore, for $\sigma>1/2$ we see that
 $d\in\dot{H}^{-\sigma}$ with
$$
\|d\|_{\dot{H}^{-\sigma}}=
\biggl(\sum_{k\in\mathbb{Z}_0}
|d_k|^2 k^{-2\sigma}\biggr)^{1/2}
\le c(\sigma)\|u\|_{\dot{H}^0}\|v\|_{\dot{H}^0}
\quad\text{and}\quad c(\sigma)=\biggl(\sum_{k\in\mathbb{Z}_0}
|k|^{-2\sigma}\biggr)^{1/2}.
$$
Hence for $\theta=\sigma+1>3/2$,
$B_1(u,v)\in \dot{H}^{-\theta}$
with
$\|B_1(u,v)\|_{\dot{H}^{-\theta}}\le
\frac{1}{2}c(\theta-1)
\|u\|_{\dot{H}^0}\|v\|_{\dot{H}^0}.
$
\end{proof}

\begin{lemma}
\label{L:B2} Let $s>-1/2$. Then the bilinear operator
$B_{2}$ defined in $(\ref{B2})$ maps $\dot{H}^{s}\times
\dot{H}^{s}$ into $\dot{H}^{s+1}$ and satisfies the
estimate
\begin{equation}
\|B_{2}(u,v)\|_{\dot{H}^{s+1}}\leq c_{2}(s+1)
\|u\|_{\dot{H}^{s}}\|v\|_{\dot{H}^{s}}.  \label{estB2}
\end{equation}
\end{lemma}

\begin{proof}
By duality, (\ref{estB2}) is equivalent to the estimate
\begin{equation} \label{B2dual}
|(B_{2}(u,v),z)|\leq c_{2}(s+1)\| u\| _{\dot{H}^{s}}
\|v\|_{\dot{H}^{s}}\|z\|_{\dot{H}^{-(s+1 )}},
\end{equation}%
where $z$ is an arbitrary element in $\dot{H}^{-(s+1)}$.
Replacing the exponentials in~(\ref{B2}) by unity, setting
below $\tilde{u}_{k}=|u_{k}||k|^{s}$,
$\tilde{v}_{k}=|v_{k}||k|^{s}$,
$\tilde{z}_{k}=|z_{k}||k|^{-s-1 }$, and using the
inequality $|k_{1}+k_{2}|^{s+1}\leq
2^{s}(|k_{1}|^{s+1}+|k_{2}|^{s+1})$,
 we have
\begin{gather*}
\aligned|(B_{2}(u,v),z)|\leq \sum_{k}\sum_{k_{1}+k_{2}=k}
\frac{|u_{k_{1}}||v_{k_{2}}||z_{k}|}{|k_{1}||k_{2}|}=
\\
\sum_{k_{1}}\sum_{k_{2}}
\frac{|u_{k_{1}}||v_{k_{2}}||z_{k_{1}+k_{2}}|}
{|k_{1}||k_{2}|}=\sum_{k_{1}}\sum_{k_{2}}
\frac{\tilde{u}_{k_{1}}\tilde{v}_{k_{2}}
\tilde{z}_{k_{1}+k_{2}}|k_{1}+k_{2}|^{1 +s}}
{|k_{1}|^{1+s}|k_{2}|^{1+s}}\leq
\\
2^{s}\sum_{k_{1}}\sum_{k_{2}}
\frac{\tilde{u}_{k_{1}}\tilde{v}_{k_{2}}
\tilde{z}_{k_{1}+k_{2}}}{|k_{2}|^{1+s}}+
2^{s}\sum_{k_{1}}\sum_{k_{2}}
\frac{\tilde{u}_{k_{1}}\tilde{v}_{k_{2}}
\tilde{z}_{k_{1}+k_{2}}} {|k_{1}|^{1+s}}\leq
\\
2^{s}\left( \sum_{k_{1}}\sum_{k_{2}}
\tilde{u}_{k_{1}}^{2}\tilde{v}_{k_{2}}^{2} \right)^{1/2}
\left(\sum_{k_{2}}\frac{1}{|k_{2}|^{2s+2}}
\sum_{k_{1}}\tilde{z}_{k_{1}+k_{2}}^{2} \right) ^{1/2}+
\\
2^{s}\left( \sum_{k_{1}}\sum_{k_{2}}
\tilde{u}_{k_{1}}^{2}\tilde{v}_{k_{2}}^{2} \right)^{1/2}
\left(\sum_{k_{1}}\frac{1}{|k_{1}|^{2s+2}}
\sum_{k_{2}}\tilde{z}_{k_{1}+k_{2}}^{2} \right)^{1/2}=
\\
2^{s+1}c(s+1)\|u\|_{\dot{H}^{s}}\|v\|_{\dot{H}^{s}} \|z\|
_{\dot{H}^{-(s+1)}},
\endaligned
\end{gather*}
where
\begin{equation} \label{c2s}
c(p)=(\sum_{j\in \mathbb{Z}_{0}}j^{-2p})^{1/2}, \quad
p>1/2.
\end{equation}
This proves~(\ref{B2dual}), (\ref{estB2}) and provides an
expression for the constant~$c_{2}(s+1)$.
\end{proof}

Lemma~\ref{L:B2} essentially proves the following
well-known Banach algebra property of the Sobolev spaces
$H^s$, $s>1/2$.

\begin{corollary}
\label{C:B2} If $u,v\in H^s$, $s>1/2$, then the product
$uv\in H^s$ and
\begin{equation}\label{product}
\|uv\|_{H^s}\le K_1(s)\|u\|_{H^s}\|v\|_{H^s}.
\end{equation}

\begin{proof}
The estimate~(\ref{product}) is equivalent to
\begin{equation} \label{dualproduct}
|(uv,z)|\leq K_{1}(s)\|u\|_{H^{s}}\|v\|_{H^{s}}
\|z\|_{H^{-s}}.
\end{equation}
Setting $\tilde{u}_{k}=|u_{k}||k|^{s}$,
$\tilde{v}_{k}=|v_{k}||k|^{s}$,
$\tilde{z}_{k}=|z_{k}||k|^{-s}$ for $k\in \mathbb{Z}_{0}$
and
 arguing as in Lemma~\ref{L:B2} we have
\begin{gather*}
\aligned|(uv,z)|\leq \sum_{k_{1}\in \mathbb{Z}}
\sum_{k_{2}\in \mathbb{Z}}
|u_{k_{1}}||v_{k_{2}}||z_{k_{1}+k_{2}}|=|z_{0}|
\sum_{k_{1}+k_{2}=0}|u_{k_{1}}||v_{k_{2}}|+
|u_{0}|\sum_{k_{2}\in \mathbb{Z}_{0}}
\tilde{v}_{k_{2}}\tilde{z}_{k_{2}}+ \\
|v_{0}|\sum_{k_{1}\in \mathbb{Z}_{0}}
\tilde{u}_{k_{1}}\tilde{z}_{k_{1}}+ \sum_{k_{1}\in
\mathbb{Z}_{0}}\sum_{k_{2}\in \mathbb{Z}_{0}}
\frac{\tilde{u}_{k_{1}}\tilde{v}_{k_{2}}|k_{1}+k_{2}|^{s}
\tilde{z}_{k_{1}+k_{2}}} {|k_{1}|^{s}|k_{2}|^{s}} \leq
(3+2^{s}c(s-1))\|u\|_{H^{s}}\|v\|_{H^{s}}\|z\|_{H^{-s}},
\endaligned
\end{gather*}%
which proves~(\ref{dualproduct}) providing an expression
for~$K_{1}(s)$.
\end{proof}
\end{corollary}

\begin{lemma}
\label{L:B21}Let $s+\alpha \geq 0$, $\alpha <3/4$,
$s>-3/4$. Then the bilinear operator $B_{2}$ defined
in~$(\ref{B2})$ maps $\dot{H}^{s}\times \dot{H}^{s}$ into
$\dot{H}^{s+\alpha }$ and satisfies the estimate
\begin{equation} \label{estB2a}
\|B_{2}(u,v)\|_{\dot{H}^{s+\alpha }}\leq c_{2}^{\prime
}(s,\alpha) \|u\| _{\dot{H}^{s}}\|v\|_{\dot{H}^{s}}.
\end{equation}
\end{lemma}

\begin{proof}
By duality, (\ref{estB2a}) is equivalent to the estimate
\begin{equation*}
|(B_{2}(u,v),z)|\leq c_{2}^{\prime }(s,\alpha)
\|u\|_{\dot{H}^{s}}\|v\|_{\dot{H}^{s}}
\|z\|_{\dot{H}^{-(s+\alpha )}},
\end{equation*}%
where $z$ is an arbitrary element in
$\dot{H}^{-(s+\alpha)}$. Replacing the exponentials
in~(\ref{B2}) by unity, setting below
$\tilde{u}_{k}=|u_{k}||k|^{s}$,
$\tilde{v}_{k}=|v_{k}||k|^{s}$,
$\tilde{z}_{k}=|z_{k}||k|^{-s-\alpha }$, and using the
inequality $|k_{1}+k_{2}|^{s+\alpha }\leq
b(|k_{1}|^{s+\alpha }+|k_{2}|^{s+\alpha })$, with $b=\max (
2^{s+\alpha -1},1)$ we have
\begin{gather*}
|(B_{2}(u,v),z)|\leq \sum_{k}\sum_{k_{1}+k_{2}=k}
\frac{|u_{k_{1}}||v_{k_{2}}||z_{k}|}{|k_{1}||k_{2}|}=
\\
\sum_{k_{1}}\sum_{k_{2}}\frac{|u_{k_{1}}||v_{k_{2}}||
\tilde{z}_{k_{1}+k_{2}}||k_{1}+k_{2}|^{\alpha +s}}
{|k_{1}||k_{2}|}\leq
b\sum_{k_{1}}\sum_{k_{2}}\frac{\tilde{u}_{k_{1}}
\tilde{v}_{k_{2}}(|k_{1}|^{s+\alpha }+|k_{2}|^{s+\alpha })
\tilde{z}_{k_{1}+k_{2}}}{|k_{1}|^{1+s}|k_{2}|^{1+s}}\leq
\\
b\left( \sum_{k_{1}}\sum_{k_{2}}\tilde{u}_{k_{1}}^{2}
\frac{\tilde{z}_{k_{1}+k_{2}}}{|k_{2}|^{2-2\alpha }}
\right)^{1/2} \left(
\sum_{k_{2}}\sum_{k_{1}}\tilde{v}_{k_{2}}^{2}\frac{\tilde{z}_{k_{1}+k_{2}}}{%
|k_{1}|^{2+2s}} \right)^{1/2}+
\\
b\left(\sum_{k_{1}}\sum_{k_{2}}\tilde{u}_{k_{1}}^{2}
\frac{\tilde{z}_{k_{1}+k_{2}}}{|k_{2}|^{2s+2}}
\right)^{1/2}
\left(\sum_{k_{2}}\sum_{k_{1}}\tilde{v}_{k_{2}}^{2}
\frac{\tilde{z}_{k_{1}+k_{2}}}{|k_{1}|^{2-2\alpha }}
\right)^{1/2},
\end{gather*}
where $4-4\alpha >1$. Note that
\begin{equation*}
\sum_{k_{2}}\frac{|\tilde{z}_{k_{1}+k_{2}}|}{|k_{2}|^{p}}
\leq \left( \sum_{k_{2}}|\tilde{z}_{k_{1}+k_{2}}|^{2}
\right)^{1/2} \left(
\sum_{k_{2}}\frac{1}{|k_{2}|^{2p}}\right)^{1/2}\leq
c(p)\|\tilde{z}\|_{\dot{H}^{0}},\quad 2p>1
\end{equation*}%
with the same $c(p)$ as in (\ref{c2s}). Since $2-2\alpha
>1/2$ and $2+2s>1/2$, we have
\begin{gather*}
|(B_{2}(u,v),z)| \leq 2b\,c( 2-2\alpha) c( 2+2s)
\|\tilde{u}\|_{\dot{H}^{0}}\|\tilde{v}\|_{\dot{H}^{0}}
\|\tilde{z}\|_{\dot{H}^{0}}= c_{2}^{\prime }(s,\alpha)
\|u\|_{\dot{H}^{s}}\|v\|_{\dot{H}^{s}}
\|z\|_{\dot{H}^{-(s+\alpha )}}.
\end{gather*}
\end{proof}

\begin{lemma}
\label{L:negB2}Let $u,v\in H^{-1}$ and  $S>1/2$. Then
\begin{equation}
\Vert B_{2}(u,v)\Vert _{\dot{H}^{-S}}\leq c(S) \Vert u\Vert
_{\dot{H}^{-1}}\Vert v\Vert _{\dot{H}^{-1}}. \label{min2}
\end{equation}
\end{lemma}

\begin{proof}
Arguing by duality, we take an arbitrary element $z$ in
$\dot{H}^{S}$ and estimate the  inner product:
\begin{gather*}
|(B_{2}(u,v),z)|\leq \sum_{k}\sum_{k_{1}+k_{2}=k}
\frac{|u_{k_{1}}||v_{k_{2}}||z_{k}|}{|k_{1}||k_{2}|}=
\sum_{k_{1}}\sum_{k_{2}}
\frac{|u_{k_{1}}||v_{k_{2}}||z_{k_{1}+k_{2}}|}
{|k_{1}||k_{2}|}\leq \newline
\\
\left( \sum_{k_{1}}\sum_{k_{2}}
\frac{|u_{k_{1}}|^{2}}{|k_{1}|^{2}}%
|z_{k_{1}+k_{2}}|\right) ^{1/2} \left(
\sum_{k_{1}}\sum_{k_{2}}
\frac{|v_{k_{2}}|^{2}|z_{k_{1}+k_{2}}|}
{|k_{2}|^{2}}\right) ^{1/2}= \Vert u\Vert
_{\dot{H}^{-1}}\Vert v\Vert _{\dot{H}^{-1}} \sum_{k\in
\mathbb{Z}_{0}}|z_{k}|.
\end{gather*}%
Since
\begin{equation}\label{zk}
\sum_{k\in \mathbb{Z}_{0}}|z_{k}| \leq \left( \sum_{k\in
\mathbb{Z}_{0}}|z_{k}|^{2}|k|^{2S} \right)^{1/2} \left(
\sum_{k\in \mathbb{Z}_{0}}|k|^{-2S}\right)^{1/2}=
c(S)\|z\|_{\dot{H}^{S}},
\end{equation}
this proves~(\ref{min2}).
\end{proof}

\begin{lemma} \label{L:negB21}
 Let $-7/4<s\le0$. Then the bilinear operator $B_{2}$
defined in $(\ref{B2})$ maps $\dot{H}^0\times \dot{H}^s$
into $\dot{H}^{s}$, and satisfies the estimate
\begin{equation}
\|B_{2}(u,v)\|_{\dot{H}^{s}}\leq
c_{2}^{\prime\prime}(s)\|u\|_{\dot{H}^{0}}\|v\|_{\dot{H}^{s}}.
\end{equation}
\end{lemma}

\begin{proof}
Using duality, it is sufficient to estimate
\begin{equation*}
|B_{2}(u,v),z)|\leq
c_{2}^{\prime\prime}(s)\|u\|_{\dot{H}^{0}}\|v\|_{\dot{H}^{s}}
\|z\|_{\dot{H}^{-s}},
\end{equation*}
where $z$ is an arbitrary element in $\dot{H}^{-s}$. We set
$p=-s$, $0\le p<7/4$. Setting below
$\tilde{u}_{k}=|u_{k}|$, $\tilde{v}_{k}=|v_{k}||k|^{s}$,
$\tilde{z}_{k}=|z_{k}||k|^{-s}=|z_{k}||k|^{p}$, we have
\begin{gather*}
|B_{2}(u,v),z)|\leq \sum_{k\in\mathbb{Z}_0}
\sum_{k_{1}+k_{2}=k}
\frac{|u_{k_{1}}||v_{k_{2}}||z_{k}|}{|k_{1}||k_{2}|}
=\sum_{k_{1}}\sum_{k_{2}}
\frac{|u_{k_{1}}||v_{k_{2}}||z_{k_{1}+k_{2}}|}
{|k_{1}||k_{2}|}=\\
\sum_{k_{1},k_{2}}\frac{\tilde{u}_{k_{1}}\tilde{v}_{k_{2}}
\tilde{z}_{k_{1}+k_{2}}}
{|k_{1}||k_{1}+k_{2}|^{p}|k_{2}|^{1-p}} =\sum_{k_{1},k_{2}}
\frac{( \tilde{u}_{k_{1}}\tilde{v}_{k_{2}})^{1/2} (
\tilde{v}_{k_{2}}\tilde{z}_{k_{1}+k_{2}})^{1/2}
(\tilde{u}_{k_{1}}\tilde{z}_{k_{1}+k_{2}})^{1/2}}
{|k_{1}||k_{1}+k_{2}|^{p}|k_{2}|^{1-p}} \leq
\\
\left( \sum_{k_{1},k_{2}}\left( \tilde{u}_{k_{1}}
\tilde{v}_{k_{2}}\right) ^{2}\right)^{1/4} \left(
\sum_{k_{1},k_{2}} \left(
\tilde{v}_{k_{2}}\tilde{z}_{k_{1}+k_{2}}\right)^{2}
\right)^{1/4} \left(\sum_{k_{1},k_{2}}\left(
\tilde{u}_{k_{1}} \tilde{z}_{k_{1}+k_{2}}\right)^{2}
\right)^{1/4}
\left(\sum_{k_{1},k_{2}}K_{2}^{4}\right)^{1/4}=
\\
\|u\|_{\dot{H}^{0}}\| v\|_{\dot{H}^{s}}\|z\|_{\dot{H}^{-s}}
\left( \sum_{k_{1},k_{2}}K_{2}^{4}\right)^{1/4},
\end{gather*}%
where
\begin{equation*}
\sum_{k_1,k_2}K_{2}^{4}=
\sum_{k_{1},k_{2}}\frac{1}{|k_{1}|^{4}|k_{1}+k_{2}|^{4p}
|k_{2}|^{4-4p}}=\sum_{k_{1}}\frac{1}{|k_{1}|^{4}}%
\sum_{k_{2}}\frac{1}{| k_{1}+k_{2}|^{4p}|k_{2}|^{4-4p}}
=:(c_2^{\prime\prime}(s))^4.
\end{equation*}%
If $p\leq 1$, then $c_2^{\prime\prime}(s)\le
 (\sum_{k\in \mathbb{Z}_{0}}|k|^{-4})^{1/2}=c(2)$, since
\begin{gather*}
\sum_{k_{2}}%
\frac{1}{|k_{1}+k_{2}|^{4p}|k_{2}|^{4-4p}}\leq \left(
\sum_{k_{2}}\frac{1}{\left\vert k_{1}+k_{2}\right\vert
^{4}}\right) ^{p} \left(
\sum_{k_{2}}\frac{1}{|k_{2}|^{4}}\right)^{1-p}= \sum_{k\in
\mathbb{Z}_{0}}|k|^{-4}.
\end{gather*}%
If $7/4>p>1$ (so that $4p-4<3$ and $4p-8<-1$) we have
(see~(\ref{c2s}))
\begin{gather*}
\sum_{k_1,k_2}K_2^4\leq \sum_{k_1,k_2}\frac{|k_2|^{4p-4}}
{|k_1|^{4}|k_1+k_2|^{4p}}\leq 8\sum_{k_1,k_2}
\frac{|k_1|^{4p-4}+|k_1+k_2|^{4p-4}}
{|k_1|^4|k_1+k_2|^{4p}}=\\
8\sum_{k_1,k_2}\frac{|k_1|^{4p-8}}{|k_1+k_2|^{4p}}+
8\sum_{k_1,k_2}\frac1{|k_1|^{4}|k_1+k_2|^4}
=8c(4-2p)^2c(2p)^2+8c(2)^4,
\end{gather*}
which gives in this case $c_2^{\prime\prime}(s)\le
\bigl(8c(4-2p)^2c(2p)^2+8c(2)^4\bigr)^{1/4}$.
\end{proof}

\begin{lemma}
\label{L:B3} Let $s\geq 0$. Then the trilinear operator
$B_{3}$ in~$(\ref{B3})$ maps $(\dot{H}^{s})^3$ into
$\dot{H}^{s+2}$ and satisfies the estimate
\begin{equation}
\Vert B_{3}(u,v,w)\Vert _{\dot{H}^{s+2}}\leq c_{3}(s)\Vert u\Vert _{\dot{H}%
^{s}}\Vert v\Vert _{\dot{H}^{s}}\Vert w\Vert
_{\dot{H}^{s}}.  \label{estB3}
\end{equation}
\end{lemma}

\begin{proof}
Arguing by duality we start with the case $s=0$. Setting $\tilde{z}%
_{k}=|z_{k}|/k^{2}$ and using the inequality
$|k_{1}+k_{2}+k_{3}|\leq |k_{1}|+|k_{2}+k_{3}|$ we have
\begin{gather*}
\aligned|(B_{3}(u,v,w),z)|\leq \sum_{k_{1},k_{2},k_{3}}^{\text{nonres}}\frac{%
|u_{k_{1}}||v_{k_{2}}||w_{k_{3}}|\,\tilde{z}_{k_{1}+k_{2}+k_{3}}%
\,|k_{1}+k_{2}+k_{3}|^{2}}{|k_{1}||k_{1}+k_{2}||k_{2}+k_{3}||k_{3}+k_{1}|}%
\leq
\\
 \sum_{k_{1},k_{2},k_{3}}^{\text{nonres}}\frac{%
|u_{k_{1}}||v_{k_{2}}||w_{k_{3}}|\,\tilde{z}_{k_{1}+k_{2}+k_{3}}%
\,|k_{1}+k_{2}+k_{3}|}{|k_{1}+k_{2}||k_{2}+k_{3}||k_{3}+k_{1}|}%
+\!\sum_{k_{1},k_{2},k_{3}}^{\text{nonres}}\frac{%
|u_{k_{1}}||v_{k_{2}}||w_{k_{3}}|\,\tilde{z}_{k_{1}+k_{2}+k_{3}}%
\,|k_{1}+k_{2}+k_{3}|}{|k_{1}||k_{1}+k_{2}||k_{3}+k_{1}|}.
\endaligned
\end{gather*}%
Next, using the inequalities
\begin{equation}\label{k123}
|k_{1}+k_{2}+k_{3}|\le\left\{%
\begin{array}{lll}
    \frac{1}{2}|k_{1}+k_{2}|+\frac{1}{2}|k_{2}+k_{3}|+%
\frac{1}{2}|k_{3}+k_{1}|, & \\ \\
    |k_{1}|+|k_{1}+k_{2}|+|k_{3}+k_{1}|, & \\
\end{array}%
\right.
\end{equation}

in the first and second sum, respectively, we continue
\begin{gather*}
|(B_{3}(u,v,w),z)|\leq \sum_{k_{1},k_{2},k_{3}}^{\text{nonres}%
}|u_{k_{1}}||v_{k_{2}}||w_{k_{3}}|\,\tilde{z}_{k_{1}+k_{2}+k_{3}}\times
\newline
\\
\times \left( \frac{1}{2|k_{2}+k_{3}||k_{3}+k_{1}|}+\frac{1}{%
2|k_{1}+k_{2}||k_{3}+k_{1}|}+\frac{1}{2|k_{1}+k_{2}||k_{2}+k_{3}|}\right.
\newline
\\
\left. +\frac{1}{|k_{1}+k_{2}||k_{3}+k_{1}|}+\frac{1}{|k_{1}||k_{3}+k_{1}|}+%
\frac{1}{|k_{1}||k_{1}+k_{2}|}\right).
\end{gather*}
By symmetry this is equal to
\begin{gather*}
\frac{5}{2}\sum_{k_{1},k_{2},k_{3}}^{\text{nonres}%
}|u_{k_{1}}||v_{k_{2}}||w_{k_{3}}|\,\frac{\tilde{z}_{k_{1}+k_{2}+k_{3}}}{%
|k_{2}+k_{3}||k_{3}+k_{1}|}+2\sum_{k_{1},k_{2},k_{3}}^{\text{nonres}%
}|u_{k_{1}}||v_{k_{2}}||w_{k_{3}}|\,\frac{\tilde{z}_{k_{1}+k_{2}+k_{3}}}{%
|k_{1}||k_{1}+k_{2}|}\leq \newline
\\
\frac{5}{2}\Vert u\Vert \Vert v\Vert \Vert w\Vert \left(
\sum_{k_{1},k_{2},k_{3}}^{\text{nonres}}\right. \left. \frac{\tilde{z}%
_{k_{1}+k_{2}+k_{3}}^{2}}{|k_{2}+k_{3}|^{2}|k_{3}+k_{1}|^{2}}\right)
^{1/2}\!\!\!\!+\!\!2\Vert u\Vert \Vert v\Vert \Vert w\Vert
\left(
\sum_{k_{1},k_{2},k_{3}}^{\text{nonres}}\right. \left. \frac{\tilde{z}%
_{k_{1}+k_{2}+k_{3}}^{2}}{|k_{1}|^{2}|k_{1}+k_{2}|^{2}}\right)
^{1/2}\!\!\!\leq \newline
\\
\text{setting $k_{2}+k_{3}=l$, $k_{3}+k_{1}=j$,
$k_{1}+k_{2}+k_{3}=k$ in the first sum, where $l,j,k\neq
0$,}\newline
\\
\Vert u\Vert \Vert v\Vert \Vert w\Vert \biggl(\frac{5}{2}\biggl(%
\sum_{l}l^{-2}\sum_{j}j^{-2}\sum_{k}\tilde{z}_{k}^{2}\biggr)^{1/2}+2\biggl(%
\sum_{k_{1}}|k_{1}|^{-2}\sum_{k_{2}}|k_{1}+k_{2}|^{-2}\sum_{k_{3}}\tilde{z}%
_{k_{1}+k_{2}+k_{3}}^{2}\biggr)^{1/2}\biggr)
\\
=\frac{9}{2}\biggl(\sum_{i}i^{-2}\biggr)\Vert u\Vert \Vert
v\Vert \Vert w\Vert \Vert z\Vert _{\dot{H}^{-2}}=\frac{3\pi
^{2}}{2}\Vert u\Vert \Vert v\Vert \Vert w\Vert \Vert z\Vert
_{\dot{H}^{-2}}.
\end{gather*}%
The proof in the case $s\geq 0$ is similar. We use the
inequality that for $k_{1}\cdot k_{2}\cdot k_{3}\neq 0$
$|k_1+k_2+k_3|\le 3|k_1||k_2||k_3|$. Then, setting
$\tilde{u}_{k}=|k|^{s}|u_{k}|$,
$\tilde{v}_{k}=|k|^{s}|v_{k}|$,
$\tilde{w}_{k}=|k|^{s}|w_{k}|$,
$\tilde{z}_{k}=|z_{k}|/|k|^{s+2}$, and following the above
argument we have
\begin{gather*}
\aligned|(B_{3}(u,v,w),z)|\leq \sum_{k_{1},k_{2},k_{3}}^{\text{nonres}}\frac{%
\tilde{u}_{k_{1}}\tilde{v}_{k_{2}}\tilde{w}_{k_{3}}\,\tilde{z}%
_{k_{1}+k_{2}+k_{3}}{|k_{1}+k_{2}+k_{3}|^{2}}}{%
|k_{1}||k_{1}+k_{2}||k_{2}+k_{3}||k_{3}+k_{1}|}\frac{|k_{1}+k_{2}+k_{3}|^{s}%
}{|k_{1}|^{s}|k_{2}|^{s}|k_{3}|^{s}}\leq \newline
\\
3^{s}\sum_{k_{1},k_{2},k_{3}}^{\text{nonres}}\frac{\tilde{u}_{k_{1}}\tilde{v}%
_{k_{2}}\tilde{w}_{k_{3}}\,\tilde{z}_{k_{1}+k_{2}+k_{3}}{%
|k_{1}+k_{2}+k_{3}|^{2}}}{|k_{1}||k_{1}+k_{2}||k_{2}+k_{3}||k_{3}+k_{1}|}%
\leq 3^{s}\frac{3\pi ^{2}}{2}\Vert u\Vert _{\dot{H}^{s}}\Vert v\Vert _{\dot{H%
}^{s}}\Vert w\Vert _{\dot{H}^{s}}\Vert z\Vert
_{\dot{H}^{-s-2}},\endaligned
\end{gather*}%
which proves the lemma with~$c_{3}(s)=3^{(s+1)}\pi ^{2}/2$.
\end{proof}

\begin{remark} \label{R:B3neg}
{\rm In fact, one can show that
\begin{equation}
\Vert B_{3}(u,v,w)\Vert _{\dot{H}^{2-4\eta }} \leq
c_{3}^{\prime }(\eta)\Vert u\Vert _{\dot{H}^{-\eta }} \Vert
v\Vert _{\dot{H}^{-\eta }}\Vert w\Vert _{\dot{H}^{-\eta
}},\qquad \eta <\frac{1}{4}. \label{estB3neg}
\end{equation}%
}
\end{remark}

\begin{lemma}\label{L:R30}
 Let $1\ge s> 0$. Then the trilinear operator
 $B_{30}^{(n)}(v,v,v)$ defined in~$(\ref{B30})$ for every $t$
maps $\dot{H}^{s}\times \dot{H}^{s}\times \dot{H}^{s}$ into
$\dot{H}^{s}$ and satisfies the estimate
\begin{equation}\label{B30sgr}
\| B_{30}^{(n)}(v,v,v)\|_{\dot{H}^{s}}\leq \frac{\pi^2}{n^{s}}
\|v\|_{\dot{H}^{0}}^{2}\|v\|_{\dot{H}^{s}}.
\end{equation}
\end{lemma}
\begin{proof}
To estimate the norm of $B_{30}^{(n)}(v,v,v)$ we take
$z\in\dot{H}^{-s}$ and consider
\begin{equation*}
(B_{30}^{(n)}(v,v,v),z)=\sum_{ k\in\mathbb{Z}_0}
\sum_{k_{1}+k_{2}+k_{3}=k}^{\text{nonres}}
\frac{e^{i3(k_{1}+k_{2})(k_{2}+k_{3})(k_{3}+k_{1})t}}
{k_{1}(k_{1}+k_{2})(k_{2}+k_{3})(k_{3}+k_{1})} v_{k_{1}}\Pi
_{-n}v_{k_{2}}\Pi_{-n}v_{k_{3}}z_{k}.
\end{equation*}%
We set $\tilde{z}_{k}=|z_{k}|/|k|^{s}$ and  using
(\ref{k123}) obtain for $1\geq s>0$
\begin{gather*}
|( B_{30}^{(n)}(v,v,v),z)|\leq \sum_{k\in\mathbb{Z}_0}
\sum_{k_{1}+k_{2}+k_{3}=k}^{\text{nonres}}
\frac{|k_{1}+k_{2}+k_{3}|^{s}} {|k_{1}||k_{1}+k_{2}||
k_{2}+k_{3}||k_{3}+k_{1}|}
|v_{k_{1}}||\Pi_{-n}v_{k_{2}}||\Pi _{-n}v_{k_{3}}|
\tilde{z}_{k} \le
\\
\sum_{k\in\mathbb{Z}_0}
\sum_{k_{1}+k_{2}+k_{3}=k}^{\text{nonres}} \left(
\frac{|v_{k_{1}}||k_{2}|^{s}|\Pi _{-n}v_{k_{2}}||
\Pi_{-n}v_{k_{3}}|\tilde{z}_{k}}
{|k_{1}|^{1-s}|k_{1}+k_{2}||k_{2}|^{s}|k_{2}+k_{3}|
|k_{3}+k_{1}|}+%
\right.\\
\left. \frac{|v_{k_{1}}||\Pi_{-n}v_{k_{2}}|| k_{3}|^{s}
|\Pi_{-n}v_{k_{3}}|\tilde{z}_{k}} {|k_{1}||
k_{1}+k_{2}|^{1-s}|k_{2}+k_{3}||k_{3}+k_{1}| |k_{3}|^{s}}+
\frac{ |v_{k_{1}}||k_{2}|^{s}|\Pi _{-n}v_{k_{2}}| |\Pi
_{-n}v_{k_{3}}|\tilde{z}_{k}}
{|k_{2}|^{s}|k_{1}||k_{1}+k_{2}|| k_{2}+k_{3}|
|k_{3}+k_{1}|^{1-s}}\right)\le
\\
\frac{1}{n^{s}} \sum_{k_{1},k_{2},k_{3}}^{\text{nonres}}
\biggl(
 \frac{%
\left\vert v_{k_{1}}\right\vert \left\vert k_{2}\right\vert
^{s}\left\vert
\Pi _{-n}v_{k_{2}}\right\vert \left\vert \Pi _{-n}v_{k_{3}}\right\vert }{%
\left\vert k_{1}\right\vert ^{1-s}\left\vert
k_{1}+k_{2}\right\vert \left\vert k_{2}+k_{3}\right\vert
\left\vert k_{3}+k_{1}\right
\vert }+\\
\frac{%
\left\vert v_{k_{1}}\right\vert \left\vert \Pi
_{-n}v_{k_{2}}\right\vert
\left\vert k_{3}\right\vert ^{s}\left\vert \Pi _{-n}v_{k_{3}}\right\vert }{%
\left\vert k_{1}\right\vert \left\vert
k_{1}+k_{2}\right\vert
^{1-s}\left\vert k_{2}+k_{3}\right\vert \left\vert k_{3}+k_{1}\right\vert }+%
\frac{\left\vert v_{k_{1}}\right\vert \left\vert
k_{2}\right\vert ^{s}\left\vert \Pi
_{-n}v_{k_{2}}\right\vert \left\vert \Pi
_{-n}v_{k_{3}}\right\vert }{\left\vert k_{1}\right\vert
\left\vert k_{1}+k_{2}\right\vert \left\vert
k_{2}+k_{3}\right\vert \left\vert k_{3}+k_{1}\right\vert
^{1-s}} \biggr) \tilde{z}_{k_{1}+k_{2}+k_{3}}\le
\\
\frac{1}{n^{s}}\|v\|_{\dot{H}^{0}}^{2}
\| v\|_{\dot{H}^{s}}%\sup_{k}\tilde{z}_{k}
\biggr[ \sum_{k_{1},k_{2},k_{3}}^{\text{nonres}}
 \frac{\tilde{z}_{k_{1}+k_{2}+k_{3}}^2}
{\left\vert k_{1}\right\vert ^{2-2s}\left\vert
k_{1}+k_{2}\right\vert ^{2}\left\vert
k_{2}+k_{3}\right\vert ^{2}\left\vert
k_{3}+k_{1}\right\vert ^{2}}+
\\
\frac{\tilde{z}_{k_{1}+k_{2}+k_{3}}^2} {\left\vert
k_{1}\right\vert ^{2}\left\vert k_{1}+k_{2}\right\vert
^{2-2s}\left\vert k_{2}+k_{3}\right\vert ^{2}\left\vert
k_{3}+k_{1} \right\vert ^{2}}+
\frac{\tilde{z}_{k_{1}+k_{2}+k_{3}}^2}{%
\left\vert k_{1}\right\vert ^{2}\left\vert
k_{1}+k_{2}\right\vert ^{2}\left\vert
k_{2}+k_{3}\right\vert ^{2}\left\vert
k_{3}+k_{1}\right\vert ^{2-2s}} \biggr] ^{1/2} \leq
\\
\frac{3}{n^{s}}\|v\|_{\dot{H}^{0}}^{2} \| v\|_{\dot{H}^{s}}
\biggr[ \sum_{k_{1},k_{2},k_{3}}^{\text{nonres}}
 \frac{\tilde{z}_{k_{1}+k_{2}+k_{3}}^2}
{%\left\vert k_{1}\right\vert ^{2-2s}
|k_{1}+k_{2}|^{2}|k_{2}+k_{3}|^{2}} \biggr] ^{1/2} =
\frac{\pi^2}{n^{s}}
\Vert v\Vert _{\dot{H}^{0}}^{2}\Vert v\Vert _{\dot{H}%
^{s}}\Vert z\Vert _{\dot{H}^{-s}}.
\end{gather*}%
\end{proof}

For $s\le0$ we have the following estimate.
\begin{lemma}\label{L:R30neg}
 Let $s\le0$. Then the trilinear operator
 $B_{30}^{(n)}(v,v,v)$
satisfies the estimate
\begin{equation}\label{B30sgrneg1}
\| B_{30}^{(n)}(v,v,v)\|_{\dot{H}^{s}}\leq
\frac{C(p,\alpha)}{n^{2\alpha}} \|v\|_{\dot{H}^{s}}^{3},
\end{equation}
where $p=-s\ge0$, $\alpha>0$ and $p+\alpha<5/6$.
\end{lemma}
\begin{proof}
If $\ s=-p\leq 0$ we set
 $\tilde{v}_{k}=|v_{k}||k|^{s}$,
 $\tilde{z}_{k}=|z_{k}|/|k|^{s}$
 so that
$\|\tilde{v}\|_{\dot{H}^0}=\|{v}\|_{\dot{H}^s}$,
$\|\tilde{z}\|_{\dot{H}^0}=\|{z}\|_{\dot{H}^{-s}}$
 and obtain
\begin{gather*}
|( B_{30}^{(n)}(v,v,v),z)|\leq \sum_{k\in\mathbb{Z}_0}
\sum_{k_{1}+k_{2}+k_{3}=k}^{\text{nonres}}
\frac{|k_{1}|^{p}\tilde{v}_{k_{1}}|k_{2}|^{p} \Pi
_{-n}\tilde{v}_{k_{2}}| k_{3}|^{p} \Pi
_{-n}\tilde{v}_{k_{3}} \tilde{z}_{k}} {|
k_{1}||k_{1}+k_{2}||k_{2}+k_{3}||k_{3}+k_{1}|
|k_{1}+k_{2}+k_{3}|^{p}} \le
\\
\frac{1}{n^{2\alpha }}\sum_{k\in\mathbb{Z}_0}
\sum_{k_{1}+k_{2}+k_{3}=k}^{\text{nonres}}
K(k_1,k_2,k_3)\,\tilde{v}_{k_{1}}\Pi _{-n}\tilde{v}_{k_{2}}
\Pi _{-n}\tilde{v}_{k_{3}} \tilde{z}_{k},
\end{gather*}%
where
\begin{equation} \label{KK}
K\left( k_{1},k_{2},k_{3}\right) =\frac{\left\vert
k_{2}\right\vert ^{p+\alpha }\left\vert k_{3}\right\vert
^{p+\alpha }}{\left\vert k_{1}\right\vert ^{1-p}\left\vert
k_{1}+k_{2}\right\vert \left\vert k_{2}+k_{3}\right\vert
\left\vert k_{3}+k_{1}\right\vert \left\vert
k_{1}+k_{2}+k_{3}\right\vert ^{p}}.
\end{equation}%
Now we include condition $|k|>0$ into definition of
nonresonant summation and write
\begin{gather}
\left\vert \left( B_{30}^{(n)}(v,v,v),z\right) \right\vert \leq \frac{1}{%
n^{2\alpha }}\sum_{k_{1},k_{2},k_{3}}^{\text{nonres}}K\left\vert \tilde{v}%
_{k_{1}}\right\vert \left\vert \tilde{v}_{k_{2}}\right\vert
\left\vert \tilde{v}_{k_{3}}\right\vert \tilde{z}_{k}
=\label{BB3sum}
\\
\frac{1}{n^{2\alpha
}}\sum_{k_{1},k_{2},k_{3}}^{\text{nonres}} \left(
\left\vert \tilde{v}_{k_{1}}\right\vert \left\vert
\tilde{v}%
_{k_{2}}\right\vert \left\vert \tilde{v}_{k_{3}}
\right\vert \right) ^{1/3}\left( \left\vert
\tilde{v}_{k_{2}}\right\vert \left\vert
\tilde{v}_{k_{3}}\right\vert \tilde{z}_{k}\right) ^{1/3}
\left( \left\vert \tilde{v}_{k_{1}}\right\vert \left\vert
\tilde{v}_{k_{3}}
\right\vert \tilde{z}%
_{k}\right) ^{1/3}\left( \left\vert \tilde{v}_{k_{1}}
\right\vert \left\vert \tilde{v}_{k_{2}}\right\vert
\tilde{z}_{k}\right) ^{1/3}K \le\notag
\\
\leq \frac{1}{n^{2\alpha }}\left( \sum_{k_{1},k_{2},k_{3}}^{\text{nonres}%
}\left( \left\vert \tilde{v}_{k_{1}}\right\vert \left\vert \tilde{v}%
_{k_{2}}\right\vert \left\vert \tilde{v}_{k_{3}}\right\vert
\right) ^{2}\right) ^{1/6}\left(
\sum_{k_{1},k_{2},k_{3}}^{\text{nonres}}\left(
\left\vert \tilde{v}_{k_{2}}\right\vert \left\vert \tilde{v}%
_{k_{3}}\right\vert \tilde{z}_{k}\right) ^{2}\right) ^{1/6}
\times\notag
\\
\left(
\sum_{k_{1},k_{2},k_{3}}^{\text{nonres}}\left( \left\vert \tilde{v}%
_{k_{1}}\right\vert \left\vert \tilde{v}_{k_{3}}\right\vert \tilde{z}%
_{k}\right) ^{2}\right) ^{1/6}\left( \sum_{k_{1},k_{2},k_{3}}^{\text{nonres}%
}\left( \left\vert \tilde{v}_{k_{1}}\right\vert \left\vert \tilde{v}%
_{k_{2}}\right\vert \tilde{z}_{k}\right) ^{2}\right)
^{1/6}\left(
\sum_{k_{1},k_{2},k_{3}}^{\text{nonres}}K^{3}\right)
^{1/3}=
\notag\\
\frac{1}{n^{2\alpha }}
\|\tilde{v}\|_{\dot{H}^{0}}^{3}\|\tilde{z}\|_{\dot{H}^{0}}
\left(\sum_{k_1,k_2,k_3}^{\text{nonres}}K^{3}\right)^{1/3}=
\frac{1}{n^{2\alpha }}\|v\|_{\dot{H}^{s}}^{3}
\|z\|_{\dot{H}^{-s}} \left(
\sum_{k_1,k_2,k_3}^{\text{nonres}}K^{3}\right)^{1/3},
\notag
\end{gather}
where
$\left( \sum_{k_{1},k_{2},k_{3}}^{\text{%
nonres}}K^{3}\right)^{1/3}=:C(p,\alpha)<\infty$ by
Lemma~\ref{L:K3} (where $\gamma=\delta:=p+\alpha$).
\end{proof}

\begin{lemma} \label{L:K3}
Let $0\le p\le1$ and let $K(k_1,k_2,k_3)$ be defined as
follows
\begin{equation} \label{KKK}
K\left( k_{1},k_{2},k_{3}\right) = \frac{|k_{2}|^\gamma
|k_{3}|^\delta } {|k_{1}|^{1-p}|k_{1}+k_{2}||k_{2}+k_{3}||
k_{3}+k_{1}| |k_{1}+k_{2}+k_{3}|^{p}}\,.
\end{equation}%
If
\begin{equation}\label{gamdel}
   \gamma+\delta<\frac53\,,
\end{equation}
then
$$
\sum_{k_{1},k_{2},k_{3}}^{\mathrm{nonres}} K(
k_{1},k_{2},k_{3})^{3}\le C(\gamma,\delta,p)<\infty.
$$
\end{lemma}
\begin{proof}
By (\ref{KKK})
\begin{equation*}
K( k_{1},k_{2},k_{3})^3=\frac{|k_{2}|^{3\gamma}
|k_{3}|^{3\delta}} {|k_{1}|^{3-3p}|k_{1}+k_{2}|^{3}
|k_{2}+k_{3}\|^{3}|k_{3}+k_{1}|^{3}
|k_{1}+k_{2}+k_{3}|^{3p}}\,.
\end{equation*}
We denote
\begin{equation*}
l_{1}=k_{1}+k_{2},\ l_{2}=k_{2}+k_{3}, \ l_{3}=k_{1}+k_{3},
\ l_{0}=k=k_{1}+k_{2}+k_{3}.
\end{equation*}%
Obviously,
\begin{equation*}
k_{1}=\frac{1}{2}\left( l_{3}+l_{1}-l_{2}\right)
,k_{3}=\frac{1}{2}\left( l_{3}-l_{1}+l_{2}\right)
,k_{2}=\frac{1}{2}\left( l_{2}-l_{3}+l_{1}\right)
,k=\frac{1}{2}\left( l_{2}+l_{3}+l_{1}\right)
\end{equation*}%
and setting $k_{1}=l_{4}$, $k=k_1+k_2+k_3=l_0$ we have
\begin{equation*}
K\left( k_{1},k_{2},k_{3}\right)^{3}=
\frac{|k_{2}|^{3\gamma}|k_{3}|^{3\delta}} {\left(
|l_{4}||l_{1}||l_{2}||l_{3}|\right)^{3-3p} \left(
|l_{0}||l_{1}||l_{2}||l_{3}|\right)^{3p}}
\end{equation*}%
and therefore
\begin{equation}\label{prodwithp}
\sum_{k_{1},k_{2},k_{3}}^{\text{nonres}} K(k_1,k_2,k_3)^3
\le \left( \sum_{k_{1},k_{2},k_{3}}^{\text{nonres}}
\frac{|k_{2}|^{3\gamma}|k_{3}|^{3\delta}}
{(|l_{4}||l_{1}||l_{2}||l_{3}|)^{3}}\right)^{1-p} \left(
\sum_{k_{1},k_{2},k_{3}}^{\text{nonres}}
\frac{|k_{2}|^{3\gamma}|k_{3}|^{3\delta}}
{(|l_{0}||l_{1}||l_{2}||l_{3}|)^{3}}\right)^{p}\,.
\end{equation}%
We have four linear functions $l_{1},l_{2},l_{3},k=l_{0}$
defined on the $\left( k_{1},k_{2},k_{3}\right)$ space, any
three of the four functions are linearly independent.
Therefore
\begin{equation*}
|k_{2}| +|k_{3}| \leq C_{1}\sum_{i\neq j}|l_{i}| ,\quad
\text{for} \quad j=0,1,2,3.
\end{equation*}
Similarly, four linear functions
$l_{1},l_{2},l_{3},k_{1}=l_{4}$ are defined on the  $\left(
k_{1},k_{2},k_{3}\right)$ space, any three of them are
linearly independent. Therefore
\begin{equation*}
|k_{2}| +|k_{3}| \leq C_{2}\sum_{i\neq j}|l_{i}|, \quad
\text{for}\quad j=1,2,3,4.
\end{equation*}%
Hence%
\begin{equation*}   \label{k2301}
|k_{2}|+|k_{3}|\leq C_{3} \left(
\prod\limits_{j=0}^{3}\sum_{i\neq j}
|l_{i}|\right)^{1/4},\quad i=0,1,2,3.
\end{equation*}
\begin{equation*} \label{k2311}
|k_{2}| +|k_{3}|\leq C_{3} \left(
\prod\limits_{j=0}^{3}\sum_{i\neq j}
|l_{i}|\right)^{1/4},\quad i=1,2,3,4.
\end{equation*}
Using inequalities of the type $l_0l_1l_2l_3\le
1/2(l_0^2l_1^2+l_2^2l_3^2)
\le1/4(l_0l_1^3+l_0^3l_1+l_2l_3^3+l_2^3l_3)$ we obtain
\begin{equation*}
|k_{2}|+|k_{3}|\leq C_{4}\sum_{j=0}^{3}\sum_{i\neq j}
(|l_{j}||l_{i}|^{3})^{1/4}, \quad | k_{2}|+|k_{3}|\leq
C_{4}\sum_{j=1}^{4}\sum_{i\neq j}
(|l_{j}||l_{i}|^{3})^{1/4}
\end{equation*}%
and
\begin{equation*}
\sum_{k_{1},k_{2},k_{3}}^{\text{nonres}}
\frac{|k_{2}|^{3\gamma}|k_{3}|^{3\delta}}
{(|l_{4}||l_{1}||l_{2}||l_{3}|)^{3}} \leq
C_{5}\sum_{j=1}^{4}\sum_{i\neq j}
\sum_{k_{1},k_{2},k_{3}}^{\text{nonres}} \frac{(|l_{i}||
l_{j}|^{3})^{3(\gamma+\delta)/4}}
{(|l_{4}||l_{1}||l_{2}||l_{3}|)^{3}}\,.
\end{equation*}%
Each term in the above sum can be estimated quite
similarly, we take, for example, $j=2$, $i=1$ (all
remaining combinations are obtained by an obvious
permutation of indices). Expressing below $l_2$ as the
linear combination of $l_1$, $l_3$ and $l_4$ we get
\begin{gather*}
\sum_{k_{1},k_{2},k_{3}}^{\text{nonres}}
\frac{(|l_{1}||l_{2}|^{3})^{3(\gamma+\delta)/4}}
{(|l_{4}||l_{1}||l_{2}||l_{3}|)^{3}}
=\sum_{k_{1},k_{2},k_{3}}^{\text{nonres}}
\frac{|l_{1}|^{3(\gamma+\delta)/4}|l_{2}|^{9(\gamma+\delta)/4}}
{(|l_{4}||l_{1}||l_{2}||l_{3}|)^{3}}=
\sum_{k_{1},k_{2},k_{3}}^{\text{nonres}}
\frac{|l_{1}|^{3(\gamma+\delta)/4}
|l_{2}|^{9(\gamma+\delta)/4-3}}
{|l_{4}|^{3}|l_{1}|^{3}|l_{3}|^{3}} \le
\\
C_{5}^{\prime }\sum_{k_{1},k_{2},k_{3}}^{\text{nonres}}
\frac{|l_{4}|^ {3(\gamma+\delta)/4+9(\gamma+\delta)/4-3}+
|l_{1}|^{3(\gamma+\delta)/4+9(\gamma+\delta)/4-3}
+|l_{3}|^{3(\gamma+\delta)/4+9(\gamma+\delta)/4-3}}
{|l_{4}|^{3}|l_{1}|^{3}|l_{3}|^{3}}\,.
\end{gather*}%
The series converges if
$3(\gamma+\delta)/4+9(\gamma+\delta)/4-3<2$, which is
condition~(\ref{gamdel}). Since the second factor
in~(\ref{prodwithp}) is treated in exactly the same way,
the proof is complete.
\end{proof}

\begin{lemma} \label{L:R310}
 If $0<s\le1$, then
\begin{equation}\label{B30s>0}
\|B_{30}^{(n)}(u,u,v)\|_{\dot{H}^{s}}+
\|B_{30}^{(n)}(u,v,u)\|_{\dot{H}^{s}}+
\|B_{30}^{(n)}(v,u,u)\|_{\dot{H}^{s}}\leq
\frac{C}{n^{s}}\|u\|_{\dot{H}^{0}}^{2}\|v\|_{\dot{H}^{s}}.
\end{equation}%
If $s\le0$ and $p=-s\leq 1$, $\alpha >0$,
 $p+2\alpha <5/3$, % and  $\alpha<5/6$,
then
\begin{equation}\label{B30sle0}
\|B_{30}^{(n)}(u,u,v)\|_{\dot{H}^{s}}+
\|B_{30}^{(n)}(u,v,u)\|_{\dot{H}^{s}}+
\|B_{30}^{(n)}(v,u,u)\|_{\dot{H}^{s}}\leq
\frac{C(p,\alpha)}{n^{2\alpha}}
\|u\|_{\dot{H}^{0}}^{2}\|v\|_{\dot{H}^{s}}.
\end{equation}
\end{lemma}

\begin{proof}
We consider the second case. If $-s=p\leq 1$ we set
$\tilde{v}_{k}=|v_{k}||k|^{s}$, and for $z\in\dot{H}^{-s}$
we set
 $\tilde{z}_{k}=|z_{k}||k|^{-s}=|z_{k}||k|^{p}$.
Note that $B_{30}^{(n)}(u,u,v)=B_{30}^{(n)}(u,v,u)$, so we
consider only $B_{30}^{(n)}(u,u,v)$ and $B_{30}^{(n)}(v,u,u)$.
Using duality we write
\begin{gather*}
|( B_{30}^{(n)}(u,u,v),z)|\leq \\
\sum_{k\in\mathbb{Z}_0}
\sum_{k_{1}+k_{2}+k_{3}=k}^{\text{nonres}} \frac{|
u_{k_{1}}|| \Pi_{-n}u_{k_{2}}|| k_{3}|^{p}
|\Pi_{-n}\tilde{v}_{k_{3}}|\tilde{z}_{k}} {|k_{1}||
k_{1}+k_{2}||k_{2}+k_{3}||k_{3}+k_{1}|
|k_{1}+k_{2}+k_{3}|^{p}}\le
 \\
\frac{1}{n^{2\alpha }}\sum_{k\in\mathbb{Z}_0}
\sum_{k_{1}+k_{2}+k_{3}=k}^{\text{nonres}}
K_{2}\,|u_{k_{1}}||\Pi _{-n}u_{k_{2}}|
|\Pi_{-n}\tilde{v}_{k_{3}}|\tilde{z}_{k}=
\\
\frac{1}{n^{2\alpha }}%
\sum_{k_{1},k_{2},k_{3}}^{\text{nonres}}K_{2}\left\vert
u_{k_{1}}\right\vert
\left\vert \Pi _{-n}u_{k_{2}}\right\vert \left\vert \Pi _{-n}\tilde{v}%
_{k_{3}}\right\vert \tilde{z}_{k_{1}+k_{2}+k_{3}},
\end{gather*}
where
\begin{equation*}
K_{2}( k_{1},k_{2},k_{3})= \frac{| k_{2}|^{\alpha
}|k_{3}|^{p+\alpha}}
{|k_{1}||k_{1}+k_{2}||k_{2}+k_{3}||k_{3}+k_{1}|
|k_{1}+k_{2}+k_{3}|^{p}}\,.
\end{equation*}%
Now we write similarly to (\ref{BB3sum})
\begin{gather*}
|( B_{30}^{(n)}(u,u,v),z)|\leq \frac{1}{n^{2\alpha }}
\sum_{k_{1},k_{2},k_{3}}^{\text{nonres}}
K_{2}|u_{k_{1}}||u_{k_{2}}||\tilde{v}_{k_{3}}|
 \tilde{z}_{k}\le
 \\
  \frac{1}{n^{2\alpha }}\|u\|_{\dot{H}^{0}}^{2}
\|v\|_{\dot{H}^{s}}\|z\|_{\dot{H}^{-s}} \left(
\sum_{k_{1},k_{2},k_{3}}^{\text{nonres}}K_{2}^{3}
\right)^{1/3}
\end{gather*}
Since $0\le p\le1$, $|k_1|\ge|k_1|^{1-p}$ and we can apply
Lemma~\ref{L:K3} with $\gamma=\alpha$ and
$\delta=p+\alpha$, which gives that
\begin{equation*}
\left(
\sum_{k_{1},k_{2},k_{3}}^{\text{nonres}}K_{2}^{3}\right)
^{1/3}\leq C^{\prime }\left( p,\alpha \right) <\infty,
\end{equation*}
provided that $ p+2\alpha<5/3. $

Similarly,
\begin{gather*}
|( B_{30,n}(v,u,u),z)|\le
\\
\frac{1}{n^{2\alpha }}\sum_{k\in\mathbb{Z}_0}
\sum_{k_{1}+k_{2}+k_{3}=k}^{\text{nonres}}
\frac{|k_{1}|^{p}|\tilde{v}_{k_{1}}||k_{2}|^{\alpha }
|\Pi_{-n}u_{k_{2}}||k_{3}|^{\alpha } |\Pi
_{-n}u_{k_{3}}|\tilde{z}_{k}}
{|k_{1}||k_{1}+k_{2}||k_{2}+k_{3}||k_{3}+k_{1}|
|k_{1}+k_{2}+k_{3}|^{p}} \le
\\
\frac{1}{n^{2\alpha }}\|u\|_{\dot{H}^{0}}^{2} \|
v\|_{\dot{H}^{s}}\| z\|_{\dot{H}^{-s}} \left(
\sum_{k_{1},k_{2},k_{3}}^{\text{nonres}}
K_{3}^{3}\right)^{1/3}\,,
\end{gather*}
where
\begin{equation*}
K_{3}\left( k_{1},k_{2},k_{3}\right) =\frac{\left\vert
k_{2}\right\vert ^{\alpha }\left\vert k_{3}\right\vert
^{\alpha }}{\left\vert k_{1}\right\vert ^{1-p}\left\vert
k_{1}+k_{2}\right\vert \left\vert k_{2}+k_{3}\right\vert
\left\vert k_{3}+k_{1}\right\vert \left\vert
k_{1}+k_{2}+k_{3}\right\vert ^{p}}\,.
\end{equation*}
By Lemma~\ref{L:K3} the series converges if $2\alpha<5/3$,
and this condition follows from the previous one since
$p\ge0$. The first case~(\ref{B30s>0}) is simpler and is
treated similarly to Lemma~\ref{L:R30}.
\end{proof}

\begin{lemma}
\label{L:B4} Let $s\geq 0$ and $\varepsilon \in (0,\frac{1}{2})$.
Then the
multi-linear operator $B_{4}$ defined in $(\ref{B4})$ maps $(\dot{H}%
^{s})^{4} $ into $\dot{H}^{s+\varepsilon }$ and satisfies
the estimate
\begin{equation}
\Vert B_{4}(u,v,w,\varphi )\Vert _{\dot{H}^{s+\varepsilon
}}\leq
c_{4}(s,\varepsilon )\Vert u\Vert _{\dot{H}^{s}}\Vert v\Vert _{\dot{H}%
^{s}}\Vert w\Vert _{\dot{H}^{s}}\Vert \varphi \Vert
_{\dot{H}^{s}}. \label{estB4}
\end{equation}
\end{lemma}

\begin{proof}
Since $B_{4}=\frac{1}{2}B_{4}^{1}+B_{4}^{2}$
(see~(\ref{B4sum})), it suffices to estimate $B_{4}^{1}$
and $B_{4}^{2}$. We first consider the case $s=0$. Setting
$\tilde{z}_{k}=|z_{k}|/|k|^{\varepsilon }$, $\varepsilon
<1/2$ and using the inequality
\begin{equation*}
|k_{1}+k_{2}+k_{3}+k_{4}|^{\varepsilon }\leq
2^{-\varepsilon }(|k_{1}+k_{2}|^{\varepsilon
}+|k_{1}+k_{3}+k_{4}|^{\varepsilon
}+|k_{2}+k_{3}+k_{4}|^{\varepsilon }),
\end{equation*}%
we have
\begin{gather*}
\aligned|(B_{4}^{1}(u,v,w,\varphi ),z)|\leq
\!\!\!\sum_{k_{1},k_{2},k_{3},k_{4}}^{\text{nonres}}\frac{%
|u_{k_{1}}||v_{k_{2}}||w_{k_{3}}||\varphi _{k_{4}}|\,\tilde{z}%
_{k_{1}+k_{2}+k_{3}+k_{4}}\,|k_{1}+k_{2}+k_{3}+k_{4}|^{\varepsilon }}{%
|k_{1}+k_{2}||k_{1}+k_{3}+k_{4}||k_{2}+k_{3}+k_{4}|}\leq
\newline
\\
2^{-\varepsilon }\sum_{k_{1},k_{2},k_{3},k_{4}}^{\text{nonres}}{%
|u_{k_{1}}||v_{k_{2}}||w_{k_{3}}||\varphi _{k_{4}}|\,\tilde{z}%
_{k_{1}+k_{2}+k_{3}+k_{4}}}\biggl(\frac{1}{|k_{1}+k_{2}|^{1-\varepsilon
}|k_{1}+k_{3}+k_{4}||k_{2}+k_{3}+k_{4}|}+\newline
\\
+\frac{1}{|k_{1}+k_{2}||k_{1}+k_{3}+k_{4}|^{1-\varepsilon
}|k_{2}+k_{3}+k_{4}|}+\frac{1}{%
|k_{1}+k_{2}||k_{1}+k_{3}+k_{4}||k_{2}+k_{3}+k_{4}|^{1-\varepsilon }}\biggr)%
\,.\endaligned
\end{gather*}%
The three terms so obtained satisfy the same bound, and it
suffices to consider any of them, say, the first:
\begin{gather*}
\sum_{k_{1},k_{2},k_{3},k_{4}}^{\text{nonres}}\frac{%
|u_{k_{1}}||v_{k_{2}}||w_{k_{3}}||\varphi _{k_{4}}|\,\tilde{z}%
_{k_{1}+k_{2}+k_{3}+k_{4}}}{|k_{1}+k_{2}|^{1-\varepsilon
}|k_{1}+k_{3}+k_{4}||k_{2}+k_{3}+k_{4}|}\leq \\
\|u\|\| v\|\|w\|\| \varphi \|\biggl(%
\sum_{k_{1},k_{2},k_{3},k_{4}}^{\text{nonres}}
\frac{\tilde{z}_{k_{1}+k_{2}+k_{3}+k_{4}}^{2}}
{|k_{1}+k_{2}|^{2-2\varepsilon}|k_{1}+k_{3}+k_{4}|^{2}
|k_{2}+k_{3}+k_{4}|^{2}}\biggr)^{1/2}\le
\\
\Vert u\Vert \Vert v\Vert \Vert w\Vert \Vert \varphi \Vert \biggl(%
\sum_{l_1}|l_1|^{-2+2\varepsilon }\sum_{l_2}(l_2+l_1)^{-2}
\sum_{l_3}l_3^{-2}\sum_{l_0}\tilde{z}_{l_0}^{2}\biggr)^{1/2}=
\\
c(1-{\varepsilon })\frac{\pi^2}3\|u\| \| v\| \|w\| \|\varphi \| \|
z\|_{\dot{H}^{-\varepsilon }},
\end{gather*}
where $c(1-{\varepsilon })$ is as in~(\ref{c2s})
 and where we set
\begin{equation*}
\begin{matrix}
k_{1} & + & k_{2} & \phantom{k_3} & \phantom{+} &
\phantom{+} & \phantom{k_4}
& = & l_{1}\\
\phantom{k_1} & \phantom{+} & \phantom{k_2} & \phantom{+}
& k_{3} & %
\phantom{k_3} & \phantom{k_4} & = & l_{2} \\
\phantom{k_1} & \phantom{+} & {k_{2}} & {+} & k_{3} & + &
{k_{4}} & = & l_{3}
\\
{k_{1}} & {+} & {k_{2}} & {+} & k_{3} & + & {k_{4}} & = & l_{0},% \\
\end{matrix}%
\qquad \text{where}\qquad l_{1},l_{2},l_{3},l_{0}\in
\mathbb{Z}_{0}.
\end{equation*}%
Hence,
\begin{equation}  \label{estB41}
|(B_{4}^{1}(u,v,w,\varphi ),z)|\leq c(1-\varepsilon)
2^{-\varepsilon }{\pi ^{2}}\Vert u\Vert \Vert v\Vert \Vert
w\Vert \Vert \varphi \Vert \Vert z\Vert
_{\dot{H}^{-\varepsilon }}\,.
\end{equation}

It remains to estimate $B_{4}^{2}$. Since
$|k_{3}+k_{4}|\leq |k_{1}+k_{3}+k_{4}|+|k_{1}|$, we have
\begin{gather*}
\aligned|(B_{4}^{2}(u,v,w,\varphi ),z)|\leq
\!\!\!\sum_{k_{1},k_{2},k_{3},k_{4}}^{\text{nonres}}\!\!\!\frac{%
|k_{3}+k_{4}||u_{k_{1}}||v_{k_{2}}||w_{k_{3}}||\varphi _{k_{4}}|\,\tilde{z}%
_{k_{1}+k_{2}+k_{3}+k_{4}}\,|k_{1}+k_{2}+k_{3}+k_{4}|^{\varepsilon }}{%
|k_{1}||k_{1}+k_{2}||k_{1}+k_{3}+k_{4}||k_{2}+k_{3}+k_{4}|}\leq
\newline
\\
\sum_{k_{1},k_{2},k_{3},k_{4}}^{\text{nonres}}\frac{%
|u_{k_{1}}||v_{k_{2}}||w_{k_{3}}||\varphi _{k_{4}}|\,\tilde{z}%
_{k_{1}+k_{2}+k_{3}+k_{4}}|k_{1}+k_{2}+k_{3}+k_{4}|^{\varepsilon }}{%
|k_{1}||k_{1}+k_{2}||k_{2}+k_{3}+k_{4}|}+\newline
\\
\sum_{k_{1},k_{2},k_{3},k_{4}}^{\text{nonres}}\frac{%
|u_{k_{1}}||v_{k_{2}}||w_{k_{3}}||\varphi _{k_{4}}|\,\tilde{z}%
_{k_{1}+k_{2}+k_{3}+k_{4}}|k_{1}+k_{2}+k_{3}+k_{4}|^{\varepsilon }}{%
|k_{1}+k_{2}||k_{1}+k_{3}+k_{4}||k_{2}+k_{3}+k_{4}|}\,.\endaligned
\end{gather*}%
As shown above, the second sum is bounded by the right-hand side in~(\ref%
{estB41}). We split the first sum using the inequality $%
|k_{1}+k_{2}+k_{3}+k_{4}|^{\varepsilon }\leq
|k_{1}|^{\varepsilon }+|k_{2}+k_{3}+k_{4}|^{\varepsilon }$.
Then the first sum is bounded by
\begin{equation*}
\sum_{k_{1},k_{2},k_{3},k_{4}}^{\text{nonres}}\frac{%
|u_{k_{1}}||v_{k_{2}}||w_{k_{3}}||\varphi _{k_{4}}|\,\tilde{z}%
_{k_{1}+k_{2}+k_{3}+k_{4}}}{|k_{1}|^{1-\varepsilon
}|k_{1}+k_{2}||k_{2}+k_{3}+k_{4}|}+\sum_{k_{1},k_{2},k_{3},k_{4}}^{\text{%
nonres}}\frac{|u_{k_{1}}||v_{k_{2}}||w_{k_{3}}||\varphi _{k_{4}}|\,\tilde{z}%
_{k_{1}+k_{2}+k_{3}+k_{4}}}{|k_{1}||k_{1}+k_{2}||k_{2}+k_{3}+k_{4}|^{1-%
\varepsilon }}\,.
\end{equation*}%
Both terms are treated in the same way and satisfy the same
bound. Let us consider the second one. Setting
$k_{1}+k_{2}+k_{3}+k_{4}=l_{0}$, so that
$k_{4}=l_0-k_{1}-k_{2}-k_{3}$ and
$k_{2}+k_{3}+k_{4}=l_{0}-k_{1}$ we have
\begin{gather*}
\sum_{k_{1},k_{2},k_{3},k_{4}}^{\text{nonres}}\frac{%
|u_{k_{1}}||v_{k_{2}}||w_{k_{3}}||\varphi _{k_{4}}|\,\tilde{z}%
_{k_{1}+k_{2}+k_{3}+k_{4}}}{|k_{1}||k_{1}+k_{2}||k_{2}+k_{3}+k_{4}|^{1-%
\varepsilon }}=\sum_{k_{1},k_{2},k_{3},l_0}^{\text{nonres}}
\frac{|u_{k_{1}}||v_{k_{2}}||w_{k_{3}}||\varphi _{l_0-k_{1}-k_{2}-k_{3}}|\,\tilde{z}%
_{l}}{|k_{1}||k_{1}+k_{2}||l_{0}-k_{1}|^{1-\varepsilon }}=
\newline
\\
\sum_{k_{1}}\sum_{k_{2}}\frac{|u_{k_{1}}||v_{k_{2}}|}{|k_{1}||k_{1}+k_{2}|}%
\sum_{l_0\neq k_{1}}|l_{0}-k_{1}|^{-1+\varepsilon }\,\tilde{z}%
_{l_{0}}\sum_{k_{3}}|w_{k_{3}}||\varphi
_{l_{0}-k_{1}-k_{2}-k_{3}}|\leq
\newline
\\
\sum_{k_{1},k_2}%\sum_{k_{2}}
\frac{|u_{k_{1}}||v_{k_{2}}|}{|k_{1}||k_{1}+k_{2}|}%
c(1-{\varepsilon })\Vert z\Vert _{\dot{H}^{-\varepsilon
}}\Vert w\Vert \Vert
\varphi \Vert \leq \\
\biggl(\sum_{k_{1}}\sum_{k_{2}}\frac{1}{%
|k_{1}|^{2}|k_{1}+k_{2}|^{2}}\biggr)^{1/2}c(1-{\varepsilon })
\Vert u\Vert \Vert v\Vert \Vert w\Vert \Vert \varphi \Vert
\Vert z\Vert _{\dot{H}^{-\varepsilon }}=
\\ c(1-{\varepsilon })
\frac{\pi ^{2}}{3}\Vert u\Vert \Vert v\Vert \Vert w\Vert
\Vert \varphi \Vert \Vert z\Vert _{\dot{H}^{-\varepsilon}}\,.
\end{gather*}%
Hence,
$$
|(B_{4}^{2}(u,v,w,\varphi ),z)|\leq c(1-{\varepsilon })
\pi^{2}(2^{-\varepsilon }+2/3)\Vert u\Vert \Vert v\Vert \Vert
w\Vert \Vert \varphi \Vert \Vert z\Vert
_{\dot{H}^{-\varepsilon }},
$$
and $B_{4}$ satisfies the estimate
\begin{equation}
\Vert B_{4}(u,v,w,\varphi )\Vert _{\dot{H}^{s+\varepsilon
}}\leq
c_{4}(s,\varepsilon )\Vert u\Vert _{\dot{H}^{s}}\Vert v\Vert _{\dot{H}%
^{s}}\Vert w\Vert _{\dot{H}^{s}}\Vert \varphi \Vert
_{\dot{H}^{s}},\ \ \varepsilon <\frac{1}{2}  \label{B4hom}
\end{equation}%
with $s=0$ and $c_{4}(0,\varepsilon )=c(1-{\varepsilon })
\pi^{2}(3\cdot 2^{-\varepsilon -1}+2/3)$. Finally, the general
case $s\geq 0$ is treated as in Lemma~\ref{L:B3} by using
the inequality
\begin{equation*}
\frac{|k_{1}+k_{2}+k_{3}+k_{4}|}{|k_{1}||k_{2}||k_{3}||k_{4}|}\leq
4,
\end{equation*}%
which proves~(\ref{B4hom}) with $c_{4}(s,\varepsilon
)=4^{s}c_{4}(0,\varepsilon )$.
\end{proof}

\begin{lemma}
\label{L:B41n} Let $-3/2<s<1/2$. Then the nonlinear
operator $B_{40}^{1}$ given by $(\ref{B41n})$ satisfies the
estimate
\begin{equation}\label{B41nest}
\aligned \Vert B_{40}^{1}(u,u,u,v)\Vert _{\dot{H}^{s}}+
\Vert B_{40}^{1}(v,u,u,u)\Vert _{\dot{H}^{s}}&+
\\
\Vert B_{40}^{1}(u,v,u,u)\Vert _{\dot{H}^{s}}+& \Vert
B_{40}^{1}(u,u,v,u)\Vert _{\dot{H}^{s}}\leq
c_8(s)\|u\|_{\dot{H}^{0}}^{3}\|v\|_{\dot{H}^{s}}.
\endaligned
\end{equation}
If $1/2\le s$ and $\theta_0>s-1/2$, then
\begin{equation}\label{B41nestpos}
\aligned \Vert B_{40}^{1}(u,u,u,v)\Vert _{\dot{H}^{s}}+
\Vert B_{40}^{1}(v,u,u,u)\Vert _{\dot{H}^{s}}&+
\\
\Vert B_{40}^{1}(u,v,u,u)\Vert _{\dot{H}^{s}}+& \Vert
B_{40}^{1}(u,u,v,u)\Vert _{\dot{H}^{s}}\leq
c_8(s, \theta_0)\|u\|_{\dot{H}^{\theta_0}}^{3}
\|v\|_{\dot{H}^{s}}.
\endaligned
\end{equation}
\end{lemma}

\begin{proof}
We  observe that in the proof below we do not
take advantage of the operators $\Pi_{-n}$
in~(\ref{B41n}) and replace them with identity operators.

We first consider a more difficult case of negative $s$:
$-3/2<s\le0$.
We set $p=-s$. Using duality, we estimate
\begin{gather*}
|( B_{40}^{1}(v,u,u,u),z)|\leq
\\
\sum_{k_{1},k_{2},k_{3},k_{4}}^{\text{nonres}}
\frac{|k_{1}|^{p}|\tilde{v}_{k_{1}}||u_{k_{2}}||u_{k_{3}}|
|u_{k_{4}}| \tilde{z}_{k_{1}+k_{2}+k_{3}+k_{4}}}
{|k_{1}+k_{2}||k_{1}+k_{3}+k_{4}|| k_{2}+k_{3}+k_{4}|
|k_{1}+k_{2}+k_{3}+k_{4}|^{p}}\le
\\
\|\tilde{z}\|_{\dot{H}^{0}} \| u\|_{\dot{H}^{0}}^{2}
\|v\|_{\dot{H}^{s}}\times
\\
\left( \sum_{k_{1},k_{2},k_{3},k_{4}}^{\text{nonres}}%
\frac{|k_{1}|^{2p}|u_{k_{3}}|^{2}}
{|k_{1}+k_{2}|^2|k_{1}+k_{3}+k_{4}|^{2}|k_{2}+k_{3}+k_{4}|^2
|k_{1}+k_{2}+k_{3}+k_{4}|^{2p}}\right)^{1/2}=
\\
\|\tilde{z}\|_{\dot{H}^{0}}\|u\|_{\dot{H}^{0}}^{2}
\|v\|_{\dot{H}^{s}}\left(\sum_{k_1,k_2,k_3,l_0}^{\text{nonres}}
\frac{|k_{1}|^{2p}|u_{k_{3}}|^{2}}
{|k_{1}+k_{2}|^{2}|l_{0}-k_{1}|^{2}|l_{0}-k_{2}|^{2}
|l_{0}|^{2p}}\right) ^{1/2}=
\\
\|\tilde{z}\|_{\dot{H}^{0}}\|u\|_{\dot{H}^{0}}^{3}
\|v\|_{\dot{H}^{s}}\left(\sum_{k_1,k_2,l_0}^{\text{nonres}}
K_{41}\right)^{1/2}\,,
\end{gather*}%
where $\tilde{z}_{k}=|z_{k}||k|^{-s}$,
$\tilde{v}_{k}=|v_{k}||k|^{s}$, $l_0=k_1+k_2+k_3+k_4$ and
where
\begin{equation*}
K_{41}(k_1,k_2,l_0)=\frac{|k_{1}|^{2p}}
{|k_{1}+k_{2}|^{2}|l_{0}-k_{1}|^{2}|
l_{0}-k_{2}|^{2}|l_{0}|^{2p}} \,.
\end{equation*}%
Since any three of the four linear functionals $l_{j}$:
$l_0=k_1+k_2+k_3+k_4$, $l_1=k_1+k_2$, $l_2=l_0-k_1$,
$l_3=l_0-k_2$ are linearly independent over the
3-dimensional space of vectors $( k_{1},k_{2},l_{0})$,
similarly to Lemma~\ref{L:K3} we have
\begin{equation*}
\sum_{k_{1},k_{2},l_0}^{\text{nonres}}K_{41}\leq
C\sum_{i\neq j}\sum_{k_{1},k_{2},l_0}^{\text{nonres}}
\frac{( l_{j}^{3}l_{i})^{p/2}}
{|l_{1}|^{2}|l_{2}|^{2}|l_{3}|^{2}|l_{0}|^{2p}}\,.
\end{equation*}%
All the terms are estimated similarly, we take two typical
examples:
\begin{gather*}
\sum_{k_{1},k_{2},l_0}^{\text{nonres}} \frac{(
l_{0}^{3}l_{1}) ^{p/2}}{|l_{1}|^{2}|l_{2}|^{2}
|l_{3}|^{2}|l_{0}|^{2p}}=\sum_{l_{3}}^{\text{nonres}}
\sum_{l_{2}}^{\text{nonres}}\frac{1}{|l_{2}|^{2}
|l_{3}|^{2}}\sum_{l_{0}}^{\text{nonres}}
\frac{1}{|l_{1}|^{2-p/2}|l_{0}|^{p/2}}\leq
\\
\sum_{l_{3}}^{\text{nonres}}\sum_{l_{2}}^{\text{nonres}}
\frac{1}{|l_{2}|^{2}|l_{3}|^{2}}\left(
\sum_{l_{0}}^{\text{nonres}} \frac{1}{|l_{1}|^{2}}\right)
^{1-p/4}\left( \sum_{l_{1}}^{\text{nonres}}\frac{1}
{|l_{0}|^{2}}\right) ^{p/4}\,,
\end{gather*}%
which is a finite constant for $p\le 4$.  Similarly,
\begin{gather*}
\sum_{k_{1},k_{2},l_0}^{\text{nonres}}\frac{\left( l_{3}^{3}l_{1}\right) ^{p/2}%
}{\left\vert l_{1}\right\vert ^{2}\left\vert
l_{2}\right\vert ^{2}\left\vert
l_{3}\right\vert ^{2}\left\vert l_{0}\right\vert ^{2p}}%
= \sum_{l_{1}}^{\text{nonres}}\sum_{l_{2}}^{\text{nonres}}
\frac{1}{\left\vert l_{1}\right\vert
^{2-p/2}\left\vert l_{2}\right\vert ^{2}}\sum_{l_{3}}^{\text{nonres}}\frac{1%
}{\left\vert l_{3}\right\vert ^{2-3p/2}\left\vert l_{0}
\right\vert ^{2p}}\,.
\end{gather*}%
First, we take $p<4/3$. Then $2-3p/2>0$ and
\begin{equation*}
\sum_{l_{3}}^{\text{nonres}}\frac{1}
{|l_{3}|^{2-3p/2}|l_{0}|^{2p}}\leq \left(
\sum_{l_{3}}^{\text{nonres}}\frac{1}
{|l_{3}|^{(2-3p/2)/(1-q)}}\right)^{1-q}
\left(\sum_{l_{3}}^{\text{nonres}}\frac{1}
{|l_{0}|^{2p/q}}\right)^{q}
\end{equation*}%
with $2p/q>1$, $0\leq q\leq 1$, $2-3p/2>1-q$.
 Since $p<4/3$\ such a $q$ exists and the series converges.

If $p\geq 4/3$ and $p<3/2$, then we express $l_3$ as the
unique linear combination of $l_0$, $l_1$, $l_2$ (namely,
$l_3=2l_0-l_1-l_2$) and obtain
\begin{equation*}
\sum_{k_{1},k_{2},l_0}^{\text{nonres}}\frac {(
l_{3}^{3}l_{1}) ^{p/2}}
{|l_{1}|^{2}|l_{2}|^{2}|l_{3}|^{2}|l_{0}|^{2p}}
=\sum_{l_0,l_1,l_2}^{\text{nonres}}\frac {|l_{3}|^{3p/2-2}}
{|l_{1}|^{2-p/2}|l_{2}|^{2}|l_{0}|^{2p}}\leq
C\sum_{l_0,l_1,l_2}^{\text{nonres}}
\frac{|l_{1}|^{3p/2-2}+|l_{2}|^{3p/2-2}+|l_{0}|^{3p/2-2}}
{|l_1|^{2-p/2}|l_{2}|^{2}|l_{0}|^{2p}}.
\end{equation*}%
The series converges if $2-p/2+2-3p/2>1$, $-3p/2+4>1$,
$2+p/2>1$, that is, if $3>2p$. The three remaining terms
in~(\ref{B41nest}) are treated in the same way. The proof
of~(\ref{B41nest}) for the case of negative $s$ is complete.

For the case of positive $s$ we observe that
Lemma~\ref{L:B4} gives the estimate
\begin{equation*}
\|B_{40}^1(u,v,w,\varphi )\|_{\dot{H}^{s}}\leq
c_{4}(0,s )\|u\|_{\dot{H}^{0}}\|v\|_{\dot{H}^{0}}
\|w\|_{\dot{H}^{0}}\|\varphi \|_{\dot{H}^{0}},
\quad 0\le s<1/2,
\end{equation*}
which implies~(\ref{B41nest}) for the remaining
interval $0\le s<1/2$.

For $s\ge 1/2$ we again use Lemma~\ref{L:B4}
to see that
\begin{equation*}
\|B_{40}^1(u,v,w,\varphi )\|_{\dot{H}^{s}}\leq
c_{4}(\theta_0,s-\theta_0 )
\|u\|_{\dot{H}^{\theta_0}}\|v\|_{\dot{H}^{\theta_0}}
\|w\|_{\dot{H}^{\theta_0}}\|\varphi \|_{\dot{H}^{\theta_0}},
\quad s-\theta_0<1/2,
\end{equation*}
which gives~(\ref{B41nestpos}).
\end{proof}

\begin{lemma}\label{L:B42n}
 Let $-3/2<s<1/2$. Then the nonlinear
operator $B_{40}^{2}$ given by $(\ref{B42n})$ satisfies the
estimate
\begin{equation}\label{B42nest}
\aligned \Vert B_{40}^{2}(u,u,u,v)\Vert _{\dot{H}^{s}}+
\Vert B_{40}^{2}(v,u,u,u)\Vert _{\dot{H}^{s}}&+
\\
\Vert B_{40}^{2}(u,v,u,u)\Vert _{\dot{H}^{s}}+& \Vert
B_{40}^{2}(u,u,v,u)\Vert _{\dot{H}^{s}}\leq
c_8(s)\|u\|_{\dot{H}^{0}}^{3}\|v\|_{\dot{H}^{s}}.
\endaligned
\end{equation}
If $1/2\le s$ and $\theta_0>s-1/2$, then
\begin{equation}\label{B42nestpos}
\aligned \Vert B_{40}^{2}(u,u,u,v)\Vert _{\dot{H}^{s}}+
\Vert B_{40}^{2}(v,u,u,u)\Vert _{\dot{H}^{s}}&+
\\
\Vert B_{40}^{2}(u,v,u,u)\Vert _{\dot{H}^{s}}+& \Vert
B_{40}^{2}(u,u,v,u)\Vert _{\dot{H}^{s}}\leq
c_8(s, \theta_0)\|u\|_{\dot{H}^{\theta_0}}^{3}
\|v\|_{\dot{H}^{s}}.
\endaligned
\end{equation}
\end{lemma}

\begin{proof}
Similar to the previous lemma.
\end{proof}

\begin{lemma}
\label{L:Ares} Let $s\geq 0$. Then the nonlinear operator
$A_{\mathrm{res}}$ defined in $(\ref{Ares})$ maps
$\dot{H}^{s}$ into $\dot{H}^{s+1}$ and satisfies the
estimate
\begin{equation} \label{estAres}
\Vert A_{\mathrm{res}}(v)\Vert _{\dot{H}^{s+1}}\leq
c_{5}(s)\Vert v\Vert _{\dot{H}^{s}}.
\end{equation}
\end{lemma}

\begin{proof}
Clearly, (\ref{Ares}) implies (\ref{estAres})
with $c_5(s)\le\|v\|_{\dot{H}^0}^2$.
If energy is conserved, then
$c_{5}(s)\leq \Vert v^{0}\Vert_{\dot{H}^0}^{2}$.
\end{proof}

\begin{lemma}
\label{L:R3} Let $s>1/2$. Then the trilinear operator $R_{3}$ defined in $(%
\ref{R3})$ maps $(\dot{H}^{s})^{3}$ into $\dot{H}^{s}$ and
satisfies the estimate
\begin{equation}
\Vert R_{3}(u,v,w)\Vert _{\dot{H}^{s}}\leq c_{6}(s)\Vert u\Vert _{\dot{H}%
^{s}}\Vert v\Vert _{\dot{H}^{s}}\Vert w\Vert
_{\dot{H}^{s}}.  \label{estR3}
\end{equation}
\end{lemma}

\begin{proof}
By (\ref{R3})
\begin{equation}\label{R31}
|R_{3}(u,v,w)_{k}|\le \sum_{k_{1}+k_{2}+k_{3}=k}
\frac{|u_{k_{1}}v_{k_{2}}w_{k_{3}}|}{|k_{1}|}.
\end{equation}
As before, the time dependent exponentials in the
definition of $R_{3}$ do
not play a role. Setting $\tilde{u}_{k}=|u_{k}||k|^{s}$, $\tilde{v}%
_{k}=|v_{k}||k|^{s}$, $\tilde{w}_{k}=|w_{k}||k|^{s}$, $\tilde{z}%
_{k}=|z_{k}||k|^{-s}$ and arguing by duality we have
(taking no advantage of the factor $|k_1|^{-1}$ below)
\begin{gather*}
\aligned|(R_{3}(u,v,w),z)|\leq \sum_{k_{1},k_{2},k_{3}}\frac{\tilde{u}%
_{k_{1}}\tilde{v}_{k_{2}}\tilde{w}_{k_{3}}\,\tilde{z}_{k_{1}+k_{2}+k_{3}}{%
|k_{1}+k_{2}+k_{3}|^{s}}}{|k_{1}|^{s+1}|k_{2}|^{s}|k_{3}|^{s}}\leq
\newline
\\
3^{s-1}\sum_{k_{1},k_{2},k_{3}}\left\vert {\tilde{u}_{k_{1}}\tilde{v}_{k_{2}}%
\tilde{w}_{k_{3}}\,\tilde{z}_{k_{1}+k_{2}+k_{3}}}\right\vert \biggl(\frac{1}{%
|k_{1}|^{s}|k_{2}|^{s}}+\frac{1}{|k_{1}|^{s}|k_{3}|^{s}}+\frac{1}{%
|k_{2}|^{s}|k_{3}|^{s}}\biggr)\leq \newline
\\
3^{s}c(s)^{2}\Vert u\Vert _{\dot{H}^{s}}\Vert v\Vert
_{\dot{H}^{s}}\Vert w\Vert _{\dot{H}^{s}}\Vert z\Vert
_{\dot{H}^{-s}}\,.\endaligned
\end{gather*}%
where $c(s)=(\sum_{j\in \mathbb{Z}_{0}}j^{-(2s+2)})^{1/2}<\infty $ for $%
s>1/2 $. This proves~(\ref{estR3}) with
$c_{6}(s)=3^{s}c(s)^{2}$.

We finally observe that $R_{3}$ is a Lipschitz map from $\dot{H}^{s}$ to $%
\dot{H}^{s}$:
\begin{equation}\label{LipR3}
\|R_{3}(u_{1},u_{1},u_{1})-R_{3}(u_{2},u_{2},u_{2})\|
_{\dot{H}^{s}}\leq
10c_{6}(s)(\|u_{1}\|_{\dot{H}^{s}}^{2}+\|u_{2}\|_{\dot{H}^{s}}^{2})
\|u_{1}-u_{2}\|_{\dot{H}^{s}}.
\end{equation}
\end{proof}

\begin{lemma}
\label{L:minR3} Let $S>1/2$ and $\beta <1/2$.
 Then the trilinear operator
$R_{3}$ defined in $(\ref{R3})$ satisfies the estimate
\begin{equation}
\Vert R_{3}(u,v,w)\Vert _{\dot{H}^{-S}}\leq c_{6}^{\prime
}(S,\beta )\Vert
u\Vert _{\dot{H}^{-\beta }}\Vert v\Vert _{\dot{H}^{0}}\Vert w\Vert _{\dot{H}%
^{0}}.  \label{estR31}
\end{equation}
\end{lemma}

\begin{proof}
\ Arguing by duality and using (\ref{R31}) and (\ref{zk})
we obtain
\begin{gather*}
|(R_{3}(u,v,w),z)|\leq \sum_{k_{1},k_{2},k_{3}}
\frac{{u}_{k_{1}}{v}_{k_{2}}{w}_{k_{3}}
{z}_{k_{1}+k_{2}+k_{3}}}{|k_{1}|} \leq
\newline
\\
\left( \sum_{k_{1},k_{2},k_{3}}\frac{\left\vert
u_{k_{1}}\right\vert \left\vert v_{k_{2}}\right\vert
^{2}\,\left\vert z_{k_{1}+k_{2}+k_{3}}\right\vert
}{|k_{1}|}\right) ^{1/2}\left(
\sum_{k_{1},k_{2},k_{3}}\frac{\left\vert
u_{k_{1}}\right\vert \left\vert
w_{k_{3}}\right\vert ^{2}\,\left\vert z_{k_{1}+k_{2}+k_{3}}\right\vert }{%
|k_{1}|}\right) ^{1/2}=\newline
\\
\left( \sum_{k_{1}}\frac{\left\vert u_{k_{1}}\right\vert }{|k_{1}|}%
\sum_{k_{2}}\left\vert v_{k_{2}}\right\vert
^{2}\sum_{k}\left\vert z_{k}\right\vert \right)
^{1/2}\left( \sum_{k_{1}}\frac{\left\vert
u_{k_{1}}\right\vert }{|k_{1}|}\sum_{k_{2}}\left\vert
w_{k_{2}}\right\vert ^{2}\sum_{k}\left\vert
z_{k}\right\vert \right) ^{1/2}\le
\\
c_{2}^{\prime }(S)\|v\|_{\dot{H}^{0}}\|w\|_{\dot{H}^{0}}
\|z\|_{\dot{H}^{S}} \left(\sum_{k_{1}}
\frac{|u_{k_{1}}|^{2}}{|k_{1}|^{2\beta }}\right)^{1/2}
\left(\sum_{k_{1}}\frac{1}{|k_{1}|^{2-2\beta
}}\right)^{1/2}\,.
\end{gather*}
This proves~(\ref{estR31}).
\end{proof}

\subsubsection*{Acknowledgments}

A.A.I. would like to thank the warm hospitality of
the Mathematics Department at the University of
California, Irvine, and the Weizmann Institute of
Science where this work was done. The work of
A.V.B. was supported by AFOSR grant
FA9550-04-1-0359.  The work of A.A.I. was supported
in part by the Russian Foundation for Fundamental
Research, grants~no.~09-01-00288 and
no.~08-01-00784, and by the RAS Programme no.1. The
work of E.S.T. was supported in part by the NSF,
grant no.~DMS-0708832, the ISF grant no. 120/6, and
the BSF grant no. 2004271.


\begin{thebibliography}{99}

\bibitem{Awad} \textrm{Awad, Y.} \textrm{%
Complex Burgers equation with rotation - a paradigm for rotation
prevents singularity.} \textit{Masters Thesis}. Department of
Mathematics, Alquds University, Palestinian Territories, (2007).

\bibitem{BMN97} \textrm{Babin A., Mahalov A., Nicolaenko B.} \textrm{%
Regularity and integrability of 3D Euler and Navier--Stokes
equations for rotating fluids. } \textit{Asymptot. Anal.}
\textbf{15}:2, 103--150 (1997).

\bibitem{BMN99} \textrm{Babin A., Mahalov A., Nicolaenko B.}
 \textrm{Global
regularity of 3D rotating Navier--Stokes equations for
resonant domains. } \textit{\ Indiana Univ. Math. J.}
\textbf{48}:3, 1133--1176 (1999).

\bibitem{BMN99a} \textrm{Babin A., Mahalov A., Nicolaenko B.}
\textrm{On the regularity of three-dimensional rotating
Euler-Boussinesq equations.} \textit{Mathematical Models
and Methods in Applied Sciences.} \textbf{9}:7, 1089--1121
(1999).

\bibitem{Bourgain97}
\textrm{Bourgain J.} \textrm{Periodic Korteweg de Vries equation
with measures as initial data.} \textit{ Selecta Math. (N.S.)}
\textbf{3}:2, 115--159 (1997).


\bibitem{Bourgain93} \textrm{Bourgain J.} \textrm{Fourier
transform restriction phenomena for certain lattice subsets and
applications to nonlinear evolution equations. II. The
KdV-equation.} \textit {Geom. Funct. Anal.} \textbf{ 3}:3, 209--262
(1993).

\bibitem{Colliander03} \textrm{Colliander J.,  Keel M.,  Staffilani G., Takaoka H., Tao,
T.} \textrm{ Sharp global well-posedness for KdV and modified
KdV on $\mathbb{%
R}$ and $\mathbb{T}$.} \textit{ J. Amer. Math. Soc.
(electronic)} \textbf{16}:3, 705--749 (2003).


\bibitem{Colliander04}
\textrm{Colliander J., Keel M., Staffilani G., Takaoka H.,
 Tao T.}
 \textrm{ Multilinear estimates for periodic KdV equations, and
applications.}
\textit{ J. Funct. Anal.} \textbf {211}:1,
173--218  (2004).

\bibitem{Colliander05}
\textrm{Colliander J., Keel M., Staffilani G., Takaoka H.,
 Tao T.}
 \textrm{ Symplectic nonsqueezing of the KdV flow.}
\textit{Acta Math.}
\textbf{195},
 197-252  (2005).


\bibitem{CF88}\textrm{Constantin~P. and Foias~C.}
\textit{ Navier-Stokes Equations}.
\textrm{ The University of Chicago Press}, 1988.




\bibitem{Embid}
\textrm{Embid P.F., Majda A.J.} \textrm{Averaging over fast
gravity waves for geophysical flows with arbitrary
potential
 vorticity.}
 \textit{Comm. Partial Diff. Eqs.} \textbf{21}, 619-658
 (1996).

\bibitem{G1} \textrm {Gallagher I.}
 \textrm{ Un r\'esultat de stabilit\'e pour les
\'equations des fluides tournants.} \textit{ C.R. Acad.
Sci. Paris, S\'erie I} \textbf{324}:2, 183-186 (1997).

\bibitem{GMS1}
\textrm {Germain P., Masmoudi N., Shatah J.}
\textrm{Global solutions for the gravity water
waves equation in dimension 3},
\textit{C. R. Acad. Sci. Paris, Ser. I}  \textbf{347},
897-902 (2009).

\bibitem{K-P-V}
\textrm {Kenig C., Ponce G., Vega L.} \textrm{A bilinear estimate
with applications to the KdV equation.} \textit{J. Amer. Math. Soc.}
 \textbf{9}:2, 573--603  (1996).

\bibitem{Kappeler-Topalov} \textrm {Kappeler T. and Topalov P.}
\textrm{Global wellposedness of KdV in $H\sp {-1}(\Bbb T,\Bbb R)$.}
\textit{Duke Math. J.} \textbf{ 135}:2 , 327--360 (2006).


\bibitem{Liu-Tadmor}  \textrm {Liu H., Tadmor E.}
\textrm{Rotation prevents finite-time breakdown.} \textit{Physica D}
\textbf{188}, 262–-276  (2004).


\bibitem{MoiseZiane}
\textrm {Moise I.,  Ziane M.} \textrm{ Renormalization
group method. Applications to partial differential
equations.} \textit{ J. Dynam. Differential Equations}
\textbf{13}:2, 275--321 (2001).

\bibitem{gren}
\textrm{Grenier E.} \textrm{ Rotating fluids and inertial
waves.} \textit{Proc. Acad Sci. Paris, S\'erie I}
\textbf{321}, 711-714 (1995).

\bibitem{Shatah}
\textrm{Shatah J.} \textrm{ Normal forms and quadratic nonlinear
Klein-Gordon equations} \textit{Comm. Pure Appl. Math.}
\textbf{38}, 685-696  (1985).




\bibitem{sch}
\textrm{Schochet S.}
 \textrm{ Fast singular limits of hyperbolic
PDE's.} \textit { J. Diff. Eq.}
 \textbf{114}, 476-512  (1994).

\bibitem{Schochet04}
\textrm{Schochet S.} \textrm{Long-time averaging for some
conservative PDEs having quadratic nonlinearities.}
\textit{ Discrete Contin. Dyn. Syst.} \textbf{11}:1,
221-233  (2004).

\bibitem{TNS}\textrm{Temam~R.}
\textit{Navier--Stokes Equations. Theory and Numerical
Analysis,} \textrm{Amsterdam}, \textrm{North-Holland} 1984.

\bibitem{Wu1} \textrm{Wu S.} \textrm{Well-posedness in
Sobolev spaces of the full water wave problem in 2-D.}
\textit{Invent. Math.} \textbf{130}:1, 39-72 (1997).

\bibitem{Wu2} \textrm{Wu S.} \textrm{ Well-posedness
in Sobolev spaces of the full water wave problem in 3-D.}
\textit{ J. Amer. Math. Soc.}  \textbf{12}:2, 445-495
(1999).

\bibitem{Wu3} \textrm{Wu S.} \textrm{Almost global
well-posedness of the 2-D full water wave problem.}
\textit{Invent. Math.} (to appear).

\bibitem{Ziane}
\textrm {Ziane M.} \textrm{ On a certain renormalization
group method.} \textit { J. Math. Phys.} \textbf{41},
3290--3299  (2000).

\end{thebibliography}
\end{document}